\documentclass[reqno]{amsart}
\usepackage{amsfonts}
\usepackage{mathrsfs}
\usepackage{color}
\usepackage{cite}
\usepackage{amscd}
\usepackage{latexsym}
\usepackage{amsfonts}
\usepackage{graphicx}
\usepackage{CJK,indentfirst,amsmath,amsfonts,amssymb,amsthm,cite,cases, subeqnarray,setspace,hyperref}
\allowdisplaybreaks[4]

\makeatletter
\@namedef{subjclassname@2020}{\textup{2020} Mathematics Subject Classification}
\makeatother

\begin{document}
\title[The bidirectional NLS approximation for the Euler-Poisson system]
{The bidirectional NLS approximation for the one-dimensional Euler-Poisson system}
\author{Huimin Liu, Yurui Lu and Xueke Pu}

\address{Huimin Liu \newline
Faculty of Applied Mathematics, Shanxi University of Finance and Economics, Taiyuan 030006, P.R.China}
\email{hmliucqu@163.com}

\address{Yurui Lu, Xueke Pu \newline
School of Mathematics and Information Sciences, Guangzhou University, Guangzhou 510006, P.R. China}
\email{luyuruigzhu@163.com; xuekepu@gzhu.edu.cn}

\thanks{The first author H. Liu was supported by the Basic Research Project of Shanxi Province of China under 202403021211156 and 202303021211140, the third author X. Pu was supported by NSFC (Grant 12471220, 12431003) and the Natural Science Foundation of Guangdong Province of China under 2019A1515012000.}

\begin{abstract}
The nonlinear Schr\"{o}dinger (NLS) equation is known as a universal equation describing the evolution of the envelopes of slowly modulated spatially and temporarily oscillating wave packet in various dispersive systems. In this paper, we prove that under a certain multiple scale transformation, solutions to the Euler-Poisson system can be approximated by the sums of two counter-propagating waves solving the NLS equations. It extends the earlier results [Liu and Pu, Comm. Math. Phys., 371(2), (2019)357-398], which justify the unidirectional NLS approximation to the Euler-Poisson system for the ion-coustic wave. We demonstrate that the solutions could be convergent to two counter-propagating wave packets, where each wave packet involves independently as a solution of the NLS equation. We rigorously prove the validity of the NLS approximation for the one-dimensional Euler-Poisson system by obtaining uniform error estimates in Sobolev spaces. The NLS dynamics can be observed at a physically relevant timespan of order $\mathcal{O}(\epsilon^{-2})$. As far as we know, this result is the first construction and valid proof of the bidirectional NLS approximation.
\end{abstract}

\maketitle \numberwithin{equation}{section}
\newtheorem{proposition}{Proposition}[section]
\newtheorem{theorem}{Theorem}[section]
\newtheorem{lemma}[theorem]{Lemma}
\newtheorem{remark}[theorem]{Remark}
\newtheorem{hypothesis}[theorem]{Hypothesis}
\newtheorem{definition}{Definition}[section]
\newtheorem{corollary}{Corollary}[section]
\newtheorem{assumption}{Assumption}[section]

\tableofcontents

\section{\textbf{Introduction}}
\setcounter{section}{1}\setcounter{equation}{0}

The two-fluid Euler-Poisson system is a classical  model to describe dynamics of two separate compressible fluids of ions and electrons interacting with their self-consistent electrostatic field. Let $n_{i}(t,x)$ and $v_{i}(t,x)$ (resp. $n_{e}(t,x)$ and $v_{e}(t,x)$) be the number density and the velocity of ions (resp. electrons) at time $t$ and position $x$, respectively. Let $\phi$ be the electric potential that satisfy the Poisson equation. In one dimension, the two-fluid Euler-Poisson system consists of the hydrodynamic equations coupled to the Poisson equation
\begin{equation}
\begin{cases}\label{twoEP}
\partial_{t}n_{i}+\partial_{x}(n_{i}v_{i})=0,\\
n_{i}M_{i}(\partial_{t}v_{i}+v_{i}\partial_{x}v_{i})+T_{i}\partial_{x}P_{i}
+en_{i}\partial_{x}\phi=0,\\
\partial_{t}n_{e}+\partial_{x}(n_{e}v_{e})=0,\\
n_{e}M_{e}(\partial_{t}v_{e}+v_{e}\partial_{x}v_{e})+T_{e}\partial_{x}P_{e}
-en_{e}\partial_{x}\phi=0,\\
-\partial_{x}^{2}\phi=4\pi e(n_{i}-n_{e}),
\end{cases}
\end{equation}
where $e>0$ is the electron charge, $M_{i}$ and $T_{i}$ (resp. $M_{e}$ and $T_{e}$) are the mass and effective temperature of an ion (resp. electron). Here the pressure laws $P_{i}=P_{i}(n_{i})$ and $P_{e}=P_{e}(n_{e})$ are assumed to be smooth and strictly increasing. 

It is well-known that the ratio of the electron mass and the ion mass $\frac{M_{e}}{M_{i}}\ll1$. That means the much heavier ions can be treated approximately as motionless with a constant density $n_{i}(t,x)\equiv n_{0}$. Therefore a reduced model of \eqref{twoEP} to describe the electron fluid dynamics in an ion background (Langmuir waves) in 1D is as follows
\begin{subequations}\label{eEP}
\begin{numcases}{(eEP)}
\partial_{t}n_{e}+\partial_{x}(n_{e}v_{e})=0,\\
n_{e}M_{e}(\partial_{t}v_{e}+v_{e}\partial_{x}v_{e})+T_{e}\partial_{x}P_{e}
-en_{e}\partial_{x}\phi=0,\\
-\partial_{x}^{2}\phi=4\pi e(n_{0}-n_{e}).
\end{numcases}
\end{subequations}
On the other hand, under the zero electron mass assumption $\frac{M_{e}}{M_{i}}\rightarrow0$, the electron density satisfies the famous Boltzmann relation
\begin{equation*}
\begin{split}
n_{e}=n_{0}e^{e\phi/T_{e}}
\end{split}
\end{equation*}
from the momentum equation for electrons in \eqref{twoEP}. The one dimensional Euler-Poisson system can then be written as
\begin{subequations}\label{iEP}
\begin{numcases}{(iEP)}
\partial_{t}n_{i}+\partial_{x}(n_{i}v_{i})=0,\\
n_{i}M_{i}(\partial_{t}v_{i}+v_{i}\partial_{x}v_{i})+T_{i}\partial_{x}P_{i}
+en_{i}\partial_{x}\phi=0,\\
-\partial_{x}^{2}\phi=4\pi e(n_{i}-n_{0}e^{e\phi/T_{e}}),
\end{numcases}
\end{subequations}
describing the motion of an ion-acoustic plasma. We often call the ion-acoustic plasma hot for $T_{i}>0$ and cold for $T_{i}=0$, respectively. We take the pressure $P_i(n_i)=n_i$ in this paper.

In recent years, there are a large number of studies on the global existence of solutions and approximations to the Euler-Poisson system. The Euler-Poisson system for ions as well as electrons are both important PDE models arising in plasma physics, share some basic difficulties as other important fluid models, and are far from being well understood. For the electron fluid in Euler-Poisson system, 3D shock waves do develop for large perturbations of the constant state equilibrium of $n_{e}\equiv n_{0}, v_{e}\equiv0$ in \eqref{eEP} because of the quasi-linear hyperbolic nature (see \cite{GTZ99}). On the other hand, a Klein-Gordon effect in the linearization system around $n_{e}\equiv n_{0}, v_{e}\equiv0$ enhances linear decay rate, so the 3D global irrotational solutions with small velocity were obtained in the seminal paper  of Guo \cite{Guo98}. In addition, the global smooth irrotational small solutions are constructed in 2D independently for \eqref{eEP} in \cite{IP13,Jang12,JLZ14,LW14}. Furthermore, Guo, Han and Zhang \cite{GHZ17} finally completely settled this problem and proved that no shocks form for the 1D Euler-Poisson system for electrons \eqref{eEP}. For the ion fluid in Euler-Poisson system, the nonlinear Poisson equation for the electric potential in \eqref{iEP} presents a new mathematical challenge to prevent shock formation near $n_{i}\equiv n_{0}, v_{i}\equiv0$. The 3D global smooth irrotational solutions with small amplitude was constructed in \cite{GP11}, while it remains an outstanding problem if shocks can develop for \eqref{iEP} in 2D and 1D of \eqref{iEP}.

For the long wave approximation, the rigorous justification of the Korteweg-de Vries (KdV) equation \cite{A} from the one dimensional Euler-Poisson system was proved by Guo and Pu \cite{GP14}, for both the hot and cold ion plasmas. Since then, the long wavelength limit for the Euler-Poisson system and related models are extensively studied. In particular,  modified KdV (mKdV) equation in 1D, Kadomtsev-Petviashvili-II (KP-II) equation in 2D and Zakharov-Kuznetsov (ZK) equation in 3D from the Euler-Poisson system for ions can be found in \cite{PX21,P,LLS}. The long wavelength limit of the Vlasov-Poisson system was studied by Han-Kwan \cite{HK13} and was recently generalized to two-directional approximation in a fixed torus by Yang and Zhao \cite{YZ24}. For a reduced quantum two-fluid Euler-Poisson system, the long wavelength limit was proved rigorously, deriving the quantum KdV equation in 1D and the KP limit in 2D for both cold and hot plasma by the authors \cite{LP,LP23}. There are also some results for the bidirectional long wavelength limit. Schneider \cite{S98} showed the KdV approximation results for a Boussinesq equation,  that the long wavelength solutions split into two pieces, one a right moving wave train and one a left moving wave train. This bidirectional KdV approximation result \cite{S98} was further extended by giving a significantly better high order approximation than the KdV equation alone. See \cite{WW02}. The bidirectional long wave limit of the Euler-Korteweg system was studied in \cite{BGC18}, which provides a new method to show the long-wave approximation of the solutions to KdV equation or Burgers equation according to the different regimes of the parameters. For the Euler-Poisson system for ions, Liu and Yang \cite{LY20} justified rigorously the bidirectional long wave approximation, demonstrating that the solutions could be convergent to two wave packets in opposite directions, where each wave packet involves independently as a solution of the Burgers or KdV equation.

Since the oscillatory solution of the KdV equation is nothing but a solution of the nonlinear Schr\"odinger (NLS) equation in the small wave number region with frequency induced from the dispersive relation, considering the NLS approximation for the Euler-Poisson system is a more interesting problem. The authors \cite{LP19} proved rigorously the NLS approximation to the Euler-Poisson system for the first time for the ion-acoustic plasma by applying normal-form transformations \cite{Shatah85} and constructing an appropriate energy function. In the paper \cite{LP19}, for the ion Euler-Poisson system \eqref{iEP}, we assume
\begin{equation}
\begin{split}\label{fom}
\begin{pmatrix} n_{i}-1 \\ v_{i} \end{pmatrix}=\epsilon\Psi_{NLS}+\mathcal{O}(\epsilon^{2}),
\end{split}
\end{equation}
with
\begin{equation}
\begin{split}\label{LS}
\epsilon\Psi_{NLS}=\epsilon A(X+c_{g}T,\epsilon T)E^{1}\varrho(k_{0})+c.c.,
\end{split}
\end{equation}
then the NLS equation is derived for the complex amplitude $A$,
\begin{equation}
\begin{split}\label{A}
\partial_{\theta}A=i\mu_{1}\partial_{X_{+}}^{2}A+i\mu_{2}A|A|^{2},
\end{split}
\end{equation}
where $T=\epsilon t\in\mathbb{R}, \ X=\epsilon x\in\mathbb{R}$, $E^{1}=e^{i(k_{0}x+\omega_{0}t)}$.  Here, in \eqref{A}, the coefficients $\mu_{j}=\mu_{j}(k_{0})\in \mathbb{R}$ with $j\in\{1,2\}$, $\varrho(k_{0})=(1,-q(k_{0}))^{T}$, $\theta=\epsilon T=\epsilon^{2}t$ is the slow time scale and $X_{+}=X+c_{g}T=\epsilon(x+c_{g}t)$ is the slow spatial scale. The approximation \eqref{fom} is also called the modulation approximation. In this approximation, $0<\epsilon\ll1$ is a small perturbation parameter, $\omega_{0}>0$ is the basic temporal wave number associated to the basic spatial wave number $k_{0}>0$ of the underlying temporally and spatially oscillating wave train $e^{i(k_{0}x+\omega_{0}t)}$, $c_{g}$ is the group velocity and `c.c.' denotes the complex conjugate. The NLS \eqref{A} is derived in order to describe the slow modulations in time and in space of the wave train $e^{i(k_{0}x+\omega_{0}t)}$, with time and space scales of the modulations being $\mathcal{O}(1/\epsilon^{2})$ and $\mathcal{O}(1/\epsilon)$, respectively. For the ion Euler-Poisson system \eqref{iEP}, the basic spatial wave number $k=k_{0}$ and the basic temporal wave number $\omega=\omega_{0}$ are related via the following linear dispersion relation
\begin{equation}
\begin{split}\label{equation3}
\omega(k)=k\sqrt{\frac{2+k^{2}}{1+k^{2}}}=k\widehat{q}(k), \ \ \ \ \ \ \ \widehat{q}(k)=\sqrt{\frac{2+k^{2}}{1+k^{2}}}.
\end{split}
\end{equation}
From the dispersion relation \eqref{equation3}, the group velocity $c_{g}=\omega'(k_{0})$ of the wave packet can be computed.

Indeed, in the past decades, various NLS approximation results have been proved. In the absence of quadratic terms a simple application of Gronwall's inequality yields the validity of modulation equations for extended systems with cubic nonlinearities \cite{KS}. However, if the dispersive system contains quasilinear quadratic terms, the appearance of the $\mathcal{O}(\epsilon)$ terms and the loss of derivatives make it difficult to justify the NLS dynamics on the correct timescale $\mathcal{O}(\epsilon^{-2})$, because of the loss of double derivatives for the transformed system and the occurrence of resonances. The idea is to apply the normal-form transformations and other techniques. By a near identity change of variables, the terms of order $\mathcal{O}(\epsilon)$ are eliminated if a non-resonance condition holds \cite{K}. For dispersive systems with quasilinear quadratic terms that lose only half a derivative and whose dispersive relation does not satisfy the non-resonance condition, the NLS approximation can be proved by using normal-form transformations and the Cauchy-Kowalevskaya theorem, such as the water wave systems in \cite{S,D,D21}. For dispersive systems with the quasilinear quadratic terms that lose one derivative but without resonances, the uniform error estimate can be obtained by using the normal-form transformation to construct a modified energy, such as the quasilinear quadratic Klein-Gordon equation \cite{D17}. While in the Euler-Poisson system \eqref{iEP} for ion-acoustic plasma, resonances occur and the quasilinear quadratic terms lose one derivative, the authors used the corresponding normal-form transformations twice to deal with the resonances and then constructed a modified energy functional to obtain uniform error estimates. In this way, we proved in \cite{LP19} rigorously that the solutions $(n_{i}-1, v_{i})$ of the ion Euler-Poisson system \eqref{iEP} can be approximated by the oscillating wave packet $\epsilon\Psi_{NLS}$ whose complex amplitude satisfies the NLS equation \eqref{A}. Moreover, we justified the NLS approximation for the cold ion Euler-Poisson system, i.e., $T_i=0$ in \eqref{iEP}, where the resonances occur, the quasilinear quadratic terms lose one derivative, and the linearized operators do not have any regularity \cite{LBP24}. For some special cases, when the quasilinear quadratic terms lose more than one derivative but all linear operators have the same regularity properties, the NLS approximation was also obtained by using a modified energy based on the normal-form transformations \cite{D21,H22}.

All the above results are unidirectional NLS approximations. Note that the ansatz \eqref{fom} with \eqref{LS} leads to waves moving to the left. By replacing the vector $\varrho(k_{0})=(1,-q(k_{0}))^{T}$ with $\sigma(k_{0})=(1,q(k_{0}))^{T}$ as well as $\omega_{0}$ and $c_{g}$ with $-\omega_{0}$ and $-c_{g}$, we will obtain waves moving to the right. I.e., if we assume
\begin{equation}
\begin{split}\label{fomm}
\begin{pmatrix} n_{i}-1 \\ v_{i} \end{pmatrix}=\epsilon\Phi_{NLS}+\mathcal{O}(\epsilon^{2}),
\end{split}
\end{equation}
with
\begin{equation}
\begin{split}\label{PLS}
\epsilon\Phi_{NLS}=\epsilon B(X-c_{g}T,\epsilon T)F^{1}\sigma(k_{0})+c.c.,
\end{split}
\end{equation}
then the NLS equation can be derived for the complex amplitude $B$,
\begin{equation}
\begin{split}\label{B}
\partial_{\theta}B=i\nu_{1}\partial_{X_{-}}^{2}B+i\nu_{2}B|B|^{2},
\end{split}
\end{equation}
where $X_{-}=X-c_{g}T=\epsilon(x-c_{g}t)$, $F^{1}=e^{i(k_{0}x-\omega_{0}t)}$, $\nu_{j}=\nu_{j}(k_{0})\in \mathbb{R}$ with $j\in\{1, \ 2\}$ and $\sigma(k_{0})=(1, \ q(k_{0}))^{T}$. Similarly, by using the same method as done in \cite{LP19} we can prove the solutions $(n_{i}-1, v_{i})$ of \eqref{iEP} can be approximated by the right going oscillating wave packet $\epsilon\Phi_{NLS}$ whose complex amplitude satisfies the NLS equation \eqref{B}, without any new essential difficulties.

In this paper, we are concerned about whether or not the solutions $(n_{i}-1, v_{i})$ of \eqref{iEP} can be approximated by $\epsilon\Psi_{NLS}+\epsilon\Phi_{NLS}$, approximations that admit the left going and right going oscillating wave packets at the same time.  However, as far as the authors know, all the existing NLS type approximation results are unidirectional. 
This paper contains the first construction and valid proof of the bidirectional NLS approximation results with error estimates in Sobolev spaces on the correct time scale $\mathcal{O}(\epsilon^{-2})$, taking the 1D Euler-Poisson system \eqref{iEP} for ions as an example. That is to say, we prove first formally and then rigorously that the solutions of the 1D Euler-Poisson system for ions can be approximated by two counter-propagating oscillating wave packets, each of which involves independently as a solution of the NLS equation. See Theorem \eqref{Thm1} below. We believe that similar results can be obtained for the afore mentioned systems, such as the water wave system and the quasilinear quadratic Klein-Gordon equation. The mathematical treatments will be much more complex than the unidirectional NLS approximation in \cite{LP19}.

To obtain the bidirectional NLS approximation of the ion Euler-Poisson system \eqref{iEP}, 
we let
\begin{equation}
\begin{split}\label{foma}
\begin{pmatrix} n_{i}-1 \\ v_{i} \end{pmatrix}=\epsilon\Psi_{NLS}+\epsilon\Phi_{NLS}+\mathcal{O}(\epsilon^{2}),
\end{split}
\end{equation}
with
\begin{equation}
\begin{cases}\label{BLS}
&\epsilon\Psi_{NLS}=\epsilon A(X+c_{g}T),\epsilon T)E^{1}\varrho(k_{0})+c.c.,\\
&\epsilon\Phi_{NLS}=\epsilon B(X-c_{g}T),\epsilon T)F^{1}\sigma(k_{0})+c.c.
\end{cases}
\end{equation}
We will show that the following two independent NLS equations can be derived
\begin{equation}
\begin{cases}\label{fAB}
\partial_{\theta}A=i\mu_{1}\partial_{X_{+}}^{2}A+i\mu_{2}A|A|^{2},\\
\partial_{\theta}B=i\nu_{1}\partial_{X_{-}}^{2}B+i\nu_{2}B|B|^{2},
\end{cases}
\end{equation}
and prove the following main theorem in this paper.
\begin{theorem}\label{Thm1}
Fix $s_{N}\geq6$. For all $k_{0}>0$, $C_{1}>0$ and $T_{0}>0$, there exist $C_{2}>0$ and $\epsilon_{0}>0$ such that for all solutions $A$ and $B$ of the NLS equations \eqref{fAB} with
\begin{equation*}
\begin{split}
\sup_{T\in[0,T_{0}]}\|A(\cdot,T),B(\cdot,T)\|_{H^{s_{N}}(\mathbb{R},\mathbb{C})}\leq C_{1},
\end{split}
\end{equation*}
the following holds. For all $\epsilon\in(0,\epsilon_{0})$, there are solutions
\begin{equation*}
\begin{split}
\begin{pmatrix} n_{i}-1 \\ v_{i} \end{pmatrix}\in \big(C([0,T_{0}/\epsilon^{2}],H^{s_{N}}(\mathbb{R},\mathbb{R}))\big)^{2}
\end{split}
\end{equation*}
of the ion Euler-Poisson system \eqref{iEP} that satisfy
\begin{equation*}
\begin{split}
\sup_{t\in[0,T_{0}/\epsilon^{2}]}\bigg\|\begin{pmatrix} n_{i}-1 \\ v_{i} \end{pmatrix}-\epsilon\Psi_{NLS}(\cdot,t)-\epsilon\Phi_{NLS}(\cdot,t)\bigg\|_{\big(H^{s_{N}}(\mathbb{R},\mathbb{R})\big)^{2}}
\leq C_{2}\epsilon^{3/2}.
\end{split}
\end{equation*}
\end{theorem}

Let us give some remarks about the proof of Theorem \ref{Thm1}. Firstly, we need to construct formal bidirectional approximations with two counter propagating oscillating wave packets, each of which evolves as an independent NLS equation.  Therefore, we need to cancel the interactions of the two NLS equations. In particular, let $E^{j_{1}}$ and $F^{j_{2}}$ be as above, we then hope that the coefficients of the order $\epsilon^{l}E^{j_{1}}F^{j_{2}}$ are zero, where the positive integer $l$ and the integers $j_{1}$ and $\ j_{2}$ satisfy $|j_{1}|+|j_{2}|\leq l$. However, the left and right moving oscillating wave trains will certainly interact at some higher orders such as in $\epsilon^{2}E^{1}F^{1}$, $\epsilon^{3}E^{1}$ and $\epsilon^{3}F^{1}$. This phenomenon will not happen in the unidirectional NLS approximation in \cite{LP19}. To cancel these interactions, we add two types of correctors to our formal approximation solutions. The first type of correctors $\epsilon^{2}\widetilde{\Upsilon}_{\pm1\pm1}$ in \eqref{appr} are used to balance the interactions in the order $\epsilon^{2}E^{\pm1}F^{\pm1}$. The amplitude functions of this type of correctors are determined by some simple algebraic relations \eqref{f11}. To obtain two independent NLS equations without any interaction terms, we add the second type of correctors $\epsilon^{2}\widetilde{\Upsilon}_{\pm1}$ in \eqref{appr} to balance the interaction terms appeared in the coefficients of the order $\epsilon^{3}E^{1}$ and $\epsilon^{3}F^{1}$. The amplitude functions of this type of correctors satisfy two inhomogeneous transport equations \eqref{Modify1}-\eqref{Modify2}. In addition, even in the case of a purely right (or left) going oscillating wave train, solutions to the Euler-Poisson system \eqref{iEP} are not exactly described by solutions of the NLS equation. Therefore we still need to add the third type of correctors $\epsilon^{2}\widetilde{\Psi}_{n}$ and $\epsilon^{2}\widetilde{\Phi}_{n}$ with $n=0, \pm2$ in \eqref{appr} to balance the coefficients of the order $\epsilon^{2}E^{n}$ and $\epsilon^{2}F^{n}$ with $n=0, \pm2$. The amplitude functions of the third type of correctors are determined by algebraic relations or inhomogeneous, linear PDE's \eqref{A2} and \eqref{A0}. So far we have constructed a proper formal approximation solution $\epsilon\widetilde{\Theta}$ in the equations \eqref{app} with \eqref{appr}-\eqref{lowapp} below.

Secondly, we need to make the residual $Res(\epsilon\widetilde{\Theta})$ defined in \eqref{residual} small enough to obtain an uniform estimate of the error $R$ defined in \eqref{before}, the difference between the real solution and the approximation solution for the Euler-Poisson system \eqref{iEP}, whose evolution satisfies the equation \eqref{error}. Here, the residual $Res(\epsilon\widetilde{\Theta})$ in \eqref{residual} represents the terms that do not cancel after inserting the approximation $\epsilon\widetilde{\Theta}$ into the Euler-Poisson system \eqref{iEP}. For this, we further modify the approximation $\epsilon\widetilde{\Theta}$ in two steps. First, we add the higher order terms such as $\epsilon^{2}\widetilde{\Psi}_{p}, \ \epsilon^{2}\widetilde{\Phi}_{p}$ and $ \epsilon^{3}\widetilde{\Upsilon}_{r}$ in \eqref{extt} into the approximation to obtain the higher order approximation $\epsilon\widetilde{\Theta}^{ext}$, which is much more complex than the unidirectional NLS approximation in \cite{LP19,D,D17}. Then we apply some cut-off function to the approximation $\epsilon\widetilde{\Theta}^{ext}$ to keep the support of the approximation being restricted to small neighborhoods of integer multiples of the basic wave number in Fourier space. By these two modifications, the final approximation $\epsilon\Theta$ becomes analytic and the residual $Res(\epsilon\Theta)$ is small enough. These modifications are treated in Sect. 4 in detail.

Thirdly, to justify the bidirectional NLS approximation, we need to obtain an uniform estimate for the error $R$ on sufficiently long timescale of order $\mathcal{O}(\epsilon^{-2})$. However, the quasilinear quadratic terms of \eqref{iEP} will lead to the appearance of $\mathcal{O}(\epsilon)$ terms in the evolution equation \eqref{error} of error $R$, which can perturb the linear evolution in such a way that the solutions begin to grow on time scale $\mathcal{O}(\epsilon^{-1})$ and hence we would lose control over the size of $R$ on the desired time scale $\mathcal{O}(\epsilon^{-2})$. Thus, we use normal-form transformation to eliminate these terms of $\mathcal{O}(\epsilon)$. In the process of elimination, two kinds of difficulties arise for the Euler-Poisson system \eqref{iEP}. One difficulty is the occurrence of the trivial and nontrivial resonances in eliminating the low frequency terms of $\mathcal{O}(\epsilon)$ due to the continuous linear dispersive relation \eqref{equation3}. The other difficulty is the loss of one derivative of the quasilinear quadratic terms, resulting into the loss of double derivatives of the transformed system in eliminating the high frequency terms of $\mathcal{O}(\epsilon)$. Therefore, it is difficult to obtain the uniform estimate for the error $R$. In this paper, we deal with the resonances by taking a weight function to modify the error $R$, and then make twice normal-from transformations directly to eliminate those low frequency terms of $\mathcal{O}(\epsilon)$. To solve the difficulty of the loss of derivatives, we use the normal-from transformations corresponding to those high frequency terms of $\mathcal{O}(\epsilon)$ to construct a modified energy function. The uniform estimate of the error $R$ is obtained finally by using the important properties of the normal-from transformation. Such a process is much more complex and different from the unidirectional NLS approximation due to the counter propagating oscillating wave packets. We believe that the method developed in this paper can be applied to obtain bidirectional NLS approximation for some other related dispersive systems.

\begin{remark}
Compared with the solution $(n_{i}-1, v_{i})$ and the approximation $\epsilon\Psi_{NLS}+\epsilon\Phi_{NLS}$, which are both of order $\mathcal{O}(\epsilon)$ in $L^{\infty}$, the error of order $\mathcal{O}(\epsilon^{3/2})$ is small enough such that the dynamics of the NLS equation can be observed in the ion Euler-Poisson system \eqref{iEP}. Secondly, the Fourier transform of $\epsilon\Psi_{NLS}+\epsilon\Phi_{NLS}$ is sufficiently concentrated around the wave numbers $\pm k_{0}$, hence by using a modified approximation that has compact support in Fourier space but differs only slightly from $\epsilon\Psi_{NLS}+\epsilon\Phi_{NLS}$, the smoothness of the error can be made equal to the assumed smoothness of the amplitude. Finally, in the following proof of Theorem \ref{Thm1}, we always assume that $s_N$ is an integer to simplify the proof, although it can be generalized to all real numbers $s_N\geq6$.
\end{remark}

\begin{remark}
In the ion Euler-Poisson system \eqref{iEP}, the ion pressure is given by $P_i(n_i)=n_i$. In fact, the result in this paper can be generalized to the general $\gamma$-law of the ion pressure, i.e. $P_i(n_i)=n_i^{\gamma}$ for $\gamma\geq1$.
\end{remark}

\textbf{Notations}. Let  $\mathbb{K}=\mathbb{R}$ or $\mathbb{K}=\mathbb{C}$. We denote the Fourier transform of a function $u\in L^{2}(\mathbb{R},\mathbb{K})$ by
\begin{equation*}
\begin{split}
\widehat{u}(k)=\frac{1}{2\pi}\int_{\mathbb{R}}u(x)e^{-ikx}dx.
\end{split}
\end{equation*}
Let $H^{s}(\mathbb{R},\mathbb{K})$ be the space of functions mapping from $\mathbb{R}$ into $\mathbb{K}$ with the standard norm
\begin{equation*}
\begin{split}
\|u\|_{H^{s}(\mathbb{R},\mathbb{K})} =\big(\int_{\mathbb{R}}|\widehat{u}(k)|^{2} (1+|k|^{2})^{s}dk\big)^{1/2}.
\end{split}
\end{equation*}
We also write $L^{2}$ and $H^{s}$ instead of $L^{2}(\mathbb{R},\mathbb{R})$ and $H^{s}(\mathbb{R},\mathbb{R})$. Moreover, we define the space $L^{p}(m)(\mathbb{R},\mathbb{K})$ of $u$ such that $\sigma^{m}u\in L^{p}(\mathbb{R},\mathbb{K})$, where $\sigma(x)=(1+x^{2})^{1/2}$. Furthermore, we write $A\lesssim B$ if $A\leq CB$ for a constant $C>0$, and $A=\mathcal{O}(B)$ if $A\lesssim B$ and $B\lesssim A$.

This paper is organized as follows. We first give the basic ideas in Section 2. Then we diagonalize the Euler-Poisson system and then give the formal derivation of the bidirectional NLS approximation in Section 3.  In Section 4, we further modify the approximation by using a cutoff function and adding high order terms, and give estimates for the residual. In Section 5, we derive the evolutionary error equation and apply the projection operators to extract the low as well as high frequency terms that need to be eliminated. We then restate the main Theorem \ref{Thm1} in terms of the error for the diagonalized system in Theorem \ref{Thm2}. These two theorems are equivalent. The last three sections are dedicated to proof of Theorem \ref{Thm2}. In Section 6, we apply normal-form transformations twice to eliminate the low frequency terms and obtain the transformed error equations. In Section 7, we analyze the properties of the normal-form transformation for the high frequency terms. In Section 8, we construct an energy functional for the transformed error by using the normal-form transformation for the high frequency terms to deal with the difficulties caused by the loss of derivatives, and then we prove the uniform energy estimates for the error.

\section{\textbf{The basic ideas}}
In order to simplify the presentation, we consider an abstract evolutionary problem
\begin{equation}
\begin{split}\label{abstractt}
\partial_{t}U=\Lambda U+Q(U,U)+N(U),
\end{split}
\end{equation}
with $U=U(x,t)\in\mathbb{R}^{2}$, $x\in\Bbb R$ and $t\in\mathbb{R}^+$. Here, $\Lambda$ is a linear operator whose symbol is a diagonal matrix of the form
\begin{equation*}
\begin{split}
\widehat{\Lambda}(k)=diag\{i\omega(k),-i\omega(k)\},
\end{split}
\end{equation*}
where $k\in\mathbb{R}$ represents the wave number and $\omega$ is a piecewise smooth real-valued odd function of $k$. In this paper below, $\omega(k)$ is the dispersive relation of \eqref{iEP}, defined in \eqref{equation3}. In \eqref{abstractt}, the linear operator $\Lambda$ generates a uniformly bounded semigroup, $Q$ is a symmetric bilinear operator and $N$ is a high order operator whose order is at least cubic. The ion Euler-Poisson system \eqref{iEP} can be reduced to \eqref{abstractt}. See \eqref{abstract} for details. 

To prove Theorem \ref{Thm1}, we need to do two things. First, the formal bidirectional NLS approximation \eqref{foma}-\eqref{fAB} for the Euler-Poisson system \eqref{iEP} should be constructed. Secondly, the error between the real solution and the approximation solution needs to be estimated uniformly in $\epsilon$ in some Sobolev space for a physically relevant timescale of order $\mathcal{O}(\epsilon^{-2})$. We sketch the main ideas in the following two subsections. 

\subsection{\textbf{Constructing an approximation solution}}
To obtain the above two independent NLS equations \eqref{fAB} formally, 
it is essential to make the following two types of corrections to such an approximation \eqref{foma}.
\begin{itemize}
  \item  Corrections to balance the interactions of  the left and right moving wave trains. In particular, we will add $\epsilon^{2}\Upsilon_{NLS}$ of \eqref{ULS} to balance the coefficients at order of $E^{j}$ and $F^{j}$ with $j=\pm1$ and $\epsilon^{2}\Upsilon_{h}$ of \eqref{ULS} to balance the coefficients at order of $E^{j_{1}}F^{j_{2}}$ with $j_{1},j_{2}=\pm1$.
  \item Corrections  $\epsilon^{2}(\Psi_{h}+\Phi_{h})$ of \eqref{ULS} to balance the coefficients at order of $E^{j}$ and $F^{j}$ with $j=0,\pm2$.
\end{itemize}
%
%

To incorporate these two types of corrections, we add additional functions $\epsilon^{2}\Psi_{h}$, $\epsilon^{2}\Phi_{h}$, $\epsilon^{2}\Upsilon_{NLS}$ and $\epsilon^{2}\Upsilon_{h}$ to leading terms $\epsilon\Psi_{NLS}$ and $\epsilon\Phi_{NLS}$. Therefore, we consider the bidirectional NLS approximation of the following form
\begin{equation}
\begin{split}\label{ULS}
 \epsilon\widetilde{\Theta}=\epsilon(\Psi_{NLS}+\Phi_{NLS})+\epsilon^{2}(\Psi_{h}+\Phi_{h})
 +\epsilon^{2}(\Upsilon_{NLS}+\Upsilon_{h}),
\end{split}
\end{equation}
with
\begin{equation}\label{formal}
\begin{cases}
\epsilon\Psi_{NLS}&=\epsilon A\big(X+c_{g}T,\epsilon T\big)E^{1}\begin{pmatrix}1\\0\end{pmatrix}+c.c,\\
\epsilon\Phi_{NLS}&=\epsilon B\big(X-c_{g}T,\epsilon T\big)F^{1}\begin{pmatrix}0\\1\end{pmatrix}+c.c,\\
\epsilon^{2}\Psi_{h}&=\epsilon^{2}\begin{pmatrix}A_{01}(X+c_{g}T,\epsilon T)\\ A_{02}(X+c_{g}T,\epsilon T)\end{pmatrix}
+\epsilon^{2}\begin{pmatrix}A_{21}(X+c_{g}T,\epsilon T)\\ A_{22}(X+c_{g}T,\epsilon T)\end{pmatrix}E^{2}+c.c,\\
\epsilon^{2}\Phi_{h}&=
\epsilon^{2}\begin{pmatrix}B_{01}(X-c_{g}T,\epsilon T)\\ B_{02}(X-c_{g}T,\epsilon T)\end{pmatrix}+\epsilon^{2}\begin{pmatrix}B_{21}(X-c_{g}T,\epsilon T)\\ B_{22}(X-c_{g}T,\epsilon T)\end{pmatrix}F^{2}+c.c,\\
\epsilon^{2}\Upsilon_{NLS}
&=\epsilon^{2}\begin{pmatrix}\mathcal{I}(X,T,\epsilon T)E^{1}\\ \mathcal{J}(X,T,\epsilon T)F^{1}\end{pmatrix}+c.c, \\ \epsilon^{2}\Upsilon_{h}&=\epsilon^{2}\begin{pmatrix}f_{111}(X,T,\epsilon T)\\ f_{112}(X,T,\epsilon T)\end{pmatrix}E^{1}F^{1}+
\epsilon^{2}\begin{pmatrix}f_{-111}(X,T,\epsilon T)\\ f_{-112}(X,T,\epsilon T)\end{pmatrix}E^{-1}F^{1}
+c.c,
\end{cases}
\end{equation}
where $X=\epsilon x, T=\epsilon t, E^{j}=e^{ij(k_{0}x+\omega_{0}t)}, F^{j}=e^{ij(k_{0}x-\omega_{0}t)}$, $\theta=\epsilon T=\epsilon^{2}t$ is the slow time scale and $X_{+}=X+c_{g}T$ and $X_{-}=X-c_{g}T$ are the slow spatial scale. In the above modulation approximation, $0<\epsilon\ll1$ is a small perturbation parameter, $\omega_{0}>0$ is the basic temporal wave number associated to the basic spatial wave number $k_{0}>0$ of the underlying temporally and spatially oscillating wave trains $E^{1}$ and $F^{1}$, and $c_{g}=\omega'(k_{0})$ is the group velocity and `c.c' denotes the complex conjugate. Here, $A, B, A_{2j}, B_{2j}, \mathcal{I}, \mathcal{J}, f_{mnj}, g_{mnj}$ with $j=1, 2$ and $m, n=\pm1$ are complex-valued functions, and $A_{0j}$ and $B_{0j}$ with $j=1,2$ are real-valued functions. NLS equations are derived in order to describe the slow modulations in time and in space of the wave train $E^{1}$ and $F^{1}$, with time and space scales being of $\mathcal{O}(\epsilon^{-2})$ and $\mathcal{O}(\epsilon^{-1})$, respectively. For the Euler-Poisson system \eqref{iEP}, the basic spatial and temporal wave numbers $k=k_{0}$ and $\omega=\omega_{0}$ are related via the linear dispersion relation \eqref{equation3}.

Inserting such formal ansatz $U=\epsilon\widetilde{\Theta}$ \eqref{ULS}-\eqref{formal} into the abstract evolutionary equation \eqref{abstractt}, we obtain that the functions from the higher order correctors $\Psi_{h}$, $\Phi_{h}$ and $\Upsilon_{h}$ satisfy
\begin{equation*}
\begin{split}
A_{0j}&=\gamma_{0j}(k_{0})|A|^{2}, \ \ \ \ \ \ \ \ \ \ \ \ \ B_{0j}=\gamma_{0j}(k_{0})|B|^{2},\\
A_{2j}&=\gamma_{2j}(k_{0})A^{2}, \ \ \ \ \ \ \ \ \ \ \ \  \ \ \ B_{2j}=\gamma_{2j}(k_{0})B^{2},\\
f_{mnj}&=\gamma_{mnj}(k_{0})A_{m}B_{n},
\end{split}
\end{equation*}
for $\gamma_{0j}(k_{0}), \gamma_{2j}(k_{0}), \gamma_{mn}(k_{0})\in \mathbb{R}$ with $j=1,2,m,n=\pm1$, and $A_{1}=A, A_{-1}=\overline{A}, B_{1}=B,B_{-1}=\overline{B}$.

The complex amplitudes $A$ and $B$ at leading order in $\epsilon$ satisfy the following independent NLS equations,
\begin{equation}
\begin{split}\label{AB}
\partial_{\theta}A&=i\omega''(k_{0})\partial_{X_{+}}^{2}A+i\nu_{1}(k_{0})A|A|^{2},\\
\partial_{\theta}B&=-i\omega''(k_{0})\partial_{X_{-}}^{2}B+i\nu_{2}(k_{0})B|B|^{2},
\end{split}
\end{equation}
where coefficients $\nu_{j}(k_{0}), \beta_{j}(k_{0})\in \mathbb{R}$ with $j=1, 2$. The correctors $\mathcal{I}$ and $\mathcal{J}$ satisfy the following transport equations,
\begin{equation}
\begin{split}\label{fg}
\partial_{T}\mathcal{I}-c_{g}\partial_{X}\mathcal{I}&=i\beta_{1}(k_{0})A|B|^{2},\\
\partial_{T}\mathcal{J}+c_{g}\partial_{X}\mathcal{J}&=i\beta_{2}(k_{0})B|A|^{2}.
\end{split}
\end{equation}
The NLS equations from \eqref{AB} are completely integrable Hamiltonian systems \cite{A} and the correctors $\mathcal{I}$ and $\mathcal{J}$ have some good estimates in \eqref{IJ1} below.

To prove uniform estimates for the error and justify the bidirectional NLS approximation, we define the residual
\begin{equation}\label{residual}
\begin{split}
Res(U)=-\partial_{t}U+\Lambda U+Q(U,U)+N(U).
\end{split}
\end{equation}
This residual $Res(U)$ contains all terms that do not cancel after inserting $U=\epsilon\widetilde{\Theta}$ into system \eqref{abstractt} and is small for $U=\epsilon\widetilde{\Theta}$. However, it is not easy to justify such bidirectional NLS approximation with such an approximation $U=\epsilon\widetilde{\Theta}$. Therefore, we consider the following two-step modifications to \eqref{ULS}.

\begin{itemize}
  \item  Extend $\epsilon\widetilde{\Theta}$ to a higher order form $\epsilon\widetilde{\Theta}^{ext}$ of \eqref{ext} by adding sufficiently higher order correction terms. This step makes the residual $Res(\epsilon\widetilde{\Theta})$ small of order $\mathcal{O}(\epsilon^{m})$ for sufficiently large positive integer $m$. See Section 4.1 for details.
  \item Modify the extended approximation $\epsilon\widetilde{\Theta}^{ext}$ by applying some cut-off function such that the support of the final approximation $\epsilon\Theta$ of \eqref{ab} is restricted to small neighborhoods of integer multiples of the basic wave number $k_{0}>0$ in Fourier space. This step changes the approximation slightly, but benefits us in controlling the resonances and makes the final approximation analytic. See Section 4.2 for details.
\end{itemize}
With these modifications, we can write the final approximation $\epsilon\Theta$ as
\begin{equation*}
\begin{split}
\epsilon\Theta=\epsilon\Psi_{1}+\epsilon\Psi_{-1}+\epsilon\Phi_{1}+\epsilon\Phi_{-1}
+\epsilon^{2}\Theta_{r}
=:\epsilon\Psi_{c}+\epsilon\Phi_{c}+\epsilon^{2}\Theta_{r},
\end{split}
\end{equation*}
with compact supports in Fourier sapce
\begin{equation*}
\begin{split}
&supp \ \widehat{\Psi}_{c}=\{k\mid|k\pm k_{0}|\leq\delta\},\\
&supp \ \widehat{\Phi}_{c}=\{k\mid|k\pm k_{0}|\leq\delta\},\\
&supp \ \widehat{\Theta}_{r}=\{k\mid|k\pm \ell k_{0}|\leq\delta, \ \ \ell=0, \pm1, \pm2, \pm3, \pm4\},
\end{split}
\end{equation*}
for some $\delta>0$ sufficiently small, but independent of $0<\epsilon\ll1$.

\subsection{\textbf{Uniform error estimates}}
To prove Theorem \ref{Thm1} we have to estimate the error
 \begin{equation}
\begin{split}\label{before}
\epsilon^{\beta}R=U-\epsilon\Theta
\end{split}
\end{equation}
to be of order $\mathcal{O}(\epsilon^{\beta})$ for some $\beta>1$ on a time scale $t\in[0,T_{0}/\epsilon^{2}]$. We also call $R$ the error and we have to show that it is of order $\mathcal{O}(1)$ for all $t\in[0,T_{0}/\epsilon^{2}]$. Inserting \eqref{before} into \eqref{abstractt}, we find
\begin{equation}
\begin{split}\label{error}
\partial_{t}R=&\Omega R+2\epsilon Q(\Psi_{c}+\Phi_{c},R)+2\epsilon^{2} Q(\Theta_{r},R)+\epsilon^{\beta} Q(R,R)+\epsilon^{2}\widetilde{N}(R)+\epsilon^{-\beta}Res(\epsilon\Theta).
\end{split}
\end{equation}
Here $\epsilon^{2}\widetilde{N}(R)$ contains all terms from $N(U)$ after inserting \eqref{before} into \eqref{abstractt} and is of $\mathcal{O}(\epsilon^{2})$ since $N(U)$ is at least cubic. For the error equation \eqref{error}, the linear operator $\Omega$ generates a uniformly bounded semigroup. Assume $\beta>2$, then the last four terms in \eqref{error} are $\mathcal{O}(\epsilon^{2})$ bounded over the relevant time interval $\mathcal{O}(\epsilon^{-2})$, thanks to Lemma \ref{L2}. However, the $\mathcal{O}(\epsilon)$ linear term $\epsilon Q(\Psi_{c}+\Phi_{c},R)$ can perturb the linear evolution largely and make the solutions begin to grow on time scales $\mathcal{O}(\epsilon^{-1})$. Therefore, it is difficult to desire a uniform bound of $R$ on the desired time scale $\mathcal{O}(\epsilon^{-2})$. To give a uniform bound of the error $R$ over the desired time intervals $\mathcal{O}(\epsilon^{-2})$, we apply 
the following normal-form transformation
\begin{equation}
\begin{split}\label{equation45}
\widetilde{R}_{j_{1}}:=R_{j_{1}}+\epsilon B^{+}_{j_{1}}(\Psi_{c},R)
+\epsilon B^{-}_{j_{1}}(\Phi_{c},R),
\end{split}
\end{equation}
with
\begin{equation}
\begin{split}\label{equation46}
\widehat{B}^{+}_{j_{1}}(\Psi_{c},R)=\sum_{j_{2\in\{\pm1\}}}\int \widehat{b}^{+}_{j_{1}j_{2}}(k,k-\ell,\ell)\widehat{\Psi}_{c}(k-\ell)\widehat{R}_{j_{2}}(\ell)d\ell,\\
\widehat{B}^{-}_{j_{1}}(\Phi_{c},R)=\sum_{j_{2\in\{\pm1\}}}\int \widehat{b}^{-}_{j_{1}j_{2}}(k,k-\ell,\ell)\widehat{\Phi}_{c}(k-\ell)\widehat{R}_{j_{2}}(\ell)d\ell,
\end{split}
\end{equation}
where $j_{1}\in\{\pm1\}$ and $\widehat{R}_{j_{1}}$ refers to the $j_{1}$ component of $\widehat{R}$. The kernel function $\widehat{b}^{\pm}_{j_{1},j_{2}}$ of the normal-form transformation is a quotient whose denominator is
\begin{equation}
\begin{split}\label{deno}
-j_{1}\omega(k)\mp\omega(k-\ell)+j_{2}\omega(\ell).
\end{split}
\end{equation}
From the linear dispersive relation \eqref{equation3}, we see that $\omega(k)$ satisfies $\omega(0)=0$ and $\omega(-k)=-\omega(k)$. Thus there exists a resonance at the wave number $k=0$ that makes the denominator \eqref{deno} of $\widehat{b}^{\pm}_{j_{1},j_{2}}$ equal to zero if $j_{2}=\mp1$. Fortunately, this resonance is trivial since the nonlinear term also vanishes at $k=0$ and hence the normal-form transformation \eqref{equation45} is well-defined for all $|k|\leq\delta$. However, there exists another nontrivial resonance at the wave number $k=\ell$ if $j_{1}=j_{2}$ such that \eqref{equation45} is not well-defined. Therefore, we cannot apply \eqref{equation45} directly. We will rather use an improved method (see \cite{S,S98}). The method is to make a suitable rescaling of the error function $R$ by introducing a weight function dependent on the wave number, and then to use a number of special normal-form transformations.

Precisely, define a weight function $\vartheta$ by its Fourier transform
\begin{equation}
\begin{split}\label{v}
\widehat{\vartheta}(k)=\Big\{\begin{matrix} 1\ \ \ \ \ \ \ \ \ \ \ \ \ \ \ \ \ \ \ \ \ \ \ \ \ \ \ \ \text{for} \ \ |k|>\delta, \\ \epsilon+(1-\epsilon)| k|/\delta \ \ \ \ \ \ \ \ \ \text{for} \ \ |k|\leq\delta,\end{matrix}
\end{split}
\end{equation}
where $\delta>0$ is sufficiently small, but independent of $0<\epsilon\ll1$. This makes $\widehat{\vartheta}(k)\widehat{R}(k)$ small at wave numbers close to zero to reflect the fact that the nonlinearity vanishes at $k=0$. Modify the solution $U$ of \eqref{before} by using the above weight function $\vartheta$
\begin{equation}
\begin{split}\label{equation48}
U=\epsilon\Psi+\epsilon^{\beta}\vartheta R.
\end{split}
\end{equation}
Here, $\vartheta R$ is defined by $\widehat{\vartheta R}=\widehat{\vartheta}\widehat{R}$, in a slight abuse of notation, avoiding using the convolution notation $\vartheta\ast R$. Inserting \eqref{equation48} into \eqref{abstractt}, we obtain
\begin{equation}
\begin{split}\label{equation49}
\partial_{t}R=& \Lambda R+2\epsilon \vartheta^{-1}Q(\Psi_{c}+\Phi_{c},\vartheta R)+2\epsilon^{2}\vartheta^{-1}Q(\Theta_{r},\vartheta R)+\epsilon^{\beta} \vartheta^{-1}Q(\vartheta R, \vartheta R)\\
&+\epsilon^{2}\vartheta^{-1}\widetilde{N}(\vartheta R)+\epsilon^{-\beta}\vartheta^{-1}Res(\epsilon\Theta).
\end{split}
\end{equation}
Note that the last four terms of \eqref{equation49} are still at least $\mathcal{O}(\epsilon^{2})$ although $\widehat{\vartheta}^{-1}$ is of order $\mathcal{O}(\epsilon^{-1})$ for $|k|\leq\delta$, thanks to Lemma \ref{L2} and the fact that the kernels of these terms are smaller than $|k|$ (see \eqref{ker}). Therefore, only the remaining linear term $2\epsilon \vartheta^{-1}Q(\Psi_{c}+\Phi_{c},\vartheta R)$ needs to be eliminated.

Since the system \eqref{abstractt} is quasilinear, the size of the nonlinear term depends on whether $k$ is close to zero or not in Fourier space. To analyze the behaviors in these two regions more clearly, we define two projection operators $P^{0}$ and $P^{1}$ by the Fourier multiplier
\begin{equation}
\begin{split}\label{equation58}
\widehat{P}^{0}(k)=\chi_{\mid k\mid\leq\delta}(k)\ \ \ and \ \ \ \widehat{P}^{1}(k)=\mathbf{1}-\widehat{P}^{0}(k),
\end{split}
\end{equation}
for $\delta>0$ sufficiently small (the same constant $\delta$ in the definition of $\vartheta$), but independent of $0<\epsilon\ll1$. We will write $R=R^{0}+R^{1}$ with $R^{j}=P^{j}R$, for $j=0,1$. Note that these superscripts are different from those subscripts that denote the components of $R$. Applying the projection operators $P^{0}$ and $P^{1}$ to \eqref{equation49}, we will obtain an equivalent system for $R^{0}$ and $R^{1}$,
\begin{align}\label{0R}
\partial_{t}R^{0}=\Lambda R^{0}+2\epsilon P^{0}\vartheta^{-1} Q(\Psi_{c}+\Phi_{c},\vartheta R^{1})+\mathcal{O}(\epsilon^{2}),
\end{align}
\begin{align}\label{1R}
\partial_{t}R^{1}=\Lambda R^{1}+2\epsilon P^{1}\vartheta^{-1} Q(\Psi_{c}+\Phi_{c},\vartheta_{0} R^{0})+2\epsilon P^{1}\vartheta^{-1} Q(\Psi_{c}+\Phi_{c},\vartheta R^{1})+\mathcal{O}(\epsilon^{2}).
\end{align}
where $\widehat{\vartheta}_{0}=\widehat{\vartheta}-\epsilon$. Here we have ignored the term $2\epsilon P^{0}\vartheta^{-1} Q(\Psi_{c}+\Phi_{c},\vartheta R^{0})$ in \eqref{0R} due to the property of $P^{0}$ and compact support of $\Psi_{c}+\Phi_{c}$ in Fourier apace. In fact, $(\widehat{\Psi}_{c}+\widehat{\Phi}_{c})(k-m)=0$ unless $|(k-m)\pm k_{0}|<\delta$ and $\widehat{R}^{0}(m)=0$ for $|m|>\delta$. We can also classify all terms of $\mathcal{O}(\epsilon)$ in \eqref{0R}-\eqref{1R} as two categories:
\begin{description}
  \item[low frequency terms]
  \begin{equation*}
  \begin{split}
  &2\epsilon P^{0}\vartheta^{-1} Q(\Psi_{c}+\Phi_{c},\vartheta R^{1})\ \ \ \ \ \ \ \text{from \eqref{0R}, and}\\
  &2\epsilon P^{1}\vartheta^{-1} Q(\Psi_{c}+\Phi_{c},\vartheta_{0} R^{0})\ \ \ \ \ \ \text{from \eqref{1R},}
  \end{split}
  \end{equation*}
  \item[high frequency terms]
  \begin{equation*}
  \begin{split}
  2\epsilon P^{1}\vartheta^{-1} Q(\Psi_{c}+\Phi_{c},\vartheta R^{1})\ \ \ \ \ \ \  \text{from \eqref{1R}.}\ \ \ \ \ \
  \end{split}
  \end{equation*}
\end{description}

In order to obtain uniform estimates for $(R^{0},R^{1})$ on a time scale $\mathcal{O}(\epsilon^{-2})$, we will make normal-form transformations to eliminate these $\mathcal{O}(\epsilon)$ terms from \eqref{0R}-\eqref{1R} as done in \eqref{equation45}. In the process of eliminating the above two low frequency terms, we find that both the trivial resonance at the wave number $k=0$ and the nontrivial resonances at the wave number $k=\pm k_{0}$ appear. The corresponding normal-form transformations are still well-defined since the nonlinear term in \eqref{abstractt} is also equal to zero at $k=0$ in Fourier space, and the nontrivial resonances at $k=\pm k_{0}$ can be balanced by $\widehat{\vartheta}_{0}$ thanks to the introduction of the weight function $\widehat{\vartheta}$. However, there is a remaining term of $\mathcal{O}(\epsilon)$ after the first normal-form transformation since $\widehat{\vartheta}^{-1}(k)=\epsilon^{-1}$ for $|k|<\delta$, which necessitates the use of the second normal-form transformation (See Section 6).

For the high frequency term $2\epsilon P^{1}\vartheta^{-1} Q(\Psi_{c}+\Phi_{c},\vartheta R^{1})$ in \eqref{1R}, there are no resonances in the corresponding normal-form transformation. However, because of the loss of one derivative in the high frequency term, the normal-form transformation loses one derivative, which makes the corresponding term of transformed equation lose even two derivatives. This makes it difficult to obtain the uniform estimate for the error. Therefore, we cannot use the normal-form transformation directly for the high frequency term in \eqref{1R}. We would rather construct an energy function via the normal-form transformation, to treat the high frequency term in \eqref{1R} and deal with the loss of derivatives at the same time (See Section 7 for details).


\section{\textbf{Formal Derivation of the NLS Approximation}}
In this section, we first diagonalize the linear part of the equation \eqref{iEP} around the equilibrium solution $(n_{i},v_{i})=(1,0)$ to obtain an abstract evolutionary system \eqref{abstract} in the form \eqref{abstractt}. We then formally obtain the two independent NLS equations \eqref{NLS1}-\eqref{NLS2} by constructing a modified modulation approximation \eqref{app} containing counter propagating oscillating wave packets of the abstract system \eqref{abstract} in subsection 3.2.
\subsection{\textbf{Reformulation}}
For notational convenience, we set the parameters $e=T_{i}=M_{i}=\frac{1}{4\pi}=1$ for the ion-Euler-Poisson system \eqref{iEP}. Setting the perturbation $(\rho,v)=(n_{i}-1,v_{i})$ around the constant equilibria $(1,0)$, we obtain the following Euler-Poisson system in terms of $(\rho, v,\phi)$,
\begin{subequations}\label{EP}
\begin{numcases}{}
\partial_{t}\rho+\partial_{x}v+\partial_{x}(\rho v)=0,\\
\partial_{t}v+v\partial_{x} v+\frac{\partial_{x}\rho}{1+\rho}+\partial_{x}\phi=0,\\
\partial_{x}^{2}\phi=e^{\phi}-1-\rho.\label{EP-poisson}
\end{numcases}
\end{subequations}
The Poisson equation \eqref{EP-poisson} defines an inverse operator $\rho\mapsto\phi(\rho)$
\begin{equation}
\begin{split}\label{equation5}
\phi(\rho)=&(1-\partial_{x}^{2})^{-1}\rho
-\frac{1}{2}(1-\partial_{x}^{2})^{-1}\big[(1-\partial_{x}^{2})^{-1}\rho\big]^{2}\\
&-\frac{1}{3!}(2+\partial_{x}^{2})(1-\partial_{x}^{2})^{-2}
\big[(1-\partial_{x}^{2})^{-1}\rho\big]^{3}+\mathcal{M}(\rho),
\end{split}
\end{equation}
where $\mathcal{M}$ represents the high order terms that satisfy some good properties \cite{GP11}. Thus we can further write \eqref{EP} the following form,
\begin{equation}
\begin{split}\label{equation6}
\partial_{t}&\begin{pmatrix}\rho\\ v\end{pmatrix}+\begin{pmatrix}0&\partial_{x}\\ \partial_{x}(1-\partial_{x}^{2})^{-1}+\partial_{x}&0\end{pmatrix}\begin{pmatrix}\rho\\ v\end{pmatrix}\\
=&\begin{pmatrix}-\partial_{x}(\rho v)\\-\frac{1}{2}\partial_{x}v^{2}+\frac{1}{2}\partial_{x}\rho^{2}
+\frac{1}{2}\partial_{x}(1-\partial_{x}^{2})^{-1}[(1-\partial_{x}^{2})^{-1}\rho]^{2}-\partial_{x}\mathcal{H}(\rho)]
\end{pmatrix},
\end{split}
\end{equation}
isolating the linear, quadratic and higher order terms. Here $$\mathcal{H}(\rho)=-\frac{1}{3!}(2+\partial_{x}^{2})(1-\partial_{x}^{2})^{-2}
\big[(1-\partial_{x}^{2})^{-1}\rho\big]^{3}+\frac{\rho^{3}}{3} +\mathcal{M}(\rho)+[\ln(1+\rho)-\rho+\frac{\rho^{2}}{2}-\frac{\rho^{3}}{3}].$$ Setting
\begin{equation}
\begin{split}\label{rho}
\begin{pmatrix}\rho\\ v\end{pmatrix}=S\begin{pmatrix}U_{1}\\ U_{-1}\end{pmatrix} \ \ \  and \ \ \ S=\begin{pmatrix}1& 1\\-q(|\partial_{x}|)& q(|\partial_{x}|)\end{pmatrix},
\end{split}
\end{equation}
we can diagonalize \eqref{equation6} into
\begin{equation}
\begin{split}\label{equation7}
\partial_{t}U_{j}
=&j\Omega U_{j}+\frac{1}{2}\partial_{x}\bigg[(U_{1}+U_{-1})q(|\partial_{x}|)(U_{1}-U_{-1})\bigg]
+\frac{j\partial_{x}}{4q(|\partial_{x}|)}\bigg[q(|\partial_{x}|)(U_{1}-U_{-1})\bigg]^{2}\\
&-\frac{j\partial_{x}}{4q(|\partial_{x}|)}(U_{1}+U_{-1})^{2}-\frac{j\partial_{x}}{4q(|\partial_{x}|)}\frac{1}{\langle\partial_{x}\rangle^{2}}
\bigg[\frac{1}{\langle\partial_{x}\rangle^{2}}(U_{1}+U_{-1})\bigg]^{2}\\
&+\frac{j}{2}\frac{\partial_{x}}{q(|\partial_{x}|)}\mathcal{H}(U_{1}+U_{-1}),
\end{split}
\end{equation}
where $j\in\{\pm1\}$, $\langle\partial_{x}\rangle=\sqrt{1-\partial_{x}^{2}}$ and $\widehat{\Omega}(k)=i\omega(k)$ with $\omega(k)=k\widehat{q}(k)$ (see \eqref{equation3}) in Fourier space. In this form, \eqref{equation7} fits in the abstract equation
\begin{equation}
\begin{split}\label{abstract}
\partial_{t}U_{j}=j\Omega U_{j}+Q_{j}(U,U)+N_{j}(U,U,U)+\mathcal{V}_{j}(U),
\end{split}
\end{equation}
where $Q_{j}(U,U)$, $N_{j}(U,U,U)$ and $\mathcal{V}_{j}(U)$ represent the quadratic, cubic and higher order terms, respectively. We remark that here we extract from the last term in \eqref{equation7} the cubic  terms $N_j(U,U,U)$,  which will be used in the process of deriving the NLS equations, in particular in the following \eqref{bet1}-\eqref{bet2}. 
In Fourier space, $Q_{j}(U,U)$ and $N_{j}(U,U,U)$ have the following form
\begin{equation}
\begin{split}\label{qua}
\widehat{Q_{j}(U,U)}(k,t)&=:\int\sum_{m,n\in\{\pm1\}}i\alpha_{mn}^{j}(k,k-l,l)
\widehat{U}_{m}(k-l,t)\widehat{U}_{n}(l,t)dl,\\
\widehat{N_{j}(U,U,U)}(k,t)&=:\int\int\sum_{m,n,p\in\{\pm1\}}i\beta_{mnp}^{j}(k,k-l,l-s,s)
\widehat{U}_{m}(k-l,t)\widehat{U}_{n}(l-s,t)\widehat{U}_{p}(s,t)dlds,
\end{split}
\end{equation}
with
\begin{equation}
\begin{split}\label{ker}
&\alpha_{mn}^{j}(k,k-l,l)=n\frac{k}{2}\hat{q}(l)
+\frac{jk}{4\hat{q}(k)}\Big[mn\hat{q}(k-l)\hat{q}(l)-1
-\frac{1}{\langle k\rangle^{2}\langle k-l\rangle^{2}\langle l\rangle^{2}}\Big],\\
&\beta_{mnp}^{j}(k,k-l,l-s,s)=-\frac{jk}{6\hat{q}(k)}\Big[\frac{2-k^{2}}{2\langle k\rangle^{2}\langle k-l\rangle^{2}\langle l-s\rangle^{2}\langle s\rangle^{2}}-1\Big].
\end{split}
\end{equation}
From the expression of the coefficients $\alpha_{mn}^{j}(k,k-l,l)$ and $\beta_{mnp}^{j}(k,k-l,l-s,s)$, we see that the worst growth in these coefficients is $\mathcal{O}(k)$ as $|k|\rightarrow\infty$ since $\hat{q}(k)=\mathcal{O}(1)$ as $|k|\rightarrow\infty$. Thus, the nonlinearity from \eqref{abstract} will lose one derivative.

\subsection{\textbf{Formal expansion}}
In this subsection, two independent NLS equations will be derived as formal approximations of the ion Euler-Poisson system. For this, we set
\begin{equation}
\begin{split}\label{app}
\epsilon\widetilde{\Theta}=&\epsilon\widetilde{\Psi}+\epsilon\widetilde{\Phi}+\epsilon^{2}\widetilde{\Upsilon},
\end{split}
\end{equation}
with
\begin{equation}
\begin{split}\label{appr}
&\epsilon\widetilde{\Psi}=\epsilon\widetilde{\Psi}_{1}+\epsilon\widetilde{\Psi}_{-1}+\epsilon^{2}\widetilde{\Psi}_{0}
+\epsilon^{2}\widetilde{\Psi}_{2}+\epsilon^{2}\widetilde{\Psi}_{-2},\\
&\epsilon\widetilde{\Phi}=\epsilon\widetilde{\Phi}_{1}+\epsilon\widetilde{\Phi}_{-1}+\epsilon^{2}\widetilde{\Phi}_{0}
+\epsilon^{2}\widetilde{\Phi}_{2}+\epsilon^{2}\widetilde{\Phi}_{-2},\\
&\epsilon^{2}\widetilde{\Upsilon}=\epsilon^{2}\widetilde{\Upsilon}_{1}+\epsilon^{2}\widetilde{\Upsilon}_{-1}
+\epsilon^{2}\widetilde{\Upsilon}_{11}+\epsilon^{2}\widetilde{\Upsilon}_{-1-1}
+\epsilon^{2}\widetilde{\Upsilon}_{1-1}+\epsilon^{2}\widetilde{\Upsilon}_{-11},
\end{split}
\end{equation}
where
\begin{align}
\epsilon\widetilde{\Psi}_{\pm1}
=&\epsilon\tilde{A}_{\pm1}(X+c_{g}T,\epsilon T)E^{\pm1}\begin{pmatrix}1\\0\end{pmatrix}, \
\epsilon^{2}\widetilde{\Psi}_{0}=\epsilon^{2}\begin{pmatrix}\tilde{A}_{01}(X+c_{g}T,\epsilon T)\\ \tilde{A}_{0-1}(X+c_{g}T,\epsilon T)\end{pmatrix},\nonumber\\
\epsilon^{2}\widetilde{\Psi}_{\pm2}=&\epsilon^{2}\begin{pmatrix}\tilde{A}_{\pm21}(X+c_{g}T,\epsilon T)\\ \tilde{A}_{\pm2-1}(X+c_{g}T,\epsilon T)\end{pmatrix}E^{\pm2},\nonumber\\
\epsilon\widetilde{\Phi}_{\pm1}
=&\epsilon\tilde{B}_{\pm1}(X-c_{g}T,\epsilon T)F^{\pm1}\begin{pmatrix}0\\1\end{pmatrix}, \
\epsilon^{2}\widetilde{\Phi}_{0}=\epsilon^{2}\begin{pmatrix} \tilde{B}_{01}(X-c_{g}T,\epsilon T)\\ \tilde{B}_{0-1}(X-c_{g}T,\epsilon T)\end{pmatrix},\nonumber\\
\epsilon^{2}\widetilde{\Phi}_{\pm2}=&\epsilon^{2}\begin{pmatrix}\tilde{B}_{\pm21}(X-c_{g}T,\epsilon T)\\ \tilde{B}_{\pm2-1}(X-c_{g}T,\epsilon T)\end{pmatrix}F^{\pm2},\nonumber\\
\epsilon^{2}\widetilde{\Upsilon}_{\pm1}
=&\epsilon^{2}\tilde{\mathcal{I}}_{\pm1}(X,T,\epsilon T)E^{\pm1}\begin{pmatrix}1\\ 0 \end{pmatrix}
+\epsilon^{2}\tilde{\mathcal{J}}_{\pm1}(X,T,\epsilon T)F^{\pm1}\begin{pmatrix}0\\ 1 \end{pmatrix}, \nonumber \\ \epsilon^{2}\widetilde{\Upsilon}_{\pm1\pm1}=&\epsilon^{2}\begin{pmatrix}\tilde{f}_{\pm1\pm11}(X,T,\epsilon T)\\ \tilde{f}_{\pm1\pm1-1}(X,T,\epsilon T)\end{pmatrix}E^{\pm1}F^{\pm1}. \label{lowapp}
\end{align}
Here and hereafter, $X=\epsilon x, \ T=\epsilon t, \ E^{j}=e^{ij(k_{0}x+\omega_{0}t)}, \ F^{j}=e^{ij(k_{0}x-\omega_{0}t)}$ with $\omega_{0}=\omega(k_{0})$, $\omega(k)$ is the linear dispersive relation defined in \eqref{equation3} and $c_{g}=\omega'(k_{0})$. Denote $\theta=\epsilon T, \ X_{+}=X+c_{g}T, \ X_{-}=X-c_{g}T$, then $\theta$ is the slow time scale and $X_{\pm}$ is the slow spatial scale. In addition, the complex functions satisfy $\overline{\tilde{A}_{j}}=\tilde{A}_{-j}$, $\overline{\tilde{B}_{j}}=\tilde{B}_{-j}$, $\overline{\tilde{\mathcal{I}}_{j}}=\tilde{\mathcal{I}}_{-j}$, $\overline{\tilde{\mathcal{J}}_{j}}=\tilde{\mathcal{J}}_{-j}$, $\overline{\tilde{A}_{2j}}=\tilde{A}_{-2j}$, $\overline{\tilde{B}_{2j}}=\tilde{B}_{-2j}$ and $\overline{\tilde{f}_{j_{1}j_{2}j}}=\tilde{f}_{-j_{1}-j_{2}j}$, and functions $\tilde{A}_{0j}$ and $\tilde{B}_{0j}$ are real functions with $j, j_{1},  j_{2}\in\{\pm1\}$.

Note that the ansatz $\epsilon\widetilde{\Psi}$ represents oscillating wave trains moving to the left, and the ansatz $\epsilon\widetilde{\Phi}$ represents oscillating wave trains moving to the right. In order to remove the interaction of these two counter propagating waves $\epsilon\widetilde{\Psi}$ and $\epsilon\widetilde{\Phi}$, we add the correction term $\epsilon^{2}\widetilde{\Upsilon}$ which will be defined later in the formal expansion $\epsilon\widetilde{\Theta}$. Inserting \eqref{app}-\eqref{lowapp} into \eqref{abstract}-\eqref{ker}, and replacing the dispersion relation $\omega=\omega(k)$ by their Taylor expansions around $k=jk_{0}$ in all the terms of $\omega\tilde{A}_{j}E^{j}, \ \omega\tilde{A}_{j\ell}E^{j}, \ \omega\tilde{B}_{j}F^{j},  \ \omega\tilde{B}_{j\ell}F^{j}, \ \omega\tilde{\mathcal{I}}_{j}E^{j},  \ \omega\tilde{\mathcal{J}}_{j}F^{j}$ and around $k=(j_{1}+j_{2})k_{0}$ in all the terms of $\omega\tilde{f}_{j_{1}j_{2}j}E^{j_{1}}F^{\tilde{j_{2}}}$, we have the following Lemma.
\begin{lemma}\label{L1}
Assume that $\Omega$ is a Fourier multiplier operator with symbol $i\omega(k)$. Let $\digamma_{\iota}$ be a complex amplitude function with $\iota=1,2,3$. Then we have
\begin{align*}
\Omega \big( \digamma_{1}(X+c_{g}T,\epsilon T)E^{j} \big)
=&i \bigg\{\omega(jk_{0})\digamma_{1}(X+c_{g}T,\epsilon T)
-i\epsilon\omega'(jk_{0})\partial_{X_{+}}\digamma_{1}(X+c_{g}T,\epsilon T)\\
&-\frac{1}{2}\epsilon^{2}\omega''(jk_{0})\partial_{X_{+}}^{2}
\digamma_{1}(X+c_{g}T,\epsilon T)+\cdot\cdot\cdot\bigg\}E^{j},\\
\Omega \big(\digamma_{2}(X-c_{g}T,\epsilon T)F^{j} \big)
=&i \bigg\{\omega(jk_{0})\digamma_{2}(X-c_{g}T,\epsilon T)
-i\epsilon\omega'(jk_{0})\partial_{X_{-}}\digamma_{2}(X-c_{g}T,\epsilon T)\\
&-\frac{1}{2}\epsilon^{2}\omega''(jk_{0})\partial_{X_{-}}^{2}\digamma_{2}(X-c_{g}T,\epsilon T)+\cdot\cdot\cdot\bigg\}F^{j},\\
\Omega \big(\digamma_{3}(X,T,\epsilon T)E^{j_{1}}F^{j_{2}} \big)
=&i \bigg\{\omega((j_{1}+j_{2})k_{0})\digamma_{3}(X,T,\epsilon T)
-i\epsilon\omega'((j_{1}+j_{2})k_{0})\partial_{X}\digamma_{3}(X,T,\epsilon T)\\
&-\frac{1}{2}\epsilon^{2}\omega''((j_{1}+j_{2})k_{0})\partial_{X}^{2}\digamma_{3}(X,T,\epsilon T)+\cdot\cdot\cdot\bigg\}E^{j_{1}}F^{j_{2}}.
\end{align*}
\end{lemma}
\begin{proof}
We can expand the symbol $\omega(k)$ in a Taylor series around $jk_{0}$ or $(j_{1}+j_{2})k_{0}$, and then compute the inverse Fourier transform of $\omega(k)\widehat{\digamma_{\iota}}(k)$ with $\iota=1,2,3$.
\end{proof}

Let the coefficients of the $\epsilon^{l}E^{j_{1}}F^{j_{2}}$ be zero, where the positive integer $l$ and the integers $j_{1}, \ j_{2}$ satisfy $|j_{1}|+|j_{2}|\leq l$.

For $\mathcal{O}(\epsilon E^{1}F^{0})$ and $\mathcal{O}(\epsilon E^{0}F^{1})$, we have
\begin{equation*}
\begin{split}
i\omega_{0}\tilde{A}_{1}&=i\omega(k_{0})\tilde{A}_{1},\\
-i\omega_{0}\tilde{B}_{1}&=-i\omega(k_{0})\tilde{B}_{1},
\end{split}
\end{equation*}
which are satisfied since $\omega_{0}=\omega(k_{0})$.

For $\mathcal{O}(\epsilon^{2}E^{0}F^{0})$, we have
\begin{equation*}
\begin{split}
0=j\omega(0)(\tilde{A}_{0j}+\tilde{B}_{0j})
+2\alpha^{j}_{11}\Big|_{k=0}\tilde{A}_{1}\tilde{A}_{-1}
+2\alpha^{j}_{-1-1}\Big|_{k=0}\tilde{B}_{1}\tilde{B}_{-1},
\end{split}
\end{equation*}
which is satisfied since $\omega(0)=\alpha^{j}_{mn}|_{k=0}=0$ for $m,n\in\{\pm1\}$ in light of \eqref{equation3} and \eqref{ker}.

For $\mathcal{O}(\epsilon^{2}E^{1}F^{0})$ and $\mathcal{O}(\epsilon^{2}E^{0}F^{1})$, we have
\begin{equation*}
\begin{split}
c_{g}\partial_{X_{+}}\tilde{A}_{1}+i\omega_{0}\tilde{\mathcal{I}}_{1}
&=\omega'(k_{0})\partial_{X_{+}}\tilde{A}_{1}+i\omega(k_{0})\tilde{\mathcal{I}}_{1},\\
-c_{g}\partial_{X_{-}}\tilde{B}_{1}-i\omega_{0}\tilde{\mathcal{J}}_{1}
&=-\omega'(k_{0})\partial_{X_{-}}\tilde{B}_{1}-i\omega(k_{0})\tilde{\mathcal{J}}_{1},
\end{split}
\end{equation*}
which are satisfied since $c_{g}=\omega'(k_{0})$ and $\omega_{0}=\omega(k_{0})$.

For $\mathcal{O}(\epsilon^{2}E^{2}F^{0})$ and $\mathcal{O}(\epsilon^{2}E^{0}F^{2})$, we have
\begin{equation*}
\begin{split}
2i\omega_{0}\tilde{A}_{2j}&=ij\omega(2k_{0})\tilde{A}_{2j}
+i\alpha^{j}_{11}(2k_{0},k_{0},k_{0})(\tilde{A}_{1})^{2},\\
-2i\omega_{0}\tilde{B}_{2j}&=ij\omega(2k_{0})\tilde{B}_{2j}
+i\alpha^{j}_{-1-1}(2k_{0},k_{0},k_{0})(\tilde{B}_{1})^{2},
\end{split}
\end{equation*}
where $\alpha_{mn}^{j}$ with $m, \ n\in\{\pm1\}$ is defined in \eqref{ker}.
Since $\omega(2k_{0})\neq \pm2\omega_{0}=2\omega(k_{0})$ according to the form of the dispersion relation \eqref{equation3}, $\tilde{A}_{2j}$ and $\tilde{B}_{2j}$ are well-defined in terms of $(\tilde{A}_{1})^{2}$ and $(\tilde{B}_{1})^{2}$ respectively as follows
\begin{equation}
\begin{split}\label {A2}
\tilde{A}_{2j}&=\frac{\alpha^{j}_{11}(2k_{0},k_{0},k_{0})}{2\omega_{0}
-j\omega(2k_{0})}(\tilde{A}_{1})^{2},\\
\tilde{B}_{2j}&=\frac{\alpha^{j}_{-1-1}(2k_{0},k_{0},k_{0})}
{-2\omega_{0}-j\omega(2k_{0})}(\tilde{B}_{1})^{2},
\end{split}
\end{equation}
where $\tilde{A}_{1}$ and $\tilde{B}_{1}$ will be fixed in the following at the order of $\mathcal{O}(\epsilon^{3}E^{1}F^{0})$ and $\mathcal{O}(\epsilon^{3}E^{0}F^{1})$.

For the interaction terms $\mathcal{O}(\epsilon^{2}E^{j_{1}}F^{j_{2}})$ with $j_{1}, j_{2}\in\{\pm1\}$, we have
\begin{equation*}
\begin{split}
i(j_{1}-j_{2})\omega_{0}\tilde{f}_{j_{1}j_{2}j}
=&ij\omega((j_{1}+j_{2})k_{0})\tilde{f}_{j_{1}j_{2}j}
+i\alpha^{j}_{j_{1}j_{2}}\left((j_{1}+j_{2})k_{0},j_{1}k_{0},j_{2}k_{0}\right)
\tilde{A}_{j_{1}}\tilde{B}_{j_{2}}.
\end{split}
\end{equation*}
Since $(j_{1}-j_{2})\omega_{0}-j\omega\big((j_{1}+j_{2})k_{0}\big)\neq0$ for $j_{1}, j_{2}\in\{\pm1\}$, $\tilde{f}_{j_{1}j_{2}j}$ is well-defined in terms of $\tilde{A}_{j_{1}}\tilde{B}_{j_{2}}$ by
\begin{equation}
\begin{split}\label{f11}
&\tilde{f}_{j_{1}j_{2}j}=\frac{\alpha^{j}_{j_{1}j_{2}}
\big((j_{1}+j_{2})k_{0},j_{1}k_{0},j_{2}k_{0}\big)}
{(j_{1}-j_{2})\omega_{0}-j\omega\big((j_{1}+j_{2})k_{0}\big)}
\tilde{A}_{j_{1}}\tilde{B}_{j_{2}},\ \ \ \ \ j\in\{\pm1\}.
\end{split}
\end{equation}
Note that $\alpha^{j}_{j_{1}j_{2}}
\big((j_{1}+j_{2})k_{0},j_{1}k_{0},j_{2}k_{0}\big)=0$ for $j_{1}+j_{2}=0$ in light of \eqref{ker}. Thus,  we have
\begin{align*}
\tilde{f}_{11j}&=\frac{\alpha^{j}_{11}
\big(2k_{0},k_{0},k_{0}\big)}
{-j\omega\big(2k_{0}\big)}
\tilde{A}_{1}\tilde{B}_{1},\\
\tilde{f}_{-1-1j}&=\overline{\tilde{f}_{11j}}, \
\tilde{f}_{1-1j}=\tilde{f}_{-11j}=0,
\end{align*}
 where $\tilde{A}_{j_{1}}, \tilde{B}_{j_{2}}$ with $j_{1},j_{2}\in\{\pm1\}$ will be fixed in the following at the order of $\mathcal{O}(\epsilon^{3}E^{1}F^{0})$ and $\mathcal{O}(\epsilon^{3}E^{0}F^{1})$.

Next we turn to $\mathcal{O}(\epsilon^{3})$. For $\mathcal{O}(\epsilon^{3}E^{0}F^{0})$, we have
\begin{align*}
c_{g}(\partial_{X_{+}}\tilde{A}_{0j}-\partial_{X_{-}}\tilde{B}_{0j})
=&j\omega'(0)(\partial_{X_{+}}\tilde{A}_{0j}
+\partial_{X_{-}}\tilde{B}_{0j})
+2\partial_{k}\alpha_{11}^{j}(0,k_{0},k_{0})\partial_{X_{+}}|\tilde{A}_{1}|^{2}\\
&+2\partial_{k}\alpha_{-1-1}^{j}(0,k_{0},k_{0})\partial_{X_{-}}|\tilde{B}_{1}|^{2}
+2\alpha_{11}^{j}(0,k_{0},k_{0})(\tilde{A}_{1}\widetilde{\mathcal{I}}_{-1}
+\tilde{A}_{-1}\widetilde{\mathcal{I}}_{1})\\
&+2\alpha_{-1-1}^{j}(0,k_{0},k_{0})(\tilde{B}_{1}\widetilde{\mathcal{J}}_{-1}
+\tilde{B}_{-1}\widetilde{\mathcal{J}}_{1}),
\end{align*}
where the terms in the form of $\tilde{A}_{1}\partial_{X_{+}}\tilde{A}_{-1}$
have been omitted since their coefficients are zero.
Note that $\alpha_{mn}^{j}(0,k_{0},k_{0})=0$ with $m,n\in\{\pm1\}$ in light of \eqref{ker}. The last two terms then vanish in the above equation. Since $c_{g}\pm\omega'(0)\neq0$, we can take the following choice such that $\tilde{A}_{0j}$ and $\tilde{B}_{0j}$ are well-defined,
\begin{equation}
\begin{split}\label{A0}
\tilde{A}_{0j}=\frac{2\partial_{k}\alpha_{11}^{j}(0,k_{0},k_{0})}
{c_{g}-j\omega'(0)}|\tilde{A}_{1}|^{2}, \ \ \ \ \
\tilde{B}_{0j}=\frac{2\partial_{k}\alpha_{-1-1}^{j}(0,k_{0},k_{0})}
{-c_{g}-j\omega'(0)}|\tilde{B}_{1}|^{2},
\end{split}
\end{equation}
where $\tilde{A}_{1}$ and $\tilde{B}_{1}$ will be fixed in the following at the order of $\mathcal{O}(\epsilon^{3}E^{1}F^{0})$ and $\mathcal{O}(\epsilon^{3}E^{0}F^{1})$.

For $\mathcal{O}(\epsilon^{3}E^{1}F^{0})$ and $\mathcal{O}(\epsilon^{3}E^{0}F^{1})$, we have
\begin{align*}
\partial_{\theta}\tilde{A}_{1}&+\partial_{T}\tilde{\mathcal{I}}_{1}
=-\frac{i}{2}\omega''(k_{0})\partial^{2}_{X_{+}}\tilde{A}_{1}
+\omega'(k_{0})\partial_{X}\tilde{\mathcal{I}}_{1}
+2i\sum_{j}\alpha_{1j}^{1}(k_{0},k_{0},0)\tilde{A}_{1}(\tilde{A}_{0j}+\tilde{B}_{0j})\nonumber\\
&+2i\sum_{j}\alpha_{1j}^{1}(k_{0},-k_{0},2k_{0})\tilde{A}_{-1}\tilde{A}_{2j}
+2i\sum_{j}\alpha_{-1j}^{1}(k_{0},k_{0},0)(\tilde{B}_{1}\tilde{f}_{1-1j}
+\tilde{B}_{-1}\tilde{f}_{11j})\nonumber\\
&+3i\beta_{111}^{1}(k_{0},k_{0},k_{0},-k_{0})\tilde{A}_{1}
(|\tilde{A}_{1}|^{2}+|\tilde{B}_{1}|^{2}),
\end{align*}
\begin{align*}
\partial_{\theta}\tilde{B}_{1}&+\partial_{T}\tilde{\mathcal{J}}_{1}
=\frac{i}{2}\omega''(k_{0})\partial^{2}_{X_{-}}\tilde{B}_{1}
-\omega'(k_{0})\partial_{X}\tilde{\mathcal{J}}_{1}
+2i\sum_{j}\alpha_{-1j}^{-1}(k_{0},k_{0},0)\tilde{B}_{1}(\tilde{B}_{0j}+\tilde{A}_{0j})\nonumber\\
&+2i\sum_{j}\alpha_{-1j}^{-1}(k_{0},-k_{0},2k_{0})\tilde{B}_{-1}\tilde{B}_{2j}
+2i\sum_{j}\alpha_{1j}^{-1}(k_{0},k_{0},0)(\tilde{A}_{1}\tilde{f}_{-11j}
+\tilde{A}_{-1}\tilde{f}_{11j})\\
&+3i\beta_{111}^{-1}(k_{0},k_{0},k_{0},-k_{0})\tilde{B}_{1}
(|\tilde{A}_{1}|^{2}+|\tilde{B}_{1}|^{2})\nonumber,
\end{align*}
where $\alpha_{mn}^{j}$ and $\beta_{mnp}^{j}$ with $m, \ n, \ p, \ j\in\{\pm1\}$ are defined in the equation \eqref{ker}.

Inserting \eqref{A2}-\eqref{A0} to above equations, we obtain
\begin{align}\label{bet1}
\partial_{\theta}\tilde{A}_{1}+\partial_{T}\tilde{\mathcal{I}}_{1}
=&-\frac{i}{2}\omega''(k_{0})\partial^{2}_{X_{+}}\tilde{A}_{1}
+\omega'(k_{0})\partial_{X}\tilde{\mathcal{I}}_{1}
+i\nu_{1}(k_{0})\tilde{A}_{1}|\tilde{A}_{1}|^{2}
+i\widetilde{\nu}_{1}(k_{0})\tilde{A}_{1}|\tilde{B}_{1}|^{2},
\end{align}
\begin{align}\label{bet2}
\partial_{\theta}\tilde{B}_{1}+\partial_{T}\tilde{\mathcal{J}}_{1}
=&\frac{i}{2}\omega''(k_{0})\partial^{2}_{X_{-}}\tilde{B}_{1}
-\omega'(k_{0})\partial_{X}\tilde{\mathcal{J}}_{1}
+i\nu_{2}(k_{0})\tilde{B}_{1}|\tilde{B}_{1}|^{2}
+i\widetilde{\nu}_{2}(k_{0})\tilde{B}_{1}|\tilde{A}_{1}|^{2},
\end{align}
where $\theta=\epsilon T=\epsilon^{2}t$ is defined below \eqref{lowapp} and $\nu_{m}(k_{0})$ and $\widetilde{\nu}_{m}(k_{0})$ with $m\in\{1,2\}$ are real valued  functions only depending on basic spatial wave number $k_{0}>0$. As mentioned above the correctors $\tilde{\mathcal{I}}_{1}$ and $\tilde{\mathcal{J}}_{1}$ are introduced to remove the interaction terms of the counter propagating waves. Thus we take the following form to make sure the above equalities are valid and then obtain the NLS equations for $\tilde{A}_{1}$ and $\tilde{B}_{1}$,
\begin{equation}
\begin{split}\label{NLS1}
\partial_{\theta}\tilde{A}_{1}=-\frac{i}{2}\omega''(k_{0})
\partial^{2}_{X_{+}}\tilde{A}_{1}
+i\nu_{1}(k_{0})\tilde{A}_{1}|\tilde{A}_{1}|^{2},
\end{split}
\end{equation}
\begin{equation}
\begin{split}\label{NLS2}
\partial_{\theta}\tilde{B}_{1}=\frac{i}{2}\omega''(k_{0})
\partial^{2}_{X_{-}}\tilde{B}_{1}
+i\nu_{2}(k_{0})\tilde{B}_{1}|\tilde{B}_{1}|^{2},
\end{split}
\end{equation}
and the evolutionary equations for the stationary waves $\tilde{\mathcal{I}}_{1}$ and $\tilde{\mathcal{J}}_{1}$,
\begin{equation}
\begin{split}\label{Modify1}
\partial_{T}\tilde{\mathcal{I}}_{1}-\omega'(k_{0})\partial_{X}\tilde{\mathcal{I}}_{1}
=i\widetilde{\nu}_{1}(k_{0})\tilde{A}_{1}|\tilde{B}_{1}|^{2},
\end{split}
\end{equation}
\begin{equation}
\begin{split}\label{Modify2}
\partial_{T}\tilde{\mathcal{J}}_{1}+\omega'(k_{0})\partial_{X}\tilde{\mathcal{J}}_{1}
=i\widetilde{\nu}_{2}(k_{0})\tilde{B}_{1}|\tilde{A}_{1}|^{2}.
\end{split}
\end{equation}
Recalling that $\omega'(k_{0})=c_{g}$ and applying the following Lemma \ref{L3}, we have the following estimates for the correctors $\tilde{\mathcal{I}}_{1}$ and $\tilde{\mathcal{J}}_{1}$
\begin{equation}
\begin{split}\label{IJ1}
\|\tilde{\mathcal{I}}_{1}(T)\|_{H^{s}}\leq C\|\tilde{A}_{1}\|_{H^{s}}\|\tilde{B}_{1}\|_{H^{s}}^{2},\\
\|\tilde{\mathcal{J}}_{1}(T)\|_{H^{s}}\leq C\|\tilde{B}_{1}\|_{H^{s}}\|\tilde{A}_{1}\|_{H^{s}}^{2},
\end{split}
\end{equation}
for all $T\geq0$.
\begin{lemma}[commutator estimate]\label{L10}
Let $m\geq1$ be an integer and define
\begin{equation}
\begin{split}\label{w}
[\nabla^{m},f]g:=\nabla^{m}(fg)-f\nabla^{m}g.
\end{split}
\end{equation}
Then
\begin{equation}
\begin{split}\label{w11}
\|[\nabla^{m},f]g\|_{L^{p}}\leq \|\nabla f\|_{L^{p_{1}}}\|\nabla^{m-1}g\|_{L^{p_{2}}}+\|\nabla^{m} f\|_{L^{p_{3}}}\|g\|_{L^{p_{4}}},
\end{split}
\end{equation}
where $p,  p_{2}, p_{3}\in(1,\infty)$ and
\begin{equation*}
\begin{split}
\frac{1}{p}=\frac{1}{p_{1}}+\frac{1}{p_{2}}=\frac{1}{p_{3}}+\frac{1}{p_{4}}.
\end{split}
\end{equation*}
\end{lemma}
\begin{proof}
The proof can be found in \cite{KP88}, for example.
\end{proof}
\begin{lemma}\label{L3}
Assume that $s$ is a nonnegative integer. Let $F^{+}, F^{-}\in H^{s}(\mathbb{R})$. If $h$ and $\tilde{h}$ are solutions of
\begin{align*}
[\partial_{T}-c_{g}\partial_{X}]h(T,X)=F^{+}(X+c_{g}T)[F^{-}(X-c_{g}T)]^{2}, \  \ \ \ \ h\big|_{T=0}=0,\\
[\partial_{T}+c_{g}\partial_{X}]\tilde{h}(T,X)=F^{-}(X-c_{g}T)[F^{+}(X+c_{g}T)]^{2}, \  \ \ \ \  \tilde{h}\big|_{T=0}=0.
\end{align*}
Then there exists a constant $C$ depending only on $c_{g}$ such that
\begin{align*}
\|h(T)\|_{H^{s}}&\leq C\|F^{+}\|_{H^{s}}\|F^{-}\|^{2}_{H^{s}},\\
\|\tilde{h}(T)\|_{H^{s}}&\leq C\|F^{-}\|_{H^{s}}\|F^{+}\|^{2}_{H^{s}},
\end{align*}
for all $s\geq0$.
\end{lemma}
\begin{proof}
By the method of characteristics, we have
\begin{align*}
h(T,X)=&\int_{0}^{T}F^{+}(X+(s-T)(-c_{g})+c_{g}s)(F^{-}(X+(s-T)(-c_{g})-c_{g}s))^{2}ds\\
=&F^{+}(X+c_{g}T)\int_{0}^{T}(F^{-}(X+c_{g}T-2c_{g}s))^{2}ds\\
=&\frac{1}{2c_{g}}F^{+}(X+c_{g}T) \int_{X-c_{g}T}^{X+c_{g}T}(F^{-}(\tau))^{2}d\tau.
\end{align*}
Then for $s=0$, we have
\begin{align*}
\|h(T)\|_{L^{2}}
\leq&\frac{1}{2c_{g}}\left\|F^{+}(X+c_{g}T)\right\|_{L^{2}}
\left\|\int_{X-c_{g}T}^{X+c_{g}T}(F^{-}(\tau))^{2}d\tau\right\|_{L^{\infty}}\\
\leq&\frac{1}{2c_{g}}\left\|F^{+}\right\|_{L^{2}}\left\|F^{-}\right\|^{2}_{L^{2}}.
\end{align*}
For $s\geq1$, we use \eqref{w} to write
\begin{align*}
\partial_{X}^{s}h(T,X)
=&\frac{1}{2c_{g}}\Bigg(\partial_{X}^{s}F^{+}(X+c_{g}T)
\int_{X-c_{g}T}^{X+c_{g}T}(F^{-}(\tau))^{2}d\tau\\
&+\left[\partial_{X}^{s}, \int_{X-c_{g}T}^{X+c_{g}T}(F^{-}(\tau))^{2}d\tau\right] F^{+}(X+c_{g}T)\Bigg).
\end{align*}
Then, we have
\begin{align*}
\|\partial_{X}^{s}h(T,X)\|_{L^{2}}
\leq&\frac{1}{2c_{g}}\Bigg(\left\|\partial_{X}^{s}F^{+}(X+c_{g}T)\right\|_{L^{2}}
\left\|\int_{X-c_{g}T}^{X+c_{g}T}(F^{-}(\tau))^{2}d\tau\right\|_{L^{\infty}}\\
&+\left\|\partial_{X}\int_{X-c_{g}T}^{X+c_{g}T}(F^{-}(\tau))^{2}d\tau\right\|_{L^{\infty}}
\|\partial_{X}^{s-1}F^{+}(X+c_{g}T)\|_{L^{2}}\\
&+\left\|\partial_{X}^{s}\int_{X-c_{g}T}^{X+c_{g}T}(F^{-}(\tau))^{2}d\tau\right\|_{L^{2}}
\|F^{+}(X+c_{g}T)\|_{L^{\infty}}\Bigg)\\
\leq&\frac{1}{2c_{g}}\Big(\|\partial_{X}^{s}F^{+}\|_{L^{2}}
\left\|F^{-}\right\|^{2}_{L^{2}}
+2\left\|F^{-}\right\|^{2}_{H^{1}}
\|\partial_{X}^{s-1}F^{+}\|_{L^{2}}
+\left\|F^{-}\right\|^{2}_{H^{s}}\|F^{+}\|_{H^{1}}\Big),
\end{align*}
where we have used commutator estimate \eqref{w11} and Sobolev embedding $H^{1}\hookrightarrow L^{\infty}$. So we have
\begin{align*}
\|h(T)\|_{H^{s}}
\leq C\left\|F^{+}\right\|_{H^{s}}\left\|F^{-}\right\|^{2}_{H^{s}}.
\end{align*}
Similarly, we can obtain the estimate for $\tilde{h}$.
\end{proof}

\section{\textbf{The Modified Approximation}}
To justify the bidirectional NLS approximation \eqref{app} (with  \eqref{NLS1} and \eqref{NLS2}), it is helpful to modify the approximation $\epsilon\widetilde{\Theta}$ to an improved approximation $\epsilon\Theta$. 
In this section, we will proceed in two steps to make such an improvement. First, we extend the approximation $\epsilon\widetilde{\Theta}$ to $\epsilon\widetilde{\Theta}^{ext}$ by adding higher order correction terms. The first step makes the residual \eqref{residual} be of order $\mathcal{O}(\epsilon^{m})$ for arbitrary large positive integer $m$. Secondly, we modify $\epsilon\widetilde{\Theta}^{ext}$ to $\epsilon\Theta$ by applying some cut-off function so that the support of $\epsilon\Theta$ is restricted to small neighborhoods of integer multiples of the basic wave number $k_{0}>0$ in Fourier space. The second step makes $\epsilon\Theta$ analytic and gives us a much simpler control on the resonances.

\subsection{\textbf{From $\epsilon\widetilde{\Theta}$ to $\epsilon\widetilde{\Theta}^{ext}$}}
In this step the previous approximation $\epsilon\widetilde{\Theta}$ is extended to $\epsilon\widetilde{\Theta}^{ext}$ by adding higher order terms. Specifically, we define
\begin{equation}
\begin{split}\label{ext}
\epsilon\widetilde{\Theta}^{ext}=&\epsilon\widetilde{\Psi}^{ext}
+\epsilon\widetilde{\Phi}^{ext}+\epsilon^{2}\widetilde{\Upsilon}^{ext},
\end{split}
\end{equation}
with
\begin{equation}
\begin{split}\label{extt}
\epsilon\widetilde{\Psi}^{ext}=\epsilon\widetilde{\Psi}+\epsilon^{2}\widetilde{\Psi}_{p}, \ \
\epsilon\widetilde{\Phi}^{ext}=\epsilon\widetilde{\Phi}+\epsilon^{2}\widetilde{\Phi}_{q}, \ \
\epsilon^{2}\widetilde{\Upsilon}^{ext}=\epsilon^{2}\widetilde{\Upsilon}
+\epsilon^{3}\widetilde{\Upsilon}_{r},
\end{split}
\end{equation}
where $\epsilon\widetilde{\Psi}$ and $\epsilon\widetilde{\Phi}$ and $\epsilon^{2}\widetilde{\Upsilon}$ are given in \eqref{appr}-\eqref{lowapp} and the others are of the form
\begin{align*}
\epsilon^{2}\widetilde{\Psi}_{p}=&\sum_{\ell\in\{\pm1\}}\sum_{n=1}^{4}\epsilon^{1+n}\begin{pmatrix} \tilde{A}_{\ell1}^{n}(X+c_{g}T,\epsilon T)\\ \tilde{A}_{\ell-1}^{n}(X+c_{g}T,\epsilon T)\end{pmatrix}E^{\ell}
    +\sum_{\ell\in\{\pm2\}}\sum_{n=1}^{3}\epsilon^{2+n}\begin{pmatrix}
    \tilde{A}_{\ell1}^{n}(X+c_{g}T,\epsilon T)
    \\ \tilde{A}_{\ell-1}^{n}(X+c_{g}T,\epsilon T)\end{pmatrix}E^{\ell} \\
    &+\sum_{n=1}^{3}\epsilon^{2+n}\begin{pmatrix}\tilde{A}_{01}^{n}(X+c_{g}T,\epsilon T)
    \\ \tilde{A}_{0-1}^{n}(X+c_{g}T,\epsilon T) \end{pmatrix}
    +\sum_{\ell\in\{\pm3\}}\sum_{n=0}^{2}\epsilon^{3+n}\begin{pmatrix}
    \tilde{A}_{\ell1}^{n}(X+c_{g}T,\epsilon T)
    \\ \tilde{A}_{\ell-1}^{n}(X+c_{g}T,\epsilon T) \end{pmatrix} E^{\ell} \\
    &+\sum_{\ell\in\{\pm4\}}\sum_{n=0}^{1}\epsilon^{4+n}\begin{pmatrix}
    \tilde{A}_{\ell1}^{n}(X+c_{g}T,\epsilon T)
    \\ \tilde{A}_{\ell-1}^{n}(X+c_{g}T,\epsilon T)  \end{pmatrix}E^{\ell}
    +\sum_{\ell\in\{\pm5\}}\epsilon^{5}\begin{pmatrix}\tilde{A}_{\ell1}^{0}(X+c_{g}T,\epsilon T)
    \\ \tilde{A}_{\ell-1}^{0}(X+c_{g}T,\epsilon T)  \end{pmatrix}E^{\ell},\\
\epsilon^{2}\widetilde{\Phi}_{q}=&\sum_{\ell\in\{\pm1\}}\sum_{n=1}^{4}\epsilon^{1+n}\begin{pmatrix} \tilde{B}_{\ell1}^{n}(X-c_{g}T,\epsilon T)\\ \tilde{B}_{\ell-1}^{n}(X-c_{g}T,\epsilon T)\end{pmatrix}F^{\ell}
    +\sum_{\ell\in\{\pm2\}}\sum_{n=1}^{3}\epsilon^{2+n}\begin{pmatrix}
    \tilde{B}_{\ell1}^{n}(X-c_{g}T,\epsilon T)
    \\ \tilde{B}_{\ell-1}^{n}(X-c_{g}T,\epsilon T)\end{pmatrix}F^{\ell} \\
    &+\sum_{n=1}^{3}\epsilon^{2+n}\begin{pmatrix}\tilde{B}_{01}^{n}(X-c_{g}T,\epsilon T)
    \\ \tilde{B}_{0-1}^{n}(X-c_{g}T,\epsilon T) \end{pmatrix}
    +\sum_{\ell\in\{\pm3\}}\sum_{n=0}^{2}\epsilon^{3+n}\begin{pmatrix}
    \tilde{B}_{\ell1}^{n}(X-c_{g}T,\epsilon T)
    \\ \tilde{B}_{\ell-1}^{n}(X-c_{g}T,\epsilon T) \end{pmatrix} F^{\ell} \\
    &+\sum_{\ell\in\{\pm4\}}\sum_{n=0}^{1}\epsilon^{4+n}\begin{pmatrix}
    \tilde{B}_{\ell1}^{n}(X-c_{g}T,\epsilon T)
    \\ \tilde{B}_{\ell-1}^{n}(X-c_{g}T,\epsilon T)  \end{pmatrix}F^{\ell}
    +\sum_{\ell=\pm5}\epsilon^{5}\begin{pmatrix}\tilde{B}_{\ell1}^{0}(X-c_{g}T,\epsilon T)
    \\ \tilde{B}_{\ell-1}^{0}(X-c_{g}T,\epsilon T)  \end{pmatrix}F^{\ell},\\
\epsilon^{3}\widetilde{\Upsilon}_{r}=&\sum_{\ell\in\{\pm1\}}\sum_{n=1}^{3}
\epsilon^{2+n}\begin{pmatrix}\tilde{\mathcal{I}}_{\ell1}^{n}(X,T,\epsilon T)\\ \tilde{\mathcal{I}}_{\ell-1}^{n}(X,T,\epsilon T) \end{pmatrix}E^{\ell}
+\sum_{\ell\in\{\pm1\}}\sum_{n=1}^{3}
\epsilon^{2+n}\begin{pmatrix}\tilde{\mathcal{J}}_{\ell1}^{n}(X,T,\epsilon T)\\ \tilde{\mathcal{J}}_{\ell-1}^{n}(X,T,\epsilon T) \end{pmatrix}F^{\ell}\\
&+\sum_{\ell_{1},\ell_{2}\in\{\pm1\}}\sum_{n=1}^{3}\epsilon^{2+n}
    \begin{pmatrix}\tilde{f}_{\ell_{1}\ell_{2}1}^{n}(X,T,\epsilon T)
    \\ \tilde{f}_{\ell_{1}\ell_{2}-1}^{n}(X,T,\epsilon T) \end{pmatrix} E^{\ell_{1}}F^{\ell_{2}}\\
&+\sum_{\substack{\ell_{1},\ell_{2}\in\{\pm1,\pm2\}\\|\ell_{1}|+|\ell_{2}|=3}}
\sum_{n=0}^{2}\epsilon^{3+n}
    \begin{pmatrix}\tilde{f}_{\ell_{1}\ell_{2}1}^{n}(X,T,\epsilon T)
    \\ \tilde{f}_{\ell_{1}\ell_{2}-1}^{n}(X,T,\epsilon T) \end{pmatrix} E^{\ell_{1}}F^{\ell_{2}}\\
&+\sum_{\substack{\ell_{1},\ell_{2}\in\{\pm1,\pm2,\pm3\}\\|\ell_{1}|+|\ell_{2}|=4}}
\sum_{n=0}^{1}\epsilon^{4+n}
    \begin{pmatrix}\tilde{f}_{\ell_{1}\ell_{2}1}^{n}(X,T,\epsilon T)
    \\ \tilde{f}_{\ell_{1}\ell_{2}-1}^{n}(X,T,\epsilon T) \end{pmatrix} E^{\ell_{1}}F^{\ell_{2}}\\
&+\sum_{\substack{\ell_{1},\ell_{2}\in\{\pm1,\pm2,\pm3,\pm4\}\\|\ell_{1}|+|\ell_{2}|=5}}
\epsilon^{5}\begin{pmatrix}\tilde{f}_{\ell_{1}\ell_{2}1}^{0}(X,T,\epsilon T)
    \\ \tilde{f}_{\ell_{1}\ell_{2}-1}^{0}(X,T,\epsilon T) \end{pmatrix} E^{\ell_{1}}F^{\ell_{2}},
\end{align*}
where $\tilde{A}_{-\ell j}^{n}=\ \overline{\tilde{A}_{\ell j}^{n}}$, $\tilde{B}_{-\ell j}^{n}=\overline{\tilde{B}_{\ell j}^{n}}$, $\tilde{\mathcal{I}}_{-\ell j}^{n}=\overline{\tilde{\mathcal{I}}_{\ell j}^{n}}$, $\tilde{\mathcal{J}}_{-\ell j}^{n}=\overline{\tilde{\mathcal{J}}_{\ell j}^{n}}$ and $\tilde{f}_{-\ell_{1}-\ell_{2}j}^{n} =\overline{\tilde{f}_{\ell_{1}\ell_{2}j}^{n}}$ with $j\in\{\pm1\}$. The functions $\tilde{A}_{\ell j}, \ \tilde{A}_{\pm 1}, \ \tilde{B}_{\ell j}, \ \tilde{B}_{\pm 1}, \ \tilde{\mathcal{I}}_{\pm 1}, \ \tilde{\mathcal{J}}_{\pm 1}$ and $\tilde{f}_{\pm1\pm1j}$ from Section 3 correspond to $\tilde{A}_{\ell j}^{0}, \ \tilde{A}_{\pm11}^{0}, \ \tilde{B}^{0}_{\ell j}, \ \tilde{B}_{\pm1-1}^{0}, \ \tilde{\mathcal{I}}_{\pm11}^{0}, \ \tilde{\mathcal{J}}_{\pm 1-1}^{0}$, and $\tilde{f}_{\pm1\pm1j}^{0}$, respectively. In the $j^{th}$ component, we can rewrite the extended approximation in \eqref{ext} as
\begin{align}\label{appext}
\epsilon\widetilde{\Theta}_{j}^{ext}=\epsilon\widetilde{\Psi}_{j}^{ext}
+\epsilon\widetilde{\Phi}_{j}^{ext}
+\epsilon^{2}\widetilde{\Upsilon}_{j}^{ext},
\end{align}
with
\begin{align*}
\epsilon\widetilde{\Psi}_{j}^{ext}
=:&\sum_{|\ell|\leq5}\sum_{\lambda_{j}(\ell,n)\leq5}\epsilon^{\lambda_{j}(\ell,n)}
 \tilde{A}_{\ell j}^{n}(X+c_{g}T,\epsilon T)E^{\ell},\nonumber\\
\epsilon\widetilde{\Phi}_{j}^{ext}
=:&\sum_{|\ell|\leq5}\sum_{\eta_{j}(\ell,n)\leq5}\epsilon^{\eta_{j}(\ell,n)} \tilde{B}_{\ell j}^{n}(X-c_{g}T,\epsilon T)F^{\ell},\nonumber\\
\epsilon^{2}\widetilde{\Upsilon}_{j}^{ext}=&\sum_{\ell\in\{\pm1\}}\sum_{2\leq\zeta_{j}(n)\leq5}
\epsilon^{\zeta_{j}(n)}\tilde{\mathcal{I}}_{\ell j}^{n}(X,T,\epsilon T)E^{\ell}+\sum_{\ell\in\{\pm1\}}\sum_{2\leq\mu_{j}(n)\leq5}
\epsilon^{\mu_{j}(n)}\tilde{\mathcal{J}}_{\ell j}^{n}(X,T,\epsilon T)F^{\ell}\nonumber\\
&+\sum_{\substack{2\leq|\ell_{1}|+|\ell_{2}|\leq5\\ \ell_{1}\ell_{2}\neq0}}
\sum_{2\leq\chi(\ell_{1},\ell_{2},n)\leq5}\epsilon^{\chi(\ell_{1},\ell_{2},n)}
\tilde{f}_{\ell_{1}\ell_{2}j}^{n}(X,T,\epsilon T)E^{\ell_{1}}F^{\ell_{2}},
\end{align*}
where $j\in\{\pm1\}, \ \ell\in\mathbb{Z}, \ n\in\mathbb{N}, \ \lambda_{j}(\ell,n)=1+||\ell|-1|+n, \ \eta_{j}(\ell,n)=1+||\ell|-1|+n, \ \zeta_{j}(n)=\mu_{j}(n)=2+n$ and $\chi(\ell_{1},\ell_{2},n)=|\ell_{1}|+|\ell_{2}|+n$. In addition, $\lambda_{1}(\ell,n)=\lambda_{-1}(\ell,n)$ except for $\lambda_{-1}(1,n)=\lambda_{1}(1,n)+2$, $\eta_{1}(\ell,n)=\eta_{-1}(\ell,n)$ except for $\eta_{1}(1,n)=\eta_{-1}(1,n)+2$, $\zeta_{-1}(n)=\zeta_{1}(n)+1$ and $\mu_{1}(n)=\mu_{-1}(n)+1$.

Inserting the above extended approximation \eqref{appext} into system \eqref{equation7}, we will show that the residual $Res(\epsilon\widetilde{\Theta}^{ext})$ (see \eqref{residual}) is formally at least of order $\mathcal{O}(\epsilon^{6})$ if $\tilde{A}^{n}_{\ell j}$, $\tilde{B}^{n}_{\ell j}$, $\tilde{\mathcal{I}}_{\ell j}^{n}$,  $\tilde{\mathcal{J}}_{\ell j}^{n}$ and $\tilde{f}_{\ell_{1}\ell_{2}j}^{n}$ are chosen in a suitable way. The functions $\tilde{A}^{n}_{\ell j}$ in $\epsilon^{2}\widetilde{\Psi}_{p}$, $\tilde{B}^{n}_{\ell j}$ in $\epsilon^{2}\widetilde{\Phi}_{q}$ and $\tilde{\mathcal{I}}_{\ell j}^{n}$, $\tilde{\mathcal{J}}_{\ell j}^{n}$, $\tilde{f}_{\ell_{1}\ell_{2}j}^{n}$ in $\epsilon^{3}\widetilde{\Upsilon}_{r}$ are chosen by a similar procedure as $\tilde{A}_{\pm 1}$, $\tilde{A}_{\pm2j}$, $\tilde{A}_{0j}$ in $\epsilon\widetilde{\Psi}$, $\tilde{B}_{\pm 1}$, $\tilde{B}_{\pm2j}$, $\tilde{B}_{0j}$ in $\epsilon\widetilde{\Phi}$ and $\tilde{\mathcal{I}}_{\pm1}$, $\tilde{\mathcal{J}}_{\pm1}$, $\tilde{f}_{\pm1\pm1j}$ in $\epsilon^{2}\widetilde{\Upsilon}$ respectively.

In Section 3, we obtain the NLS equations \eqref{NLS1}-\eqref{NLS2} for $\tilde{A}_{1}=\tilde{A}^{0}_{11}$ and $\tilde{B}_{1}=\tilde{B}^{0}_{1-1}$ by letting the coefficients of $\mathcal{O}(\epsilon^{3}E^{1}F^{0})$ and $\mathcal{O}(\epsilon^{3}E^{0}F^{1})$ be zero, i.e.,
\begin{equation}
\begin{split}\label{NLS11}
\partial_{\theta}\tilde{A}^{0}_{11}=-\frac{i}{2}\omega''(k_{0})
\partial^{2}_{X_{+}}\tilde{A}^{0}_{11}
+i\nu_{1}(k_{0})\tilde{A}^{0}_{11}|\tilde{A}^{0}_{11}|^{2},
\end{split}
\end{equation}
\begin{equation}
\begin{split}\label{NLS12}
\partial_{\theta}\tilde{B}^{0}_{1-1}=\frac{i}{2}\omega''(k_{0})
\partial^{2}_{X_{-}}\tilde{B}^{0}_{1-1}
+i\nu_{2}(k_{0})\tilde{B}^{0}_{1-1}|\tilde{B}^{0}_{1-1}|^{2}.
\end{split}
\end{equation}
In fact, we also obtain the equations of $\tilde{A}^{0}_{1-1}$ and $\tilde{B}^{0}_{11}$ by letting the coefficients of $\mathcal{O}(\epsilon^{3}E^{1}F^{0})$ and $\mathcal{O}(\epsilon^{3}E^{0}F^{1})$ for the other components be zero
\begin{align*}
i\omega_{0}(\tilde{A}^{0}_{1-1}&+\tilde{\mathcal{I}}^{0}_{1-1})
=-i\omega(k_{0})(\tilde{A}^{0}_{1-1}+\tilde{\mathcal{I}}^{0}_{1-1})
+2i\sum_{j}\alpha_{1j}^{-1}(k_{0},k_{0},0)\tilde{A}^{0}_{11}(\tilde{A}^{0}_{0j}+\tilde{B}^{0}_{0j})\\
&+2i\sum_{j}\alpha_{1j}^{-1}(k_{0},-k_{0},2k_{0})\tilde{A}^{0}_{-11}\tilde{A}^{0}_{2j}
+2i\sum_{j}\alpha_{-1j}^{-1}(k_{0},k_{0},0)\tilde{B}^{0}_{1-1}\tilde{f}^{0}_{1-1j}\\
&+2i\sum_{j}\alpha_{-1j}^{-1}(k_{0},-k_{0},2k_{0})\tilde{B}^{0}_{-1-1}\tilde{f}^{0}_{11j}
+3i\beta_{111}^{-1}(k_{0},k_{0},k_{0},-k_{0})\tilde{A}^{0}_{11}
|\tilde{A}_{1}|^{2}\\
&+6i\beta_{111}^{-1}(k_{0},k_{0},k_{0},-k_{0})\tilde{A}^{0}_{11}|\tilde{B}_{1}|^{2},\\
-i\omega_{0}(\tilde{B}^{0}_{11}&+\tilde{\mathcal{J}}^{0}_{11})
=i\omega(k_{0})(\tilde{B}^{0}_{11}+\tilde{\mathcal{J}}^{0}_{11})
+2i\sum_{j}\alpha_{-1j}^{1}(k_{0},k_{0},0)\tilde{B}^{0}_{1-1}(\tilde{B}^{0}_{0j}+\tilde{A}^{0}_{0j})\\
&+2i\sum_{j}\alpha_{-1j}^{1}(k_{0},-k_{0},2k_{0})\tilde{B}^{0}_{-1-1}\tilde{B}^{0}_{2j}
+2i\sum_{j}\alpha_{1j}^{1}(k_{0},k_{0},0)\tilde{A}_{-1}\tilde{f}_{11j}\\
&+2i\sum_{j}\alpha_{1j}^{1}(k_{0},-k_{0},2k_{0})\tilde{A}_{-1}\tilde{f}_{11j}
+6i\beta_{111}^{1}(k_{0},k_{0},k_{0},-k_{0})\tilde{B}^{0}_{1-1}
|\tilde{A}^{0}_{11}|^{2}\\
&+3i\beta_{111}^{1}(k_{0},k_{0},k_{0},-k_{0})\tilde{B}^{0}_{1-1}|\tilde{B}^{0}_{1-1}|^{2}.
\end{align*}
Inserting \eqref{A2}-\eqref{A0} to above equations, we obtain
\begin{align*}
i\omega_{0}(\tilde{A}^{0}_{1-1} +\tilde{\mathcal{I}}^{0}_{1-1}) &=-i\omega(k_{0})(\tilde{A}^{0}_{1-1} +\tilde{\mathcal{I}}^{0}_{1-1}) +i\sigma_{1}(k_{0})\tilde{A}^{0}_{11}|\tilde{A}^{0}_{11}|^{2} +i\tilde{\sigma}_{1}(k_{0})\tilde{A}^{0}_{11}|\tilde{B}^{0}_{1-1}|^{2},\\
-i\omega_{0}(\tilde{B}^{0}_{11} +\tilde{\mathcal{J}}^{0}_{11}) &=i\omega(k_{0})(\tilde{B}^{0}_{11}+\tilde{\mathcal{J}}^{0}_{11}) +i\sigma_{2}(k_{0})\tilde{B}^{0}_{1-1}|\tilde{B}^{0}_{1-1}|^{2} +i\tilde{\sigma}_{2}(k_{0})\tilde{B}^{0}_{1-1}|\tilde{A}^{0}_{11}|^{2},
\end{align*}
where $\sigma_{m}(k_{0})$ and $\widetilde{\sigma}_{m}(k_{0})$ with $m\in\{1,2\}$ are real functions depending only on basic spatial wave number $k_{0}>0$. Thus we can take the following form such that the above equalities hold
\begin{equation}
\begin{split}\label{Modify22}
\tilde{A}^{0}_{1-1}&
=\frac{\sigma_{1}(k_{0})}{2\omega(k_{0})}\tilde{A}^{0}_{11}|\tilde{A}^{0}_{11}|^{2}, \
 \ \ \ \ \ \ \ \ \  \tilde{\mathcal{I}}^{0}_{1-1}
=\frac{\tilde{\tilde{\sigma}}_{1}(k_{0})}{2\omega(k_{0})}\tilde{A}^{0}_{11}|\tilde{B}^{0}_{1-1}|^{2},\\
\tilde{B}^{0}_{11}&
=-\frac{\sigma_{2}(k_{0})}{2\omega(k_{0})}\tilde{B}^{0}_{1-1}|\tilde{B}^{0}_{1-1}|^{2}, \ \ \ \ \
\tilde{\mathcal{J}}^{0}_{11}
=-\frac{\tilde{\sigma}_{1}(k_{0})}{2\omega(k_{0})}\tilde{B}^{0}_{1-1}|\tilde{A}^{0}_{11}|^{2},
\end{split}
\end{equation}
where $\tilde{A}^{0}_{11}$ and $\tilde{B}^{0}_{1-1}$ satisfy the NLS equations \eqref{NLS11} and \eqref{NLS12}.

For $\mathcal{O}(\epsilon^{3}E^{2}F^{0})$ and $\mathcal{O}(\epsilon^{3}E^{0}F^{2})$, we have
\begin{align}\label{A21}
(2i\omega_{0}-ij\omega(2k_{0}))\tilde{A}^{1}_{2j}
=&(j\omega'(2k_{0})-c_{g})\partial_{X_{+}}\tilde{A}^{0}_{2j}
+2i\sum_{\iota=1,2,3} \partial_{\iota}\alpha_{11}^{j}(2k_{0},k_{0}, k_{0})\partial_{X_{+}}(\tilde{A}^{0}_{11})^{2} \nonumber\\
&+2\alpha_{11}^{j}(2k_{0},k_{0},k_{0})\tilde{A}^{0}_{11} (\tilde{A}^{1}_{11}
+\tilde{\mathcal{I}}^{0}_{11}),\nonumber\\
(-2i\omega_{0}-ij\omega(2k_{0}))\tilde{B}^{1}_{2j}
=&(j\omega'(2k_{0})+c_{g})\partial_{X_{-}}\tilde{B}^{0}_{2j} +2i\sum_{\iota=1,2,3}\partial_{\iota}\alpha_{-1-1}^{j}(2k_{0}, k_{0},k_{0})\partial_{X_{-}}(\tilde{B}^{0}_{1-1})^{2}\nonumber\\
&+2\alpha_{-1-1}^{j}(2k_{0},k_{0},k_{0})\tilde{B}^{0}_{1-1}(\tilde{B}^{1}_{1-1}
+\tilde{\mathcal{J}}^{0}_{1-1}).
\end{align}
In light of \eqref{A2} and \eqref{Modify1}-\eqref{Modify2}, we know that $\tilde{A}^{1}_{2j}$ and $\tilde{B}^{1}_{2j}$ will be determined since $2\omega_{0}\pm j\omega(2k_{0})\neq0$, where $\tilde{A}^{1}_{11}$ and $\tilde{B}^{1}_{1-1}$ satisfy inhomogeneous linear NLS equations \eqref{linearNLS} which will be fixed at the order of $\mathcal{O}(\epsilon^{4}E^{1}F^{0})$ and $\mathcal{O}(\epsilon^{4}E^{0}F^{1})$ below.

For $\mathcal{O}(\epsilon^{3}E^{3}F^{0})$ and $\mathcal{O}(\epsilon^{3}E^{0}F^{3})$, we have
\begin{equation}
\begin{split}\label{A03}
(3i\omega_{0}&-ij\omega(3k_{0}))\tilde{A}^{0}_{3j}
=2i\sum_{\tilde{j}}\alpha_{1\tilde{j}}^{j}(3k_{0},k_{0},2k_{0})
\tilde{A}^{0}_{11}\tilde{A}^{0}_{2\tilde{j}}
+i\beta_{111}^{j}(3k_{0},k_{0},k_{0},k_{0})(\tilde{A}^{0}_{11})^{3},\\
(-3i\omega_{0}&-ij\omega(3k_{0}))\tilde{B}^{0}_{3j}
=2i\sum_{\tilde{j}}\alpha_{-1\tilde{j}}^{j}(3k_{0},k_{0},2k_{0})
\tilde{B}^{0}_{1-1}\tilde{B}^{0}_{2\tilde{j}}
+i\beta_{-1-1-1}^{j}(3k_{0},k_{0},k_{0},k_{0})(\tilde{B}^{0}_{1-1})^{3}.
\end{split}
\end{equation}
In light of \eqref{A2} and the NLS equations \eqref{NLS11}-\eqref{NLS12}, $\tilde{A}^{0}_{3j}$ and $\tilde{B}^{0}_{3j}$ can be determined since $3\omega_{0}\pm j\omega(3k_{0})\neq0$.

For the interaction terms $\mathcal{O}(\epsilon^{3}E^{\ell_{1}}F^{\ell_{2}})$ with $\ell_{1}\in\{\pm1\}$ and $\ell_{2}\in\{1, \ 2\}$, we have
\begin{equation}
\begin{split}\label{f111}
\big[i(\ell_{1}&-1)\omega_{0}-ij\omega((\ell_{1}+1)k_{0})\big]\tilde{f}^{1}_{\ell_{1}1j}
=j\omega'((\ell_{1}+1)k_{0})\partial_{X}\tilde{f}^{0}_{\ell_{1}1j}
-\partial_{T}\tilde{f}^{0}_{\ell_{1}1j}\\
&+2\sum_{\iota=1,2,3}\partial_{\iota}\alpha^{j}_{1-1}((\ell_{1}+1)k_{0},\ell_{1}k_{0},k_{0})
\partial_{X}(\tilde{A}^{0}_{\ell_{1}1}\tilde{B}^{0}_{1-1})\\
&+2i\alpha^{j}_{1-1}((\ell_{1}+1)k_{0},\ell_{1}k_{0},k_{0})
\Big(\tilde{A}^{0}_{\ell_{1}1}(\tilde{\mathcal{J}}^{0}_{1-1}+\tilde{B}_{1-1}^{1})
+\tilde{B}^{0}_{1-1}(\tilde{\mathcal{I}}^{0}_{\ell_{1}1}+\tilde{A}_{\ell_{1}1}^{1})\Big),\\
\big[i(\ell_{1}&-2)\omega_{0}-ij\omega((\ell_{1}+2)k_{0})\big]\tilde{f}^{0}_{\ell_{1}2j}
=2i\sum_{\tilde{j}}\alpha^{j}_{1\tilde{j}}((\ell_{1}+2)k_{0},\ell_{1}k_{0},2k_{0})
\tilde{A}^{0}_{\ell_{1}1}\tilde{B}^{0}_{2\tilde{j}}\\
&+2i\sum_{\tilde{j}}\alpha^{j}_{-1\tilde{j}}((\ell_{1}+2)k_{0},\ell_{1}k_{0},2k_{0})
\tilde{B}^{0}_{1-1}\tilde{f}_{\ell_{1}1\tilde{j}}^{0}).
\end{split}
\end{equation}
In light of \eqref{A2}, \eqref{f11}, \eqref{Modify1}-\eqref{Modify2} and the NLS equations \eqref{NLS11}-\eqref{NLS12}, $\tilde{f}^{1}_{\ell_{1}1j}$ and $\tilde{f}^{0}_{\ell_{1}2j}$ will be determined since $i(\ell_{1}-\ell_{2})\omega_{0}-ij\omega((\ell_{1}+\ell_{2})k_{0})\neq0$ for $\ell_{1}=\pm1, \ \ell_{2}=1,2$, where $\tilde{A}^{1}_{11}$ and $\tilde{B}^{1}_{1-1}$ satisfy inhomogeneous linear NLS equations \eqref{linearNLS} which will be fixed at the order of $\mathcal{O}(\epsilon^{4}E^{1}F^{0})$ and $\mathcal{O}(\epsilon^{4}E^{0}F^{1})$. Similarly, the case $\ell_{1}=1, \ 2, \ \ell_{2}=\pm1$ can be obtained.

Now we consider the terms of $\mathcal{O}(\epsilon^{4})$. We proceed exactly as before by writing out the terms proportional to $\epsilon^{4}E^{\ell_{1}}F^{\ell_{2}}$ for $|\ell_{1}|+|\ell_{2}|\leq4$, though we have to consider more choices of $(\ell_{1},\ell_{2})$ to account for additional terms generated by the nonlinearity. In fact, only the cases $(\ell_{1},\ell_{2})=(0,0), \ (\ell_{1},\ell_{2})=(1,0)$ and $(\ell_{1},\ell_{2})=(0,1)$ really need additional comment.

For $\mathcal{O}(\epsilon^{4}E^{0}F^{0})$, we have
\begin{align*}
c_{g}&(\partial_{X_{+}}\tilde{A}^{1}_{0j}-\partial_{X_{-}}\tilde{B}^{1}_{0j})
+\partial_{\theta}\tilde{A}^{0}_{0j}+\partial_{\theta}\tilde{B}^{0}_{0j}
=j\omega'(0)(\partial_{X_{+}}\tilde{A}^{1}_{0j}
+\partial_{X_{-}}\tilde{B}^{1}_{0j})\\
&+2\partial_{k}\alpha_{11}^{j}(0,k_{0},-k_{0})\partial_{X}
\Big[\tilde{A}^{0}_{11}(\tilde{A}^{1}_{-11}+\mathcal{I}^{0}_{-11})
+\tilde{A}^{0}_{-11}(\tilde{A}^{1}_{11}+\mathcal{I}^{0}_{11})\Big]\\
&+2\partial_{k}\alpha_{-1-1}^{j}(0,k_{0},-k_{0})\partial_{X}
\Big[\tilde{B}^{0}_{1-1}(\tilde{B}^{1}_{-1-1}+\mathcal{I}^{0}_{-1-1})
+\tilde{B}^{0}_{-1-1}(\tilde{B}^{1}_{1-1}+\mathcal{J}^{0}_{1-1})\Big].
\end{align*}
The expression is simplified since $\omega''(0)=0$ and $\alpha_{\ell_{1}\ell_{2}}^{j}\Big|_{k=0}=\partial_{k}^{2}\alpha_{11}^{j}\Big|_{k=0}=0$ by the form \eqref{qua}-\eqref{ker} of the quadratic term $Q(U,U)$ in Fourier space. We can take the following choice such that the above equations are valid
\begin{equation}
\begin{split}\label{A01}
(c_{g}-&j\omega'(0))\partial_{X_{+}}\tilde{A}^{1}_{0j}
+\partial_{\theta}\tilde{A}^{0}_{0j}\\
&=2\partial_{k}\alpha_{11}^{j}(0,k_{0},-k_{0})\partial_{X}
\Big[\tilde{A}^{0}_{11}(\tilde{A}^{1}_{-11}+\mathcal{I}^{0}_{-11})
+\tilde{A}^{0}_{-11}(\tilde{A}^{1}_{11}+\mathcal{I}^{0}_{11})\Big],\\
(-c_{g}&-j\omega'(0))\partial_{X_{-}}\tilde{B}^{1}_{0j}
+\partial_{\theta}\tilde{B}^{0}_{0j}\\
&=2\partial_{k}\alpha_{-1-1}^{j}(0,k_{0},-k_{0})\partial_{X}
\Big[\tilde{B}^{0}_{1-1}(\tilde{B}^{1}_{-1-1}+\mathcal{J}^{0}_{-1-1})
+\tilde{B}^{0}_{-1-1}(\tilde{B}^{1}_{1-1}+\mathcal{J}^{0}_{1-1})\Big],
\end{split}
\end{equation}
then $\tilde{A}^{1}_{0j}$ and $\tilde{B}^{1}_{0j}$ can be determined since $ c_{g}\pm j\omega'(0)\neq0$. All other terms except for $\tilde{A}^{1}_{\pm11}$ and $\tilde{B}^{1}_{\pm1-1}$ in \eqref{A01} have been defined at previous steps in the iterative procedure. The resulting equation requires some integration w.r.t $X$ to give $\tilde{A}^{1}_{0j}$ and $\tilde{B}^{1}_{0j}$. In order to avoid this integration we have to show that the problematic terms $\partial_{\theta}\tilde{A}^{0}_{0j}$ and $\partial_{\theta}\tilde{B}^{0}_{0j}$ can be written as a derivative w.r.t $X$ in order to insure that $\tilde{A}^{1}_{0j}$ and $\tilde{B}^{1}_{0j}$ have the appropriate decay properties as $X\rightarrow\pm\infty$.

Recall that from \eqref{A0}, we have
\begin{equation}
\begin{split}\label{A02}
\partial_{\theta}\tilde{A}^{0}_{0j}&=\frac{2\partial_{k}\alpha_{11}^{j}(0,k_{0},k_{0})}
{c_{g}-j\omega'(0)}\Big((\partial_{\theta}\tilde{A}^{0}_{11})\tilde{A}^{0}_{-11}
+\tilde{A}^{0}_{11}(\partial_{\theta}\tilde{A}^{0}_{-11})\Big),\\
\partial_{\theta}\tilde{B}^{0}_{0j}&=\frac{2\partial_{k}\alpha_{-1-1}^{j}(0,k_{0},k_{0})}
{-c_{g}-j\omega'(0)}\Big((\partial_{\theta}\tilde{B}^{0}_{1-1})\tilde{B}^{0}_{-1-1}
+\tilde{B}^{0}_{1-1}(\partial_{\theta}\tilde{B}^{0}_{-1-1})\Big).
\end{split}
\end{equation}
Inserting the NLS equations \eqref{NLS11}-\eqref{NLS12} satisfied by $\tilde{A}^{0}_{11}$ and $\tilde{B}^{0}_{1-1}$ into \eqref{A02} and recalling that $\tilde{A}^{0}_{-11}=\overline{\tilde{A}^{0}_{11}}$ and $\tilde{B}^{0}_{-1-1}=\overline{\tilde{B}^{0}_{1-1}}$, we find that \eqref{A02} becomes
\begin{align*}
\partial_{\theta}\tilde{A}^{0}_{0j}&=-\frac{i\partial_{k}\alpha_{11}^{j}(0,k_{0},k_{0})}
{c_{g}-j\omega'(0)}\omega''(k_{0})
\partial_{X_{+}}\Big[(\partial_{X_{+}}\tilde{A}^{0}_{11})\tilde{A}^{0}_{-11}
-\tilde{A}^{0}_{11}(\partial_{X_{+}}\tilde{A}^{0}_{-11})\Big],\\
\partial_{\theta}\tilde{B}^{0}_{0j}&=\frac{i\partial_{k}\alpha_{-1-1}^{j}(0,k_{0},k_{0})}
{-c_{g}-j\omega'(0)}\omega''(k_{0})
\partial_{X_{-}}\Big[(\partial_{X_{-}}\tilde{B}^{0}_{1-1})\tilde{B}^{0}_{-1-1}
-\tilde{B}^{0}_{1-1}(\partial_{X_{-}}\tilde{B}^{0}_{-1-1})\Big].
\end{align*}
Inserting this expression into \eqref{A01}, we see that we can choose $\tilde{A}^{1}_{0j}$ and $\tilde{B}^{1}_{0j}$ to be
\begin{equation}
\begin{split}\label{A0111}
\tilde{A}^{1}_{0j}=\frac{1}{c_{g}-j\omega'(0)}
&\bigg\{2\partial_{k}\alpha_{11}^{j}(0,k_{0},-k_{0})
\Big[\tilde{A}^{0}_{11}(\tilde{A}^{1}_{-11}+\mathcal{I}^{0}_{-11})
+\tilde{A}^{0}_{-11}(\tilde{A}^{1}_{11}+\mathcal{I}^{0}_{11})\Big]\\
&+\frac{i\partial_{k}\alpha_{11}^{j}(0,k_{0},k_{0})}
{c_{g}-j\omega'(0)}\omega''(k_{0})
\Big[(\partial_{X_{+}}\tilde{A}^{0}_{11})\tilde{A}^{0}_{-11}
-\tilde{A}^{0}_{11}(\partial_{X_{+}}\tilde{A}^{0}_{-11})\Big]\bigg\},\\
\tilde{B}^{1}_{0j}=\frac{-1}{-c_{g}-j\omega'(0)}
&\bigg\{2\partial_{k}\alpha_{-1-1}^{j}(0,k_{0},-k_{0})
\Big[\tilde{B}^{0}_{1-1}(\tilde{B}^{1}_{-1-1}+\mathcal{J}^{0}_{-1-1})
+\tilde{B}^{0}_{-1-1}(\tilde{B}^{1}_{1-1}+\mathcal{J}^{0}_{1-1})\Big]\\
&+\frac{i\partial_{k}\alpha_{-1-1}^{j}(0,k_{0},k_{0})}
{-c_{g}-j\omega'(0)}\omega''(k_{0})
\Big[(\partial_{X_{-}}\tilde{B}^{0}_{1-1})\tilde{B}^{0}_{-1-1}
-\tilde{B}^{0}_{1-1}(\partial_{X_{-}}\tilde{B}^{0}_{-1-1})\Big]\bigg\}.
\end{split}
\end{equation}
In light of \eqref{Modify1}-\eqref{Modify2} and the NLS equations \eqref{NLS11}-\eqref{NLS12}, $\tilde{A}^{1}_{0j}$ and $\tilde{B}^{1}_{0j}$ can be determined by the above equations when $\tilde{A}^{1}_{\pm11}$ and $\tilde{B}^{1}_{\pm1-1}$ satisfy inhomogeneous linear NLS equations \eqref{linearNLS} which will be fixed at the order of $\mathcal{O}(\epsilon^{4}E^{1}F^{0})$ and $\mathcal{O}(\epsilon^{4}E^{0}F^{1})$.

For $\mathcal{O}(\epsilon^{4}E^{1}F^{0})$ and $\mathcal{O}(\epsilon^{4}E^{0}F^{1})$, we have
\begin{align*}
\partial_{\theta}\tilde{A}^{1}_{11}+\partial_{T}\tilde{\mathcal{I}}^{1}_{11}
+\partial_{\theta}\tilde{\mathcal{I}}^{0}_{11}
=&-\frac{i}{2}\omega''(k_{0})\partial^{2}_{X_{+}}\tilde{A}^{1}_{11}
-\frac{1}{3!}\omega'''(k_{0})\partial^{3}_{X_{+}}\tilde{A}^{0}_{11}\\
&+\omega'(k_{0})\partial_{X}\tilde{\mathcal{I}}^{1}_{11}
-\frac{i}{2}\omega''(k_{0})\partial^{2}_{X}\tilde{\mathcal{I}}^{0}_{11}\\
&+2i\sum_{j}\alpha_{1j}^{1}(k_{0},k_{0},0)\Big[\tilde{A}^{1}_{11}
(\tilde{A}^{0}_{0j}+\tilde{B}^{0}_{0j})+\tilde{A}^{0}_{11}(\tilde{A}^{1}_{0j}+\tilde{B}^{1}_{0j})\Big]\\
&+2i\sum_{j}\alpha_{1j}^{1}(k_{0},-k_{0},2k_{0})
\Big[\tilde{A}^{1}_{-11}\tilde{A}^{0}_{2j}
+\tilde{A}^{0}_{-11}\tilde{A}^{1}_{2j}\Big]\\
&+2\sum_{j}\Big[\partial_{k}\alpha_{1j}^{1}(k_{0},k_{0},0)
\partial_{X}(\tilde{A}^{0}_{11}(\tilde{A}^{0}_{0j}+\tilde{B}^{0}_{0j}))\\
& \ \ \ \ \ \ \ \ \ \ \ \ \ +\partial_{k}\alpha_{1j}^{1}(k_{0},-k_{0},2k_{0})
\partial_{X}(\tilde{A}^{0}_{-11}\tilde{A}^{0}_{2j})\Big]\\
&+2i\sum_{j}\alpha_{-1j}^{1}(k_{0},k_{0},0)
\Big[(\tilde{B}^{0}_{1-1}\tilde{f}^{1}_{1-1j}+\tilde{B}^{0}_{-1-1}\tilde{f}^{1}_{11j}\Big]\\
&+2\sum_{j}\partial_{k}\alpha_{-1j}^{1}(k_{0},k_{0},0)
\partial_{X}\Big[(\tilde{B}^{0}_{1-1}\tilde{f}^{0}_{1-1j}+\tilde{B}^{0}_{-1-1}\tilde{f}^{0}_{11j}\Big]\\
&+3\partial_{k}\beta_{111}^{1}(k_{0},k_{0},k_{0},-k_{0})
\partial_{X}\big[\tilde{A}^{0}_{11}
(|\tilde{A}^{0}_{11}|^{2}+|\tilde{B}^{0}_{1-1}|^{2})\big],
\end{align*}
and
\begin{align*}
\partial_{\theta}\tilde{B}^{1}_{1-1}+\partial_{T}\tilde{\mathcal{J}}^{1}_{1-1}
+\partial_{\theta}\tilde{\mathcal{J}}^{0}_{1-1}
=&\frac{i}{2}\omega''(k_{0})\partial^{2}_{X_{-}}\tilde{B}^{1}_{1-1}
+\frac{1}{3!}\omega'''(k_{0})\partial^{3}_{X_{-}}\tilde{B}^{0}_{1-1}\\
&-\omega'(k_{0})\partial_{X}\tilde{\mathcal{J}}^{1}_{1-1}
+\frac{i}{2}\omega''(k_{0})\partial^{2}_{X}\tilde{\mathcal{J}}^{0}_{1-1}\\
&+2i\sum_{j}\alpha_{-1j}^{-1}(k_{0},k_{0},0)\Big[\tilde{B}^{1}_{1-1}
(\tilde{A}^{0}_{0j}+\tilde{B}^{0}_{0j})+\tilde{B}^{0}_{1-1}(\tilde{A}^{1}_{0j}+\tilde{B}^{1}_{0j})\Big]\\
&+2i\sum_{j}\alpha_{-1j}^{-1}(k_{0},-k_{0},2k_{0})
\Big[\tilde{B}^{1}_{-1-1}\tilde{B}^{0}_{2j}
+\tilde{B}^{0}_{-1-1}\tilde{B}^{1}_{2j}\Big]\\
&+2\sum_{j}\Big[\partial_{k}\alpha_{-1j}^{-1}(k_{0},k_{0},0)
\partial_{X}(\tilde{B}^{0}_{1-1}(\tilde{A}^{0}_{0j}+\tilde{B}^{0}_{0j}))\\
& \ \ \ \ \ \ \ \ \ \ \ \ \ +\partial_{k}\alpha_{-1j}^{-1}(k_{0},-k_{0},2k_{0})
\partial_{X}(\tilde{B}^{0}_{-1-1}\tilde{B}^{0}_{2j})\Big]\\
&+2i\sum_{j}\alpha_{1j}^{-1}(k_{0},k_{0},0)
\Big[(\tilde{A}^{0}_{-11}\tilde{f}^{1}_{11j}+\tilde{A}^{0}_{11}\tilde{f}^{1}_{-11j}\Big]\\
&+2\sum_{j}\partial_{k}\alpha_{1j}^{-1}(k_{0},k_{0},0)
\partial_{X}\Big[(\tilde{A}^{0}_{-11}\tilde{f}^{0}_{11j}+\tilde{A}^{0}_{11}\tilde{f}^{0}_{-11j}\Big]\\
&+3\partial_{k}\beta_{-1-1-1}^{1}(k_{0},k_{0},k_{0},-k_{0})
\partial_{X}\big[\tilde{B}^{0}_{1-1}
(|\tilde{A}^{0}_{11}|^{2}+|\tilde{B}^{0}_{1-1}|^{2})\big].
\end{align*}
Inserting \eqref{A2}-\eqref{A0} and \eqref{A21}-\eqref{A0111} to above equations, we obtain
\begin{align*}
\partial_{\theta}\tilde{A}^{1}_{11}&+\partial_{T}\tilde{\mathcal{I}}^{1}_{11}
+\partial_{\theta}\tilde{\mathcal{I}}^{0}_{11}
=-\frac{i}{2}\omega''(k_{0})\partial^{2}_{X_{+}}\tilde{A}^{1}_{11}
-\frac{1}{3!}\omega'''(k_{0})\partial^{3}_{X_{+}}\tilde{A}^{0}_{11}
+\omega'(k_{0})\partial_{X}\tilde{\mathcal{I}}^{1}_{11}\\
&-\frac{i}{2}\omega''(k_{0})\partial^{2}_{X}\tilde{\mathcal{I}}^{0}_{11}
+i\gamma_{3}(k_{0})|\tilde{A}^{0}_{11}|^{2}\tilde{A}^{1}_{11}
+i\gamma_{4}(k_{0})(\tilde{A}^{0}_{11})^{2}\tilde{A}^{1}_{-11}
+\gamma_{5}(k_{0})|\tilde{A}^{0}_{11}|^{2}\partial_{X_{+}}\tilde{A}^{0}_{11}\\
&+\gamma_{6}(k_{0})(\tilde{A}^{0}_{11})^{2}\partial_{X_{+}}\tilde{A}^{0}_{-11}
+\mathcal{Q}_{1}(\tilde{A}^{0}_{\pm11}, \ \tilde{B}^{0}_{\pm1-1}, \ \tilde{\mathcal{I}}^{0}_{\pm11}, \ \tilde{\mathcal{J}}^{0}_{\pm1-1}, \ \tilde{A}^{1}_{11}, \ \tilde{B}^{1}_{\pm1-1}),\\
\partial_{\theta}\tilde{B}^{1}_{1-1}&+\partial_{T}\tilde{\mathcal{J}}^{1}_{1-1}
+\partial_{\theta}\tilde{\mathcal{J}}^{0}_{1-1}
=\frac{i}{2}\omega''(k_{0})\partial^{2}_{X_{-}}\tilde{B}^{1}_{1-1}
+\frac{1}{3!}\omega'''(k_{0})\partial^{3}_{X_{-}}\tilde{B}^{0}_{1-1}
-\omega'(k_{0})\partial_{X}\tilde{\mathcal{J}}^{1}_{1-1}\\
&+\frac{i}{2}\omega''(k_{0})\partial^{2}_{X}\tilde{\mathcal{J}}^{0}_{1-1}
+i\gamma_{7}(k_{0})|\tilde{B}^{0}_{1-1}|^{2}\tilde{B}^{1}_{1-1}
+i\gamma_{8}(k_{0})(\tilde{B}^{0}_{1-1})^{2}\tilde{B}^{1}_{-1-1}
+\gamma_{9}(k_{0})|\tilde{B}^{0}_{1-1}|^{2}\partial_{X_{-}}\tilde{B}^{0}_{1-1}\\
&+\gamma_{10}(k_{0})(\tilde{B}^{0}_{1-1})^{2}\partial_{X_{-}}\tilde{B}^{0}_{-1-1}
+\mathcal{Q}_{2}(\tilde{A}^{0}_{\pm11}, \ \tilde{B}^{0}_{\pm1-1}, \ \tilde{\mathcal{I}}^{0}_{\pm11}, \ \tilde{\mathcal{J}}^{0}_{\pm1-1}, \ \tilde{A}^{1}_{\pm11}, \ \tilde{B}^{1}_{1-1}),
\end{align*}
where $\gamma_{m}(k_{0})$ with $m\in\{3,\cdot\cdot\cdot,10\}$ are real functions only depending on basic spatial wave number $k_{0}>0$, and $\mathcal{Q}_{n}$ with $n\in\{1, 2\}$ are functions of $\tilde{A}^{0}_{\pm11}$, $\tilde{B}^{0}_{\pm1-1}$, $\tilde{\mathcal{I}}^{0}_{\pm11}$,  $\tilde{\mathcal{J}}^{0}_{\pm1-1}$, $\tilde{A}^{1}_{\pm11}$ and $\tilde{B}^{1}_{\pm1-1}$. Specifically, $\mathcal{Q}_{n}$ is linear polynomial for each component
\begin{align*}
\mathcal{Q}_{1}
=&\mathcal{Q}_{1}\Big(|\tilde{B}^{0}_{1-1}|^{2}\tilde{A}^{1}_{11}, \
\tilde{A}^{0}_{11}\tilde{B}^{0}_{1-1}\tilde{B}^{1}_{-1-1}, \
\tilde{A}^{0}_{11}\tilde{B}^{0}_{-1-1}\tilde{B}^{1}_{1-1}, \
|\tilde{A}^{0}_{11}|^{2}\tilde{\mathcal{I}}^{0}_{11}, \
|\tilde{B}^{0}_{1-1}|^{2}\tilde{\mathcal{I}}^{0}_{11},\\
& \ \ \ \ \ \ \ \ (\tilde{A}^{0}_{11})^{2}\tilde{\mathcal{I}}^{0}_{-11}, \
|\tilde{B}^{0}_{1-1}|^{2}\partial_{X_{+}}\tilde{A}^{0}_{11}, \
(\tilde{B}^{0}_{1-1})^{2}\partial_{X_{+}}\tilde{A}^{0}_{-11}, \
\tilde{A}^{0}_{11}\partial_{X_{-}}|\tilde{B}^{0}_{1-1}|^{2},\\
& \ \ \ \ \ \ \ \ |\tilde{B}^{0}_{1-1}|^{2}\partial_{T}\tilde{A}^{0}_{11}, \
\tilde{A}^{0}_{11}\partial_{T}|\tilde{B}^{0}_{1-1}|^{2}, \
\tilde{A}^{0}_{11}\tilde{B}^{0}_{1-1}\tilde{\mathcal{J}}^{0}_{-1-1}, \
\tilde{A}^{0}_{11}\tilde{B}^{0}_{-1-1}\tilde{\mathcal{J}}^{0}_{1-1}\Big),\\
\mathcal{Q}_{2}
=&\mathcal{Q}_{2}\Big(|\tilde{A}^{0}_{11}|^{2}\tilde{B}^{1}_{1-1}, \
\tilde{B}^{0}_{1-1}\tilde{A}^{0}_{11}\tilde{A}^{1}_{-11}, \
\tilde{B}^{0}_{1-1}\tilde{A}^{0}_{-11}\tilde{A}^{1}_{11}, \
|\tilde{B}^{0}_{1-1}|^{2}\tilde{\mathcal{J}}^{0}_{1-1}, \
|\tilde{A}^{0}_{11}|^{2}\tilde{\mathcal{J}}^{0}_{1-1},\\
& \ \ \ \ \ \ \ \ (\tilde{B}^{0}_{1-1})^{2}\tilde{\mathcal{J}}^{0}_{-1-1}, \
|\tilde{A}^{0}_{11}|^{2}\partial_{X_{-}}\tilde{B}^{0}_{1-1}, \
(\tilde{A}^{0}_{11})^{2}\partial_{X_{-}}\tilde{B}^{0}_{-1-1}, \
\tilde{B}^{0}_{1-1}\partial_{X_{+}}|\tilde{A}^{0}_{11}|^{2},\\
& \ \ \ \ \ \ \ \ |\tilde{A}^{0}_{11}|^{2}\partial_{T}\tilde{B}^{0}_{1-1}, \
\tilde{B}^{0}_{1-1}\partial_{T}|\tilde{A}^{0}_{11}|^{2}, \
\tilde{B}^{0}_{1-1}\tilde{A}^{0}_{11}\tilde{\mathcal{I}}^{0}_{-11}, \
\tilde{B}^{0}_{1-1}\tilde{A}^{0}_{-11}\tilde{\mathcal{I}}^{0}_{11}\Big).
\end{align*}
We can take the following form to make sure the above equalities valid and then obtain the linear but inhomogeneous Schr\"{o}dinger equations for $\tilde{A}^{1}_{11}$ and $\tilde{B}^{1}_{1-1}$
\begin{align}\label{linearNLS}
\partial_{\theta}\tilde{A}^{1}_{11}
=&-\frac{i}{2}\omega''(k_{0})\partial^{2}_{X_{+}}\tilde{A}^{1}_{11}
+i\gamma_{3}(k_{0})|\tilde{A}^{0}_{11}|^{2}\tilde{A}^{1}_{11}
+i\gamma_{4}(k_{0})(\tilde{A}^{0}_{11})^{2}\tilde{A}^{1}_{-11}
-\frac{1}{3!}\omega'''(k_{0})\partial^{3}_{X_{+}}\tilde{A}^{0}_{11}\nonumber\\
&+\gamma_{5}(k_{0})|\tilde{A}^{0}_{11}|^{2}\partial_{X_{+}}\tilde{A}^{0}_{11}
+\gamma_{6}(k_{0})(\tilde{A}^{0}_{11})^{2}\partial_{X_{+}}\tilde{A}^{0}_{-11},\nonumber\\
\partial_{\theta}\tilde{B}^{1}_{1-1}
=&\frac{i}{2}\omega''(k_{0})\partial^{2}_{X_{-}}\tilde{B}^{1}_{1-1}
+i\gamma_{7}(k_{0})|\tilde{B}^{0}_{1-1}|^{2}\tilde{B}^{1}_{1-1}
+i\gamma_{8}(k_{0})(\tilde{B}^{0}_{1-1})^{2}\tilde{B}^{1}_{-1-1}\\
&+\frac{1}{3!}\omega'''(k_{0})\partial^{3}_{X_{-}}\tilde{B}^{0}_{1-1}
+\gamma_{9}(k_{0})|\tilde{B}^{0}_{1-1}|^{2}\partial_{X_{-}}\tilde{B}^{0}_{1-1}
+\gamma_{10}(k_{0})(\tilde{B}^{0}_{1-1})^{2}\partial_{X_{-}}\tilde{B}^{0}_{-1-1}\nonumber,
\end{align}
where the inhomogeneous terms have been defined at prior steps in this process. In addition,
$\tilde{\mathcal{I}}^{1}_{11}$ and $\tilde{\mathcal{J}}^{1}_{1-1}$ satisfy the evolutionary equations as follows
\begin{align}\label{I111}
\partial_{T}\tilde{\mathcal{I}}^{1}_{11}
-\omega'(k_{0})\partial_{X}\tilde{\mathcal{I}}^{1}_{11}
=&-\partial_{\theta}\tilde{\mathcal{I}}^{0}_{11}
-\frac{i}{2}\omega''(k_{0})\partial^{2}_{X}\tilde{\mathcal{I}}^{0}_{11}\nonumber\\
&+\mathcal{Q}_{1}(\tilde{A}^{0}_{\pm11}, \ \tilde{B}^{0}_{\pm1-1}, \ \tilde{\mathcal{I}}^{0}_{\pm11}, \
\tilde{\mathcal{J}}^{0}_{\pm1-1}, \ \tilde{A}^{1}_{11}, \ \tilde{B}^{1}_{\pm1-1}),\nonumber\\
\partial_{T}\tilde{\mathcal{J}}^{1}_{1-1}
+\omega'(k_{0})\partial_{X}\tilde{\mathcal{J}}^{1}_{1-1}
=&-\partial_{\theta}\tilde{\mathcal{J}}^{0}_{1-1}
+\frac{i}{2}\omega''(k_{0})\partial^{2}_{X}\tilde{\mathcal{J}}^{0}_{1-1}\\
&+\mathcal{Q}_{2}(\tilde{A}^{0}_{\pm11}, \ \tilde{B}^{0}_{\pm1-1}, \ \tilde{\mathcal{I}}^{0}_{\pm11}, \
\tilde{\mathcal{J}}^{0}_{\pm1-1}, \ \tilde{A}^{1}_{\pm11}, \ \tilde{B}^{1}_{1-1})\nonumber.
\end{align}
Recall that $\omega'(k_{0})=c_{g}$, $\tilde{\mathcal{I}}^{0}_{11}=\tilde{\mathcal{I}}_{1}$ and $\tilde{\mathcal{J}}^{0}_{1-1}=\tilde{\mathcal{J}}_{1}$ satisfy the transport equations \eqref{Modify1}-\eqref{Modify2}, and $\tilde{A}^{0}_{\pm11}$ and $\tilde{B}^{0}_{\pm1-1}$ satisfy the NLS equations \eqref{NLS11}-\eqref{NLS12}. We have
\begin{align*}
(\partial_{T}&\partial_{\theta}\tilde{\mathcal{I}}^{0}_{11}
-c_{g}\partial_{X}\partial_{\theta}\tilde{\mathcal{I}}^{0}_{11})
=i\widetilde{\nu}_{1}(k_{0})\Big[(\partial_{\theta}\tilde{A}^{0}_{11})|\tilde{B}^{0}_{1-1}|^{2}
+\tilde{A}^{0}_{11}\partial_{\theta}|\tilde{B}^{0}_{1-1}|^{2}\Big]\\
=&\widetilde{\nu}_{1}(k_{0})\bigg\{\frac{1}{2}\omega''(k_{0})
\Big[(\partial_{X}^{2}\tilde{A}^{0}_{11})|\tilde{B}^{0}_{1-1}|^{2}
+2\tilde{A}^{0}_{11}(\tilde{B}^{0}_{1-1}\partial_{X}^{2}\tilde{B}^{0}_{-1-1}
-\tilde{B}^{0}_{-1-1}\partial_{X}^{2}\tilde{B}^{0}_{1-1})\Big]\\
& \ \ \ \ \ \ \ \ \ \ \ \ -\nu_{1}(k_{0})\tilde{A}^{0}_{11}|\tilde{A}^{0}_{11}|^{2}
|\tilde{B}^{0}_{1-1}|^{2}\bigg\},\\
(\partial_{T}&\partial_{\theta}\tilde{\mathcal{J}}^{0}_{1-1}
+c_{g}\partial_{X}\partial_{\theta}\tilde{\mathcal{J}}^{0}_{1-1})
=i\widetilde{\nu}_{2}(k_{0})\Big[(\partial_{\theta}\tilde{B}^{0}_{1-1})|\tilde{A}^{0}_{11}|^{2}
+\tilde{B}^{0}_{1-1}\partial_{\theta}|\tilde{A}^{0}_{11}|^{2}\Big]\\
=&-\widetilde{\nu}_{2}(k_{0})\bigg\{\frac{1}{2}\omega''(k_{0})
\Big[(\partial_{X}^{2}\tilde{B}^{0}_{1-1})|\tilde{A}^{0}_{11}|^{2}
+2\tilde{B}^{0}_{1-1}\partial_{X}(\tilde{A}^{0}_{11}\partial_{X}\tilde{A}^{0}_{-11}
-\tilde{A}^{0}_{-11}\partial_{X}\tilde{A}^{0}_{11})\Big]\\
& \ \ \ \ \ \ \ \ \ \ \ \ +\nu_{2}(k_{0})\tilde{B}^{0}_{1-1}|\tilde{B}^{0}_{1-1}|^{2}
|\tilde{A}^{0}_{11}|^{2}\bigg\}.
\end{align*}
By Lemma \ref{L3}, we have $\partial_{\theta}\tilde{\mathcal{I}}^{0}_{11}$ and $\partial_{\theta}\tilde{\mathcal{J}}^{0}_{1-1}\in H^{s-2}$ and hence $\tilde{\mathcal{I}}^{1}_{11}$ and $\tilde{\mathcal{J}}^{1}_{1-1}\in H^{s-2}$ when $\tilde{A}^{0}_{11}$ and $\tilde{B}^{0}_{1-1} \in H^{s}$.

For the terms of $\mathcal{O}(\epsilon^{5})$, we can handle the expansion in a similar fashion and we omit the details here. In order to close the system at $\mathcal{O}(\epsilon^{5})$ we have to compute the linear inhomogeneous Schr\"{o}dinger equations for $(\tilde{A}^{3}_{11}, \tilde{B}^{3}_{1-1})$ and the transport equations for the correctors $(\tilde{\mathcal{I}}^{3}_{11}, \tilde{\mathcal{J}}^{3}_{1-1})$ at $\mathcal{O}(\epsilon^{6})$. In fact, the above process implies that for all $\ell,\ell_{1},\ell_{2}, j, n$ the amplitudes $(\tilde{A}^{n}_{\ell j}, \ \tilde{B}^{n}_{\ell j})$ with $\ell\in\{0, 2, 3, 4, 5\}$, the correctors $\tilde{f}^{n}_{\ell_{1}\ell_{2} j}$ and $(\tilde{\mathcal{I}}^{n}_{\pm1j}, \tilde{\mathcal{J}}^{n}_{\pm1j})$ are uniquely determined if the amplitudes $(\tilde{A}^{n'}_{\pm1j'}, \tilde{B}^{n'}_{\pm1j'})$ are fixed with $n'\leq n$.

Specifically, we first consider $\tilde{A}^{n}_{\ell j}$ and $\tilde{B}^{n}_{\ell j}$ with $j\in\{\pm1\}$ and $\ell\in\{2, \ 3, \ 4, \ 5\}$. If we write out the expression for the coefficients of $E^{\ell}$ and $F^{\ell}$ of $\mathcal{O}(\epsilon^{\ell+n})$, we find that \eqref{equation7} leads to equations of the form
\begin{equation}
\begin{split}\label{ABjn}
(\ell \omega_{0}-j\omega(\ell k_{0}))\tilde{A}^{n}_{\ell j}=g^{n}_{\ell j},\\
(-\ell \omega_{0}-j\omega(\ell k_{0}))\tilde{B}^{n}_{\ell j}=h^{n}_{\ell j},
\end{split}
\end{equation}
where $g^{n}_{\ell j}$ depends polynomially on $\tilde{A}^{n'}_{\ell' j'}$ with $n'\leq n$ (in particular if $n'=n$, then $\ell'<\ell$), and $\tilde{\mathcal{B}}^{n''}_{\ell'' j''}$, $\tilde{\mathcal{I}}^{n''}_{\pm1 j''}$, $\tilde{\mathcal{J}}^{n''}_{\pm1 j''}$  and $\tilde{f}^{n''}_{\ell_{1}''\ell_{1}''j''}$ with $n''<n$. Similarly, $h^{n}_{\ell j}$ depends polynomially on $\tilde{B}^{n'}_{\ell' j'}$ with $n'\leq n$ (in particular if $n'=n$, then $\ell'<\ell$), and $\tilde{\mathcal{A}}^{n''}_{\ell'' j''}$, $\tilde{\mathcal{I}}^{n''}_{\pm1 j''}$, $\tilde{\mathcal{J}}^{n''}_{\pm1 j''}$ and $\tilde{f}_{\ell_{1}''\ell_{1}''j''}$ with $n''<n$.


Since $\ell \omega_{0}\pm j\omega(\ell k_{0})\neq0$ for $j\in\{\pm1\}$ and $\ell\in\{2, 3,4,5\}$, each $\tilde{A}^{n}_{\ell j}$ and $\tilde{B}^{n}_{\ell j}$ on the left hand side of \eqref{ABjn} can be uniquely determined. 

Next we consider $\tilde{A}^{n}_{\ell j}$ and $\tilde{B}^{n}_{\ell j}$ with $j\in\{\pm1\}$, $n\in\{2,3\}$ and $\ell=0$. In this case, $E^{0}=F^{0}=1$, so the derivatives that give rise to the left hand side of \eqref{ABjn} all vanish and the lowest order contributions to the equations for $\tilde{A}^{n}_{0j}$ and $\tilde{B}^{n}_{0j}$ will be $\mathcal{O}(\epsilon^{n+3})$. Writing out these equations we obtain
\begin{equation}
\begin{split}\label{AB0n}
(c_{g}-j\omega'(0))\partial_{X^{+}}\tilde{A}^{n}_{0j}+\partial_{\theta}\tilde{A}^{n-1}_{0j}=g^{n}_{0j},\\
(-c_{g}-j\omega'(0))\partial_{X^{-}}\tilde{B}^{n}_{0j}+\partial_{\theta}\tilde{B}^{n-1}_{0j}=h^{n}_{0j},
\end{split}
\end{equation}
where $g^{n}_{0j}$ and $h^{n}_{0j}$ depend polynomially on $\tilde{A}^{n'}_{\ell' j'}$ and $\tilde{B}^{n'}_{\ell' j'}$ with $n'\leq n$ (in particular if $n'=n$, then $\ell'=\pm1$), and  $\tilde{\mathcal{I}}^{n''}_{\pm1 j''}, \ \tilde{\mathcal{J}}^{n''}_{\pm1 j''}$ and $\tilde{f}^{n''}_{\ell''_{1}\ell''_{2} j}$ with $n''<n$.

Similar to $\tilde{A}^{1}_{0j}$ and $\tilde{B}^{1}_{0j}$ in \eqref{A01}, the nonlinear terms $g^{n}_{0j}$ and $h^{n}_{0j}$ are also of the form of an $X$-derivative of an expression involving the amplitudes $\tilde{A}^{n'}_{\ell' j'}, \ \tilde{\mathcal{I}}^{n''}_{\pm1 j'}$ and $\tilde{B}^{n'}_{\ell' j'}, \ \tilde{\mathcal{J}}^{n''}_{\pm1 j'}$, respectively. The $X$-derivative arises because all the nonlinear terms in the Euler-Poisson system contain a derivative. Since $c_{g}\pm j\omega'(0)\neq0$, $\tilde{A}^{n}_{0j}$ and $\tilde{B}^{n}_{0j}$ can be fixed 
by a straightforward integration.

For the corrector $\tilde{f}^{n}_{\ell_{1}\ell_{2} j}$, by requiring
that the coefficients of the $E^{\ell_{1}}F^{\ell_{2}}$ vanish up to $\epsilon^{|\ell_{1}|+|\ell_{2}|+n}$, we obtain
\begin{equation}
\begin{split}\label{f12n}
((\ell_{1}-\ell_{2})\omega_{0}-j\omega((\ell_{1}+\ell_{2})k_{0}))\tilde{f}^{n}_{\ell_{1}\ell_{2} j}=g^{n}_{\ell_{1}\ell_{2}j},
\end{split}
\end{equation}
where $g^{n}_{\ell_{1}\ell_{2}j}$ depends polynomially on $\tilde{A}^{n'}_{\ell' j'}$ and $\tilde{B}^{n'}_{\ell' j'}$ with $n'\leq n$, $\tilde{f}^{n''}_{\ell'_{1}\ell'_{2} j'}$ with $n''\leq n$ (in particular if $n''=n$, then $|\ell'_{1}|+|\ell'_{2}|<|\ell_{1}|+|\ell_{2}|$), and $\tilde{\mathcal{I}}^{n''}_{\pm1 j'}$ and $ \tilde{\mathcal{J}}^{n''}_{\pm1 j'}$ with $n''<n$. Then $\tilde{f}^{n}_{\ell_{1}\ell_{2} j}$ can be fixed because of $(\ell_{1}-\ell_{2})\omega_{0}-j\omega((\ell_{1}+\ell_{2})k_{0})\neq0$.

It remains to consider $\tilde{A}^{n}_{11}, \tilde{B}^{n}_{1-1}$ and $\tilde{\mathcal{I}}^{n}_{\pm11}, \tilde{\mathcal{J}}^{n}_{\pm1-1}$ with $n\in\{2, 3, 4\}$. We have
\begin{equation}
\begin{split}\label{ABIJ}
&\partial_{\theta}\tilde{A}^{n}_{11}=
-\frac{i}{2}\omega''(k_{0})\partial^{2}_{X_{+}}\tilde{A}^{n}_{11}
+g^{n}_{1},\\
&\partial_{\theta}\tilde{B}^{n}_{1-1}
=\frac{i}{2}\omega''(k_{0})\partial^{2}_{X_{-}}\tilde{B}^{n}_{1-1}
+g^{n}_{2},\\
&\partial_{T}\tilde{\mathcal{I}}^{n}_{11}
-\omega'(k_{0})\partial_{X}\tilde{\mathcal{I}}^{n}_{11}
=h^{n}_{1},\\
&\partial_{T}\tilde{\mathcal{J}}^{n}_{1-1}
+\omega'(k_{0})\partial_{X}\tilde{\mathcal{J}}^{n}_{1-1}
=h^{n}_{2}.
\end{split}
\end{equation}
Here $g^{n}_{1}$ and $g^{n}_{2}$ are affine in $\tilde{A}^{n}_{11}$ and $\tilde{B}^{n}_{1-1}$ and depend polynomial on $\tilde{A}^{n'}_{\ell j}$ and $\tilde{B}^{n'}_{\ell j}$ and their derivatives, respectively, with $n'<n$. Similarly, $h^{n}_{1}$ and $h^{n}_{2}$ are affine in $\tilde{\mathcal{I}}^{n}_{11}$ and $\tilde{\mathcal{J}}^{n}_{1-1}$ and depend polynomial on $\tilde{A}^{n'}_{\ell j}, \tilde{B}^{n'}_{\ell j}, \tilde{\mathcal{I}}^{n''}_{11}, \tilde{\mathcal{J}}^{n''}_{1-1}$ and their derivatives with $n'\leq n$ and $n''<n$. 
The amplitudes $\tilde{A}^{n}_{11}$ and $\tilde{B}^{n}_{1-1}$ satisfy for $n\in\{1,2,3\}$ linear inhomogeneous Schr\"{o}dinger equations and have solutions on a time interval $T\in[0,T_{0}]$ with $T_{0}\sim\mathcal{O}(1)$, provided that 
$\tilde{A}^{0}_{11}$ and $\tilde{B}^{0}_{1-1}$ are solutions of the NLS equations \eqref{NLS11} and \eqref{NLS12} on $T\in[0,T_{0}]$, respectively. Furthermore, since $\tilde{A}^{4}_{11}$, $\tilde{B}^{4}_{1-1}$, $\tilde{\mathcal{I}}^{4}_{11}$ and $\tilde{\mathcal{J}}^{4}_{11}$ do not appear in equations for other amplitudes, we can set
\begin{equation}
\begin{split}\label{A4}
\tilde{A}^{4}_{11}=0, \ \ \ \tilde{B}^{4}_{1-1}=0, \ \ \ \tilde{\mathcal{I}}^{4}_{11}=0, \ \ \ \tilde{\mathcal{J}}^{4}_{11}=0.
\end{split}
\end{equation}
Now we have fixed all amplitudes $\tilde{A}^{n}_{\ell j}$ and $\tilde{B}^{n}_{\ell j}$ and all correctors $\tilde{\mathcal{I}}^{n}_{\pm1j}, \tilde{\mathcal{J}}^{n}_{\pm1j}$ and $\tilde{f}^{n}_{\ell_{1}\ell_{2}j}$. Concerning their regularity we note that the derivatives that appear on the RHS of these equations come either from the derivatives in the nonlinear terms of \eqref{abstract} or from the derivatives that appear when we expand the dispersion relation. However, each such derivative contributes one extra power of $\epsilon$ thanks to the scaling of the amplitudes and correctors. As a consequence, one can check that for $n\in\{1,2,3\}$, the maximum number of derivatives that can appear in these equations is $n$ if $(\ell,j)\neq(\pm1,1),(\ell,j')\neq(\pm1,-1)$ and $n+2$ if $(\ell,j)=(\pm1,1), (\ell,j')=(\pm1,-1)$ for $\tilde{A}^{n}_{\ell j}, \tilde{B}^{n}_{\ell j'}$, respectively, $2n$ for $\tilde{\mathcal{I}}^{n}_{\pm11}$ and $\tilde{\mathcal{J}}^{n}_{\pm1-1}$ and $n$ for $\tilde{f}^{n}_{\ell_{1}\ell_{2}j}$. Therefore, we have
\begin{proposition}\label{reg}
Fix $s\geq6$. Assume $\tilde{A}_{1}$, $\tilde{B}_{1}\in C([0,T_{0}],H^{s})$ are solutions of the NLS equations \eqref{NLS1}-\eqref{NLS2}. Then $\tilde{A}^{n}_{\ell j}, \tilde{B}^{n}_{\ell j}, \tilde{\mathcal{I}}^{n}_{\pm1j}, \tilde{\mathcal{J}}^{n}_{\pm1j}$ and $\tilde{f}^{n}_{\ell_{1}\ell_{2}j}$ defined through \eqref{appext} exist for all $T\in[0,T_{0}]$ and satisfy
\begin{align*}
&\tilde{A}^{n}_{\ell j}, \ \tilde{B}^{n}_{\ell j'} \in C([0,T_{0}], \ H^{s-n}), \ \ \ \ \ \ \ \ \ \ \  \text{for} \ \ (\ell,j)\neq(\pm1,1),\ (\ell,j')\neq(\pm1,-1),\\
&\tilde{A}^{n}_{\ell j}, \ \tilde{B}^{n}_{\ell j'}\in C([0,T_{0}], \ H^{s-n-2}), \ \  \ \ \ \ \ \  \text{for} \ \ (\ell,j)=(\pm1,1),\ (\ell,j')=(\pm1,-1),\\
&\tilde{\mathcal{I}}^{n}_{\pm11}, \ \tilde{\mathcal{J}}^{n}_{\pm1-1}\in C([0,T_{0}], \ H^{s-2n}),\\
&\tilde{f}^{n}_{\ell_{1}\ell_{2}j}\in C([0,T_{0}], \ H^{s-n}),
\end{align*}
where $n\in\{1,2,3\}$, and the respective Sobolev norms are uniformly bounded by the $H^{s}$ norms of $\tilde{A}_{1}$ and $\tilde{B}_{1}$.
\end{proposition}

From the form of the ansatz for $\epsilon\widetilde{\Theta}^{ext}$ and the discussion above we see that the residual is of order $\mathcal{O}(\epsilon^{6})$. One could extend the approximation $\epsilon\widetilde{\Theta}^{ext}$ by terms of even higher order, resulting in a smaller residual of order $\mathcal{O}(\epsilon^{m})$ with $m>6$.

\subsection{\textbf{From $\epsilon\widetilde{\Theta}^{ext}$ to $\epsilon\Theta$}}
By multiplying the Fourier transform of each function by a suitable cut-off function, we obtain our final approximation $\epsilon\Theta$ below. 
Specifically, we introduce the scaling operator $(S_{\epsilon}u)(x)=u(\epsilon x)$ and the translation operator $(\tau_{\alpha}u)(x)=u(x+\alpha)$ in space variable.  Then we rewrite terms like $S_{\epsilon}\tau_{c_{g}t}\tilde{A}_{\ell j}^{n}$ as $\tilde{A}_{\ell j}^{n}(\epsilon(x+c_{g}t),\epsilon^{2}t)$ and terms like $S_{\epsilon}\tilde{\mathcal{I}}^{n}_{\ell j}$ as $\tilde{\mathcal{I}}^{n}_{\ell j}(\epsilon x,\epsilon t,\epsilon^{2}t)$. Setting the characteristic function
\begin{equation*}
\begin{split}
\chi_{[-\delta,\delta]}(k)=\Big\{\begin{matrix} 1, \ |k|\leq\delta, \\ 0, \ |k|>\delta,\end{matrix}
\end{split}
\end{equation*}
we introduce the mode filter $(E_{\delta}u)(x)=\mathcal{F}^{-1}(\chi_{[-\delta, \delta]}\mathcal{F}u)(x)$. 
It is easy to show that
\begin{equation}
\begin{split}\label{1/2}
\|E_{\delta}S_{\epsilon}u-S_{\epsilon}u\|_{H^{m}}
&\leq C\|(\chi_{[-\delta,\delta]}-1)\epsilon^{-1}S_{1/\epsilon}\|_{H^{0}(m)}\\
&\leq\sup_{k\in\mathbb{R}}\bigg|\frac{(\chi_{[-\delta,\delta]}-1)(1+k^{2})^{m/2}}
{(1+|k/\epsilon|^{2})^{(m+M)/2}}\bigg|\epsilon^{-1/2}\|u\|_{H^{m+M}}\\
&\leq C\epsilon^{m+M-1/2}\|u\|_{H^{m+M}},
\end{split}
\end{equation}
for all $m, M\geq0$.

Thus, we can modify our approximation by applying the mode filter to all terms of the form $S_{\epsilon}\tau_{c_{g}t}\tilde{A}_{\ell j}^{n}$, $S_{\epsilon}\tau_{-c_{g}t}\tilde{B}_{\ell j}^{n}, \ S_{\epsilon}\tilde{\mathcal{I}}^{n}_{\ell j}, \ S_{\epsilon}\tilde{\mathcal{J}}^{n}_{\ell j}$ and $S_{\epsilon}\tilde{f}_{\ell_{1}\ell_{2}j}^{n}$ in the extended approximation $\epsilon\widetilde{\Theta}_{ext}$. In particular, we set
\begin{equation}
\begin{split}\label{SE}
&S_{\epsilon}\tau_{c_{g}t}A_{\ell j}^{n}:=E_{\delta}S_{\epsilon}\tau_{c_{g}t}\tilde{A}_{\ell j}^{n},\\
&S_{\epsilon}\tau_{-c_{g}t}B_{\ell j}^{n}:=E_{\delta}S_{\epsilon}\tau_{-c_{g}t}\tilde{B}_{\ell j}^{n},\\
&S_{\epsilon}\mathcal{I}^{n}_{\ell j}:=E_{\delta}S_{\epsilon}\tilde{\mathcal{I}}^{n}_{\ell j},\\
&S_{\epsilon}\mathcal{J}^{n}_{\ell j}:=E_{\delta}S_{\epsilon}\tilde{\mathcal{J}}^{n}_{\ell j},\\
&S_{\epsilon}f_{\ell_{1}\ell_{2}j}^{n}:
=E_{\delta}S_{\epsilon}\tilde{f}_{\ell_{1}\ell_{2}j}^{n},
\end{split}
\end{equation}
and define a new approximation $\epsilon\Theta$ with its $j^{th}$ component $\epsilon\Theta_{j}$ for $j\in\{\pm1\}$ from \eqref{appext} and \eqref{SE} as
\begin{align}\label{new}
\epsilon\Theta_{j}=\epsilon\Psi^{ext}_{j}
+\epsilon\Phi^{ext}_{j}
+\epsilon^{2}\Upsilon^{ext}_{j},
\end{align}
with
\begin{align*}
\epsilon\Psi^{ext}_{j}
:=&\sum_{|\ell|\leq5}\sum_{\lambda_{j}(\ell,n)\leq5}\epsilon^{\lambda_{j}(\ell,n)}
 A_{\ell j}^{n}(X+c_{g}T,\epsilon T)E^{\ell},\nonumber\\
\epsilon\Phi^{ext}_{j}
:=&\sum_{|\ell|\leq5}\sum_{\eta_{j}(\ell,n)\leq5}\epsilon^{\eta_{j}(\ell,n)} B_{\ell j}^{n}(X-c_{g}T,\epsilon T)F^{\ell},\nonumber\\
\epsilon^{2}\Upsilon^{ext}_{j}:=&\sum_{\ell\in\{\pm1\}}\sum_{2\leq\zeta_{j}(n)\leq5}
\epsilon^{\zeta_{j}(n)}\mathcal{I}_{\ell j}^{n}(X,T,\epsilon T)E^{\ell}+\sum_{\ell\in\{\pm1\}}\sum_{2\leq\mu_{j}(n)\leq5}
\epsilon^{\mu_{j}(n)}\mathcal{J}_{\ell j}^{n}(X,T,\epsilon T)F^{\ell}\nonumber\\
&+\sum_{\substack{2\leq|\ell_{1}|+|\ell_{2}|\leq5\\ \ell_{1}\ell_{2}\neq0}}
\sum_{2\leq\chi(\ell_{1},\ell_{2},n)\leq5}\epsilon^{\chi(\ell_{1},\ell_{2},n)}
f_{\ell_{1}\ell_{2}j}^{n}(X,T,\epsilon T)E^{\ell_{1}}F^{\ell_{2}}.
\end{align*}

Note that $A_{\ell j}^{n}$, $B_{\ell j}^{n}$, $\mathcal{I}_{\ell j}^{n}$ and $\mathcal{J}_{\ell j}^{n}$ have compact support around $\ell k_{0}$, and $f_{\ell_{1}\ell_{2}j}^{n}$ has compact support around $(\ell_{1}+\ell_{2})k_{0}$. For convenience, we extract the terms of order $\mathcal{O}(\epsilon)$ from the final approximation $\epsilon\Theta$ from \eqref{new} and denote
\begin{align*}
\epsilon\Psi_{\pm1}&:=\epsilon\psi_{\pm1}\begin{pmatrix}1\\0\end{pmatrix}:=\epsilon A_{\pm1}(\epsilon(x+c_{g}t),\epsilon^{2}t)E^{\pm1}\begin{pmatrix}1\\0\end{pmatrix},\\
\epsilon\Phi_{\pm1}&:=\epsilon\phi_{\pm1}\begin{pmatrix}0\\1\end{pmatrix}
:=\epsilon B_{\pm1}(\epsilon(x-c_{g}t),\epsilon^{2}t)F^{\pm1}\begin{pmatrix}0\\1\end{pmatrix},\\
\end{align*}
We then denote
\begin{align*}
\epsilon\Psi^{ext}&=:\epsilon\Psi_{1}+\epsilon\Psi_{-1}+\epsilon^{2}\Psi_{d},\\
\epsilon\Phi^{ext}&=:\epsilon\Phi_{1}+\epsilon\Phi_{-1}+\epsilon^{2}\Phi_{e},\\
\epsilon^{2}\Theta_{r}&=:\epsilon^{2}\Psi_{d} +\epsilon^{2}\Phi_{e}+\epsilon^{2}\Upsilon^{ext},
\end{align*}
and finally, we have
\begin{equation}
\begin{split}\label{ab}
\epsilon\Theta=\epsilon\Psi_{1}+\epsilon\Psi_{-1}+\epsilon\Phi_{1}+\epsilon\Phi_{-1}
+\epsilon^{2}\Theta_{r}
=\epsilon\begin{pmatrix}\psi_{1}+\psi_{-1}\\ \phi_{1}+\phi_{-1}\end{pmatrix}+\epsilon^{2}\begin{pmatrix}\Theta_{r1}\\ \Theta_{r-1}\end{pmatrix}.
\end{split}
\end{equation}

Then, for the residual the following estimates hold.
\begin{lemma}\label{L2}
Let $s_{N}\geq6$ and $\tilde{A}_{1}, \tilde{B}_{1}\in C([0,T_{0}], \ H^{s_{N}}(\mathbb{R},\mathbb{C}))$ be a solution of the NLS equations \eqref{NLS1} and \eqref{NLS2} with
\begin{equation*}
\begin{split}
\sup_{T\in[0,T_{0}]}\|\tilde{A}_{1}\|_{H^{s_{N}}}\leq C_{A},\\
\sup_{T\in[0,T_{0}]}\|\tilde{B}_{1}\|_{H^{s_{N}}}\leq C_{B}.
\end{split}
\end{equation*}
Assume that the $A_{\ell j}^{n}$, $B_{\ell j}^{n}$, $\mathcal{I}^{n}_{\pm1j}$, $\mathcal{J}^{n}_{\pm1j}$ and $f_{\ell_{1}\ell_{2}j}^{n}$ satisfy \eqref{new} and $\tilde{A}_{\ell j}^{n}$, $\tilde{B}_{\ell j}^{n}$, $\widetilde{\mathcal{I}}^{n}_{\pm1j}$, $\widetilde{\mathcal{J}}^{n}_{\pm1j}$ and $\tilde{f}_{\ell_{1}\ell_{2}j}^{n}$ satisfy \eqref{A2}-\eqref{Modify2} and \eqref{Modify22}-\eqref{A4}. Then for all $s\geq0$ there exist $C_{Res}, C_{\Theta},\epsilon_{0}>0$ depending on $C_{A}$ and $C_{B}$ such that for all $\epsilon\in(0,\epsilon_{0})$ the corresponding approximation $\epsilon\Theta$ satisfies
\begin{equation}
\begin{split}\label{Aesti-1}
\sup_{t\in[0,T_{0}/\epsilon^{2}]}\|Res(\epsilon\Theta)\|_{H^{s}}\leq C_{Res} \epsilon^{11/2},
\end{split}
\end{equation}
\begin{equation}
\begin{split}\label{Aesti-2}
\sup_{t\in[0,T_{0}/\epsilon^{2}]}\|\epsilon\Theta-(\epsilon\widetilde{\Psi}_{\pm1}
+\epsilon\widetilde{\Phi}_{\pm1})\|_{H^{s_{N}}}\leq C_{\Theta} \epsilon^{3/2},
\end{split}
\end{equation}
\begin{equation}
\begin{split}\label{Aesti-3}
\sup_{t\in[0,T_{0}/\epsilon^{2}]}(\|\widehat{\Psi}_{\pm1}\|_{L^{1}(s+1)}+\|\widehat{\Phi}_{\pm1}\|_{L^{1}(s+1)}
+\|\widehat{\Theta}_{r}\|_{L^{1}(s+1)})\leq C_{\Theta}.
\end{split}
\end{equation}
\end{lemma}
\begin{proof}
As explained above, the previous extended approximation $\epsilon\widetilde{\Theta}^{ext}$ has been constructed in a way such that we have $Res(\epsilon\widetilde{\Theta}^{ext})=\mathcal{O}(\epsilon^{6})$ and $\epsilon\widetilde{\Theta}^{ext}-(\epsilon\widetilde{\Psi}_{\pm1}
+\epsilon\widetilde{\Phi}_{\pm1})=\mathcal{O}(\epsilon^{2})$ on the time interval $[0,T_{0}/\epsilon^{2}]$. However, because of the way the $L^{2}$-norm scales, i.e., $\|A\|_{L^{2}}=\epsilon^{1/2}\|S_{\epsilon}A\|_{L^{2}}$ if $A\in L^{2}$ and $(S_{\epsilon}A)(x)=A(\epsilon x)$, we lose a factor $\epsilon^{1/2}$ in \eqref{Aesti-1} and \eqref{Aesti-2}.

Due to our cut-off procedure our final approximation $\epsilon\Theta$ can be written as
\begin{equation*}
\begin{split}
\epsilon\Theta=\sum_{\ell=-5}^{5}u_{\ell} \ \text{with} \ \text{supp}\widehat{u}_{\ell}\subset[\ell k_{0}-\delta,\ell k_{0}+\delta].
\end{split}
\end{equation*}
Hence, there exists a $C=C(k_{0})>0$ such that $\|\Theta\|_{H^{s}}\leq C\|\Theta\|_{H^{0}}$ and $\|\widehat{\Theta}\|_{L^{1}(s)}\leq C\|\widehat{\Theta}\|_{L^{1}}$ for all $s\geq0$.

Now, applying estimate \eqref{1/2} to $u=A_{\ell j}^{n},  B_{\ell j}^{n},  \mathcal{I}^{n}_{\pm1j}, \mathcal{J}^{n}_{\pm1j},  f_{\ell_{1}\ell_{2}j}^{n}$ with $m=0$, where $M$ is determined by the maximal regularity of $u$ (see Proposition \ref{reg}), we obtain
\eqref{Aesti-1}-\eqref{Aesti-2} by construction of $\epsilon\Theta$ if we have $s_{N}\geq6$.

In contrast to the $L^{2}$-norm, we have $\|u\|_{C_{b}^{0}}=\|S_{\epsilon}u\|_{C_{b}^{0}}$ and since $\mathcal{F}(S_{\epsilon}u)=\epsilon^{-1}S_{1/\epsilon}(\mathcal{F}u)$ we also have $\|\widehat{u}\|_{L^{1}}=\|\epsilon^{-1}S_{1/\epsilon}\widehat{u}\|_{L^{1}}$. Consequently, \eqref{Aesti-3} follows by construction of $\Theta$.
\end{proof}

We remark that since the Fourier transform of the functions ${\Theta}$ is strongly concentrated around the wave number $\ell k_{0}$ with $\ell\in\{0, \pm1, \pm2, \pm3, \pm4, \pm5\}$, the approximation changes slightly by the second modification if ${A}_{\ell j}^{n}$, ${B}_{\ell j}^{n}$, ${\mathcal{I}}^{n}_{\pm1j}$, ${\mathcal{J}}^{n}_{\pm1j}$ and ${f}_{\ell_{1}\ell_{2}j}^{n}$ are sufficiently regular. This operation will give us a simpler control of the error and make the approximation an analytic function.

\begin{remark}\label{R5}
The bound \eqref{Aesti-3} will be used for instance to estimate
\begin{equation*}
\begin{split}
\|\Theta R\|_{H^{s}}\leq C\|\Theta\|_{C_{b}^{s}}\|R\|_{H^{s}}\leq C\|\widehat{\Theta}\|_{L^{1}(s)}\|R\|_{H^{s}}
\end{split}
\end{equation*}
without loss of powers in $\epsilon$.
\end{remark}
Moreover, by an analogous argument as in the proof of Lemma 3.3 in \cite{D} we obtain the fact that $\partial_{t}\Psi_{\pm1}$ and $\partial_{t}\Phi_{\pm1}$ can be approximated by $\Omega\Psi_{\pm1}$ and $-\Omega\Psi_{\pm1}$ respectively. More precisely, we get
\begin{lemma}\label{LO}
For all $s>0$ there exist constants $C_{A}>0$ and $C_{B}>0$ depending on $\|\tilde{A}_{1}\|_{H^{s}}$ and $\|\tilde{B}_{1}\|_{H^{s}}$ respectively and $k_{0}$ such that
\begin{align}\label{l1s}
\|\partial_{t}\widehat{\psi}_{\pm1}-i\omega\widehat{\psi}_{\pm1}\|_{L^{1}(s)}\leq C_{A}\epsilon^{2},\\
\|\partial_{t}\widehat{\phi}_{\pm1}+i\omega\widehat{\phi}_{\pm1}\|_{L^{1}(s)}\leq C_{B}\epsilon^{2}.
\end{align}
\end{lemma}

\section{\textbf{Evolutionary equation for the error $R$ and restatement of Theorem \ref{Thm1}}}
To prove Theorem \ref{Thm1}, we need to show uniform estimate for the error $R$ between the real solutions and the approximation to justify the bidirectional NLS approximation in the desired long time intervals. Below we first deduce the evolutionary equation for the error $R$, and then restate the main Theorem \ref{Thm1} in terms of $R$ in Theorem \ref{Thm2} in this Section.
\subsection{\textbf{The error equation}}
To justify the NLS approximation, we need to show that the error
\begin{equation}
\begin{split}\label{err}
\epsilon^{\beta}R:=U-\epsilon\Theta
\end{split}
\end{equation}
is of order $\mathcal{O}(\epsilon^{\beta})$ for all $t\in[0,T_{0}/\epsilon^{2}]$ and some $\beta>3/2$. This means that $R$ has to be of order $\mathcal{O}(1)$ for all $t\in[0,T_{0}/\epsilon^{2}]$. Recall the form of the approximation solution \eqref{ab} and we denote
\begin{equation}
\begin{split}\label{pc}
\psi_{c}:=\psi_{1}+\psi_{-1}, \ \phi_{c}:=\phi_{1}+\phi_{-1}, \
\varphi_{s_{1}}:=\Theta_{r1}+\Theta_{r-1}, \ \varphi_{s_{2}}:=\Theta_{r1}-\Theta_{r-1}.
\end{split}
\end{equation}
Inserting \eqref{err} into \eqref{equation7}, we obtain the evolutionary equation for the error $R$,
\begin{align}\label{Rj}
\partial_{t}R_{j}=&j\Omega R_{j}+\frac{\epsilon}{2}\partial_{x}\bigg[(\psi_{c}+\phi_{c})q(R_{1}-R_{-1})\bigg]
+\frac{\epsilon}{2}\partial_{x}\bigg[q(\psi_{c}-\phi_{c})(R_{1}+R_{-1})\bigg]\nonumber\\
&+\frac{j\epsilon}{2q}\partial_{x}\bigg[q(\psi_{c}-\phi_{c})q(R_{1}-R_{-1})\bigg]
-\frac{j\epsilon}{2q}\partial_{x}\bigg[(\psi_{c}+\phi_{c})(R_{1}+R_{-1})\bigg]\nonumber\\
&-\frac{j\epsilon}{2q}\frac{\partial_{x}}{\langle\partial_{x}\rangle^{2}}
\bigg[\frac{1}{\langle\partial_{x}\rangle^{2}}(\psi_{c}+\phi_{c})
\frac{1}{\langle\partial_{x}\rangle^{2}}(R_{1}+R_{-1})\bigg]\nonumber\\
&+\frac{\epsilon^{2}}{2}\partial_{x}\bigg[\varphi_{s_{1}}q(R_{1}-R_{-1})\bigg]
+\frac{\epsilon^{2}}{2}\partial_{x}\bigg[q\varphi_{s_{2}}(R_{1}+R_{-1})\bigg]\nonumber\\
&+\frac{j\epsilon^{2}}{2q}\partial_{x}
\bigg[q\varphi_{s_{2}}(R_{1}-R_{-1})\bigg]
-\frac{j\epsilon^{2}}{2q}
\bigg[\varphi_{s_{1}}(R_{1}+R_{-1})\bigg]\nonumber\\
&-\frac{j\epsilon^{2}}{2q}\frac{\partial_{x}}{\langle\partial_{x}\rangle^{2}}
\bigg[\frac{1}{\langle\partial_{x}\rangle^{2}}\varphi_{s_{1}}
\frac{1}{\langle\partial_{x}\rangle^{2}}(R_{1}+R_{-1})\bigg]\nonumber\\
&+\frac{\epsilon^{\beta}}{2}\partial_{x}\bigg[(R_{1}+R_{-1})q(R_{1}-R_{-1})\bigg]
+\frac{j\epsilon^{\beta}}{4q}\bigg[q(R_{1}-R_{-1})\bigg]^{2}\nonumber\\
&-\frac{j\epsilon^{\beta}}{4q}\partial_{x}\bigg[(R_{1}+R_{-1})\bigg]^{2}
-\frac{j\epsilon^{\beta}}{4q}\frac{\partial_{x}}{\langle\partial_{x}\rangle^{2}}
\bigg[\frac{1}{\langle\partial_{x}\rangle^{2}}(R_{1}+R_{-1})\bigg]^{2}\nonumber\\
&+\frac{j\epsilon^{2}}{2q}\partial_{x}\mathcal{H}_{j}
\big(R_{1}+R_{-1})
+\epsilon^{-\beta}Res_{j}(\epsilon\Theta),
\end{align}
where $j\in\{\pm1\}$ and $\beta>2$.

In the above error equation \eqref{Rj}, the terms of order $\mathcal{O}(\epsilon^{2})$ and $\mathcal{O}(\epsilon^{\beta})$ can be controlled over the relevant time interval. Also, the residual term $\epsilon^{-\beta}Res(\epsilon\Psi)$ can be made of $\mathcal{O}(\epsilon^{2})$ by Lemma \ref{L2}. However, the terms of order $\mathcal{O}(\epsilon)$ can perturb the linear evolution in such a way that the solutions begin to grow on time scales $\mathcal{O}(\epsilon^{-1})$ and hence we would lose all control over the size of $R$ on the desired time scale $\mathcal{O}(\epsilon^{-2})$.
To show that the error remains small over the desired time intervals $(t\sim\mathcal{O}(1/\epsilon^{2}))$, we need to eliminate all $\mathcal{O}(\epsilon)$ terms from \eqref{Rj} via a normal-form transformation.

After careful calculations, we find that the kernel function of the normal-form transformation can be written as a quotient whose denominator is
\begin{equation*}
\begin{split}
j\omega(k)\mp\omega(k-\ell)-p\omega(\ell).
\end{split}
\end{equation*}
Such a normal form transformation is well defined when the denominator remains away from zero, or equivalently when a non-resonance condition is satisfied
\begin{equation}
\begin{split}\label{equation11}
|j\omega(k)\mp\omega(k_{0})-p\omega(k-k_{0})|>0
\end{split}
\end{equation}
for $j,p\in\{\pm1\}$ and for all $k\in \mathbb{R}$ uniformly. It is easy to see that $\omega(k)=k\widehat{q}(k)$ in \eqref{equation3} for the ion Euler-Poisson system does not satisfy \eqref{equation11}. For example, there is a resonance at the wave number $k=0$ (whenever $p=\pm1$). Luckily, $k=0$ is a trivial resonance which will not cause essential difficulty for the definition of the normal-form transformation because the nonlinear term of the Euler-Poisson system also vanishes linearly at $k=0$. However, there is always another nontrivial resonance $k=k_{0}$ (also for $j=\pm1$). Therefore, we cannot apply the normal-form method directly and an improved method has to be applied (refer to \cite{S,S98}). The method is to construct a modified normal-form transformation by introducing a suitable rescaling of the error function $R$ dependent on the wave number.

More precisely, we scale the error $R$ to reflect the fact that the nonlinearity vanishes at $k=0$. For some $\delta>0$ sufficiently small, but independent of $0<\epsilon\ll1$, we define a weight function $\vartheta$ by its Fourier transform
\begin{equation*}
\begin{split}
\widehat{\vartheta}(k)=\Big\{\begin{matrix} 1\ \ \ \ \ \ \ \ \ \ \ \ \ \ \ \ \ \ \ \ \ \ \ \ \ \ \ \ \text{for} \ \ |k|>\delta, \\ \epsilon+(1-\epsilon)| k|/\delta \ \ \ \ \ \ \ \ \ \ \text{for} \ \ |k|\leq\delta.\end{matrix}
\end{split}
\end{equation*}
Rewrite the solution $U$ of \eqref{equation7} as a sum of the approximation $\epsilon\Theta$ and the scaled error $\vartheta R$:
\begin{equation*}
\begin{split}
U=\epsilon\Theta+\epsilon^{\frac{5}{2}}\vartheta R.
\end{split}
\end{equation*}
As mentioned above, whether $k$ is close to zero or not will directly influence the size of the nonlinear term in Fourier space. Therefore, we define two projection operators $P^{0}$ and $P^{1}$ by the Fourier multiplier to separate the behavior in these two regions:
\begin{equation*}
\begin{split}
\widehat{P}^{0}(k)=\chi_{\mid k\mid\leq\delta}(k)\ \ \ \text{and} \ \ \ \widehat{P}^{1}(k)=\mathbf{1}-\widehat{P}^{0}(k),
\end{split}
\end{equation*}
for a $\delta>0$ sufficiently small, but independent of $0<\epsilon\ll1$. We also write $R=R^{0}+R^{1}$ with $R^{j}=P^{j}R$, for $j=0,1$.

Applying these two projection operators $P^{0}$ and $P^{1}$ to the error equation \eqref{Rj}, we obtain two evolutionary equations for $R^{0}$ and $R^{1}$,
\begin{align}\label{Rj0}
\partial_{t}R_{j}^{0}=&j\Omega R_{j}^{0}
+P^{0}\frac{\epsilon}{2\vartheta}\partial_{x}
\bigg[(\psi_{c}+\phi_{c})q\vartheta(R_{1}^{1}-R_{-1}^{1})\bigg]
+P^{0}\frac{\epsilon}{2\vartheta}\partial_{x}
\bigg[q(\psi_{c}-\phi_{c})\vartheta(R_{1}^{1}+R_{-1}^{1})\bigg]\nonumber\\
&+P^{0}\frac{j\epsilon}{2q\vartheta}\partial_{x}
\bigg[q(\psi_{c}-\phi_{c})q\vartheta(R_{1}^{1}-R_{-1}^{1})\bigg]
-P^{0}\frac{j\epsilon}{2q\vartheta}\partial_{x}
\bigg[(\psi_{c}+\phi_{c})\vartheta(R_{1}^{1}+R_{-1}^{1})\bigg]\nonumber\\
&-P^{0}\frac{j\epsilon}
{2q\vartheta}\frac{\partial_{x}}{\langle\partial_{x}\rangle^{2}}
\bigg[\frac{1}{\langle\partial_{x}\rangle^{2}}(\psi_{c}+\phi_{c})
\frac{1}{\langle\partial_{x}\rangle^{2}}\vartheta(R_{1}^{1}+ R_{-1}^{1})\bigg]\nonumber\\
&+\epsilon^{2}P^{0}\frac{\partial_{x}}{2\vartheta}\mathcal{F}_{j}^{1}
+P^{0}\frac{\epsilon^{-\frac{5}{2}}}{\vartheta}Res_{j}(\epsilon\Theta)\nonumber\\
=&:j\Omega R_{j}^{0}
+ P^{0}\frac{\epsilon}{2\vartheta}\partial_{x}
\sum_{p=\pm 1}\big[N^{+}_{j,p}(\psi_{c},\vartheta R_{p}^{1})
+ N^{-}_{j,p}(\phi_{c},\vartheta R_{p}^{1})\big]\nonumber\\
&+\epsilon^{2}P^{0}\frac{\partial_{x}}{2\vartheta}\mathcal{F}_{j}^{1}
+P^{0}\frac{\epsilon^{-\frac{5}{2}}}{\vartheta}Res_{j}(\epsilon\Theta),
\end{align}
and
\begin{align}\label{Rj1}
\partial_{t}R_{j}^{1}=&j\Omega R_{j}^{1}
+P^{1}\frac{\epsilon}{2\vartheta}\partial_{x}
\bigg[(\psi_{c}+\phi_{c})q\vartheta(R_{1}^{0}-R_{-1}^{0})\bigg]
+P^{1}\frac{\epsilon}{2\vartheta} \partial_{x}
\bigg[q(\psi_{c}-\phi_{c})\vartheta(R_{1}^{0}+R_{-1}^{0})\bigg]\nonumber\\
&+P^{1}\frac{j\epsilon}{2q\vartheta}\partial_{x}
\bigg[q(\psi_{c}-\phi_{c})q\vartheta(R_{1}^{0}-R_{-1}^{0})\bigg]
-P^{1}\frac{j\epsilon}{2q\vartheta}\partial_{x}
\bigg[(\psi_{c}+\phi_{c})\vartheta(R_{1}^{0}+R_{-1}^{0})\bigg]\nonumber\\
&-P^{1}\frac{j\epsilon}{2q\vartheta}\frac{\partial_{x}}{\langle\partial_{x}\rangle^{2}}
\bigg[\frac{1}{\langle\partial_{x}\rangle^{2}}(\psi_{c}+\phi_{c})
\frac{1}{\langle\partial_{x}\rangle^{2}}\vartheta(R_{1}^{0}+R_{-1}^{0})\bigg]\nonumber\\
&+P^{1}\frac{\epsilon}{2\vartheta}\partial_{x}
\bigg[(\psi_{c}+\phi_{c})q\vartheta(R_{1}^{1}-R_{-1}^{1})\bigg]
+P^{1}\frac{\epsilon}{2\vartheta}\partial_{x}
\bigg[q(\psi_{c}-\phi_{c})\vartheta(R_{1}^{1}+R_{-1}^{1})\bigg]\nonumber\\
&+ P^{1}\frac{j\epsilon}{2q\vartheta}\partial_{x}
\bigg[q(\psi_{c}-\phi_{c})q\vartheta(R_{1}^{1}-R_{-1}^{1})\bigg]
-P^{1}\frac{j\epsilon}{2q\vartheta}\partial_{x}
\bigg[(\psi_{c}+\phi_{c})\vartheta(R_{1}^{1}+R_{-1}^{1})\bigg]\nonumber\\
&-P^{1}\frac{j\epsilon}{2q\vartheta}\frac{\partial_{x}}{\langle\partial_{x}\rangle^{2}}
\bigg[\frac{1}{\langle\partial_{x}\rangle^{2}}(\psi_{c}+\phi_{c})
\frac{1}{\langle\partial_{x}\rangle^{2}}\vartheta(R_{1}^{1}+R_{-1}^{1})\bigg]\nonumber\\
&+\epsilon^{2}P^{1}\frac{\partial_{x}}{2\vartheta}\mathcal{F}_{j}^{1}
+\epsilon^{-\frac{5}{2}}\frac{P^{1}}{\vartheta}Res_{j}(\epsilon\Theta)\nonumber\\
=&:j\Omega R_{j}^{1}
+ P^{1}\frac{\epsilon}{2\vartheta}\partial_{x}
\sum_{p=\pm 1}\big[N^{+}_{j,p}(\psi_{c},\vartheta R_{p}^{0})
+ N^{-}_{j,p}(\phi_{c},\vartheta R_{p}^{0})+N^{+}_{j,p}(\psi_{c},\vartheta R_{p}^{1})\nonumber\\
&+ N^{-}_{j,p}(\phi_{c},\vartheta R_{p}^{1})\big]
+\epsilon^{2}P^{1}\frac{\partial_{x}}{2\vartheta}\mathcal{F}_{j}^{1}
+\epsilon^{-\frac{5}{2}}\frac{P^{1}}{\vartheta}Res_{j}(\epsilon\Theta),
\end{align}
with $j=\pm1$. The quasilinear terms of order $\mathcal{O}(\epsilon)$ satisfy
\begin{equation}
\begin{split}\label{alpha}
\widehat{N}^{+}_{j,p}(\psi_{c},\vartheta R_{p}^{0,1})(k)
=&\int\zeta^{+}_{j,p}(k,k-m,m)\widehat{\psi}_{c}(k-m)\widehat{\vartheta}(m)\widehat{R}_{p}^{0,1}(m)dm,\\
\widehat{N}^{-}_{j,p}(\phi_{c},\vartheta R_{p}^{0,1})(k)
=&\int\zeta^{-}_{j,p}(k,k-m,m)\widehat{\phi}_{c}(k-m)\widehat{\vartheta}(m)\widehat{R}_{p}^{0,1}(m)dm,\\
\zeta^{\pm}_{j,p}(k,k-m,m)=&p\hat{q}(m)\pm\hat{q}(k-m)\pm\frac{jp}{\hat{q}(k)}\hat{q}(k-m)\hat{q}(m)\\
&-\frac{j}{\hat{q}(k)}-\frac{j}{\hat{q}(k)}\frac{1}{\langle k\rangle^{2}}
\frac{1}{\langle k-m\rangle^{2}}\frac{1}{\langle m\rangle^{2}}.
\end{split}
\end{equation}
In addition,
\begin{align}\label{F1}
\mathcal{F}_{j}^{1}=&
\varphi_{s_{1}}q\vartheta(R_{1}^{0}-R_{-1}^{0}+R_{1}^{1}-R_{-1}^{1})
+q\varphi_{s_{2}}
\vartheta(R_{1}^{0}+R_{-1}^{0}+R_{1}^{1}+R_{-1}^{1})\nonumber\\
&+\frac{j}{q}
\big(q\varphi_{s_{2}}q\vartheta(R_{1}^{0}-R_{-1}^{0}+R_{1}^{1}-R_{-1}^{1})\big)
-\frac{j}{q}
\big(\varphi_{s_{1}}\vartheta(R_{1}^{0}+R_{-1}^{0}+R_{1}^{1}+R_{-1}^{1})\big)
\nonumber\\
&-\frac{j}{q}\frac{1}{\langle\partial_{x}\rangle^{2}}
(\frac{1}{\langle\partial_{x}\rangle^{2}}\varphi_{s_{1}})
(\frac{1}{\langle\partial_{x}\rangle^{2}}\vartheta(R_{1}^{0}+ R_{-1}^{0}+R_{1}^{1}+R_{-1}^{1}))\nonumber\\
&+\epsilon^{\frac{1}{2}}
\vartheta(R_{1}^{0}+R_{-1}^{0}+R_{1}^{1}+R_{-1}^{1}))
q\vartheta(R_{1}^{0}-R_{-1}^{0}+R_{1}^{1}-R_{-1}^{1})\\
&+\frac{j\epsilon^{\frac{1}{2}}}{2 q}
\big(q\vartheta(R_{1}^{0}-R_{-1}^{0}+R_{1}^{1}-R_{-1}^{1})\big)^{2}
-\frac{j\epsilon^{\frac{1}{2}}}{2q}
\big(\vartheta(R_{1}^{0}+R_{-1}^{0}+R_{1}^{1}+R_{-1}^{1})\big)^{2}\nonumber\\
&-\frac{j\epsilon^{\frac{1}{2}}}{2q}\frac{1}{\langle\partial_{x}\rangle^{2}}
\big(\frac{1}{\langle\partial_{x}\rangle^{2}}\vartheta(R_{1}^{0}+ R_{-1}^{0}+R_{1}^{1}+R_{-1}^{1})\big)^{2}\nonumber\\
&+\frac{j}{q}
\mathcal{H} (\vartheta(R_{1}^{0}+R_{-1}^{0}+R_{1}^{1}+R_{-1}^{1}))\nonumber.
\end{align}
Note that we need to obtain a uniform estimate for the error $R$ on a time scale $\mathcal{O}(\epsilon^{-2})$ and $\vartheta^{-1}$ is of order $\mathcal{O}(\epsilon^{-1})$ when $|k|<\delta$. Then terms $\epsilon^{2}P^{j}\frac{\partial_{x}}{2\vartheta}\mathcal{F}^{1}
$ and $\epsilon^{-\frac{5}{2}}\frac{P^{j}}{\vartheta}Res_{j}(\epsilon\Theta)$ with $j=0, 1$ on the RHS of \eqref{Rj0} and \eqref{Rj1} are at least of order $\mathcal{O}(\epsilon^{2})$ with the aid of Lemma \ref{L2} and the fact that $\left|\epsilon^{2}\frac{k}{\hat{\vartheta}(k)}\right|\leq C\epsilon^{2}$ for $|k|<\delta$. In addition, the term $\Omega R$ can be explicitly computed and causes no growth in $R$. Hence we need only transform
the remaining linear terms of $\mathcal{O}(\epsilon)$ to the terms of order $\mathcal{O}(\epsilon^{2})$ by normal-form transformations. All the linear terms can be classified into the following two categories:
\begin{description}
  \item[low frequency terms]
  \begin{equation*}
  \begin{split}
  &P^{0}\displaystyle\frac{\epsilon}{2\vartheta}\partial_{x}
\sum_{p=\pm 1}\big[N^{+}_{j,p}(\psi_{c},\vartheta R_{p}^{1})
+ N^{-}_{j,p}(\phi_{c},\vartheta R_{p}^{1})\big] \text{\ \ \ from \eqref{Rj0} and} \\ &P^{1}\displaystyle\frac{\epsilon}{2\vartheta}\partial_{x}
\sum_{p=\pm 1}\big[N^{+}_{j,p}(\psi_{c},\vartheta R_{p}^{0})
+ N^{-}_{j,p}(\phi_{c},\vartheta R_{p}^{0})\big] \text{ \ \ \ from \eqref{Rj1}},
\end{split}
\end{equation*}
  \item[high frequency terms]
  \begin{equation*}
  \begin{split}
  P^{1}\displaystyle\frac{\epsilon}{2\vartheta}\partial_{x}
\sum_{p=\pm 1}\big[N^{+}_{j,p}(\psi_{c},\vartheta R_{p}^{1})
+ N^{-}_{j,p}(\phi_{c},\vartheta R_{p}^{1})\big] \ \text{\ \ \  from \eqref{Rj1}.}\ \ \ \
\end{split}
\end{equation*}
\end{description}

\subsection{\textbf{Restatement of Theorem \ref{Thm1}}}
In terms of $R$, we can restate our main Theorem \ref{Thm1} for the diagonalized system \eqref{equation7} in the following theorem.
\begin{theorem}\label{Thm2}
Fix $s_{N}\geq6, \ s>0$. For all $C_{N},T_{0}>0$ there exist $C_{0},C_{R}$ and $\epsilon_{0}$ such that for all $\epsilon\in(0,\epsilon_{0})$ the following is true. Let $\tilde{A}_{1}\in C([0,T_{0}],H^{s_{N}})\in C([0,T_{0}],H^{s_{N}})$ be a solution of \eqref{NLS1} and $\tilde{B}_{1}\in C([0,T_{0}],H^{s_{N}})$ be a solutions of \eqref{NLS2} with
\begin{equation*}
\begin{split}
\sup_{T\in[0,T_{0}]}\|\tilde{A}_{1}\|_{H^{s_{N}}}\leq C_{N},\\
\sup_{T\in[0,T_{0}]}\|\tilde{B}_{1}\|_{H^{s_{N}}}\leq C_{N},
\end{split}
\end{equation*}
and $U_{0}=\epsilon\Theta|_{t=0}+\epsilon^{5/2}R_{0}$ with $\|R_{0}\|_{H^{s}}\leq C_{0}$. Then there is a unique solution $U=\epsilon\Theta+\epsilon^{5/2}R\in C([0,T_{0}/\epsilon^{2}],(H^{s})^{2})$ with $U|_{t=0}=U_{0}$ and $R|_{t=0}=R_{0}$ of \eqref{equation7} that satisfies
\begin{equation*}
\begin{split}
\sup_{t\in[0,T_{0}/\epsilon^{2}]}\|R(t)\|_{(H^{s})^{2}}\leq C_{R}.
\end{split}
\end{equation*}
\end{theorem}

We note that Theorem \ref{Thm1} is a consequence of Theorem \ref{Thm2} by combining the relation \eqref{rho} between the original variables $(\rho,v)$ of \eqref{EP} and diagonilized variable $U$ of \eqref{equation7} and the inequality \eqref{Aesti-2} in Lemma \ref{L2}.

The rest of this paper is devoted to the proof of Theorem \ref{Thm2}. The proof consists of an estimate showing that the error function $R$ stays $\mathcal{O}(1)$ bounded on the long time interval of length $\mathcal{O}(\epsilon^{-2})$. In particular, in order to do so, the terms of $\mathcal{O}(\epsilon)$ from the error equations \eqref{Rj0}-\eqref{Rj1} have to be eliminated by normal-form transformations.

\section{\textbf{The normal-form transformation for the low frequency terms}}
In this section, we will attempt to transform the low frequency terms of $\mathcal{O}(\epsilon)$ from the evolutionary equations \eqref{Rj0} and \eqref{Rj1} to terms of $\mathcal{O}(\epsilon^{2})$ by constructing appropriate normal-form transformations. Particularly, we need to make twice normal-form transformations for the low frequency terms in \eqref{Rj0} since $\hat{\vartheta}^{-1}(k)=\epsilon^{-1}$ for $|k|<\delta$.

Before constructing the first normal-form transformation we show a Lemma to simplify the subsequent discussions and help us extract the real dangerous terms from the quasilinear terms of $\mathcal{O}(\epsilon)$. This lemma mainly uses the strong localization of $\psi_{c}$ and $\phi_{c}$ near the wave numbers $k=\pm k_{0}$ in Fourier space.
\begin{lemma}\label{L6}
Fix $p\in\mathbb{R}$ and assume that $\kappa=\kappa(k,k-m,m)\in C(\mathbb{R},\mathbb{C})$. Assume further that the real-value function $\psi$ has a finitely supported Fourier transform and that $R\in H^{s}$. Then, \\
(i)\  if $\kappa$ is Lipschitz with respect to its second argument in some neighborhood of $p\in\mathbb{R}$,  there exists $C_{\psi,\kappa,p}>0$ such that
\begin{equation*}
\begin{split}
\|\int&\kappa(\cdot,\cdot-m,m)\widehat{\psi}(\frac{\cdot-m-p}{\epsilon})\widehat{R}(m)dm
-\int\kappa(\cdot,p,m)\widehat{\psi}(\frac{\cdot-m-p}{\epsilon})\widehat{R}(m)dm\|_{H^{s}}\\
&\leq C_{\psi,\kappa,p}\epsilon\|R\|_{H^{s}},
\end{split}
\end{equation*}
(ii) if $\kappa$ is globally Lipschitz with respect to its third argument, there exists $D_{\psi,\kappa}>0$ such that
\begin{equation*}
\begin{split}
\|\int&\kappa(\cdot,\cdot-m,m)\widehat{\psi}(\frac{\cdot-m-p}{\epsilon})\widehat{R}(m)dm
-\int\kappa(\cdot,\cdot-m,\cdot-p)\widehat{\psi}(\frac{\cdot-m-p}{\epsilon})\widehat{R}(m)dm\|_{H^{s}}\\
&\leq D_{\psi,\kappa}\epsilon\|R\|_{H^{s}}.
\end{split}
\end{equation*}
\end{lemma}
\begin{proof}
For (i), we have
 \begin{equation*}
\begin{split}
\Big\|\int&\kappa(\cdot,\cdot-m,m)\widehat{\psi}\big(\frac{\cdot-m-p}{\epsilon}\big)\widehat{R}(m)dm
-\int\kappa(\cdot,\cdot-m,\cdot-p)\widehat{\psi}\big(\frac{\cdot-m-p}{\epsilon}\big)\widehat{R}(m)dm\Big\|^{2}_{H^{s}}\\
&=\int\Big(\int(\kappa(k,k-m,m)-\kappa(k,p,m))\widehat{\psi}(\frac{k-m-p}{\epsilon})\widehat{R}(m)dm\Big)^{2}(1+k^{2})^{s}dk\\
&\leq\int\Big(C_{\kappa,p}\int|k-m-p|\widehat{\psi}(\frac{k-m-p}{\epsilon})\widehat{R}(m)dm\Big)^{2}(1+k^{2})^{s}dk\\
&\leq C_{\kappa,p}^{2}\Big(\int(1+\ell)^{s/2}|\ell||\widehat{\psi}(\frac{\ell}{\epsilon})|d\ell\Big)^{2}\|R\|^{2}_{H^{s}}\\
&\leq C_{\psi,\kappa,p}\epsilon^{2}\|R\|^{2}_{H^{s}},
\end{split}
\end{equation*}
thanks to the Young's inequality and the fact that $\widehat{\psi}$ has compact support. The second one (ii) is similar.
\end{proof}

In this Section, our target is to eliminate the low frequency terms of $\mathcal{O}(\epsilon)$ from \eqref{Rj0} and \eqref{Rj1} by constructing normal-form transformations, and then transform the system \eqref{Rj0}-\eqref{Rj1} for the error $(R^{0}, R^{1})$ to the transformed system for $(\mathcal{R}^{0}, \mathcal{R}^{1})$, i.e.
 \begin{equation*}
\begin{split}
 (R^{0}, R^{1})\rightarrow(\mathcal{R}^{0}, \mathcal{R}^{1}).
 \end{split}
\end{equation*}
To describe clearly the difficulties that arise from the normal-form transformation, we divide this process into two steps
\begin{equation*}
\begin{split}
(R^{0}, R^{1})\xrightarrow[{\eqref{equ3}} \ in\ \S6.1]{Step 1}(\widetilde{R}^{0}, \widetilde{R}^{1})\xrightarrow[{\eqref{equation88}} \ in\ \S6.2]{Step 2}(\mathcal{R}^{0}, \mathcal{R}^{1}).
\end{split}
\end{equation*}
After these two steps, we finally state the main result of this section in Proposition \ref{P3}.

We make the first normal-form transformation in Section 6.1 and then we examine the effects that the normal-form transformation causes on the full equations.

\subsection{\textbf{The first normal-form transformation}}
According to the form of the low frequency terms of $\mathcal{O}(\epsilon)$ from \eqref{Rj0} and \eqref{Rj1} we construct the following transformations
\begin{equation}
\begin{split}\label{equ3}
\widetilde{R}_{j}^{0}=R_{j}^{0}+\epsilon \sum_{p\in\{\pm1\}}\big[B_{j,p}^{0,1,+}(\psi_{c},R_{p}^{1})
+B_{j,p}^{0,1,-}(\phi_{c},R_{p}^{1})\big],\\
\widetilde{R}_{j}^{1}=R_{j}^{1}+\epsilon \sum_{p\in\{\pm1\}}\big[B_{j,p}^{1,0,+}(\psi_{c},R_{p}^{0})
+B_{j,p}^{1,0,-}(\phi_{c},R_{p}^{0})\big],
\end{split}
\end{equation}
where
\begin{equation}
\begin{split}\label{equ4}
\widehat{B}_{j,p}^{0,1,+}(\psi_{c},R_{p}^{1})=\int b_{j,p}^{0,1,+}(k,k-m,m)\widehat{\psi}_{c}(k-m)\widehat{R}_{p}^{1}(m)dm,\\
\widehat{B}_{j,p}^{0,1,-}(\phi_{c},R_{p}^{1})=\int b_{j,p}^{0,1,-}(k,k-m,m)\widehat{\phi}_{c}(k-m)\widehat{R}_{p}^{1}(m)dm,\\
\widehat{B}_{j,p}^{1,0,+}(\psi_{c},R_{p}^{0})=\int b_{j,p}^{1,0,+}(k,k-m,m)\widehat{\psi}_{c}(k-m)\widehat{R}_{p}^{0}(m)dm,\\
\widehat{B}_{j,p}^{1,0,-}(\phi_{c},R_{p}^{0})=\int b_{j,p}^{1,0,-}(k,k-m,m)\widehat{\phi}_{c}(k-m) \widehat{R}_{p}^{0}(m)dm.
\end{split}
\end{equation}
Note that $\psi_{c}=\psi_{1}+\psi_{-1}$ and $\phi_{c}=\phi_{1}+\phi_{-1}$, where $\psi_{\pm1}$ and $\phi_{\pm1}$ are located near $k=\pm k_{0}$. In the following discussion we will focus on the terms containing $\psi_{1}$ or $\phi_{1}$, those containing $\psi_{-1}$ or $\phi_{-1}$ are treated in an almost identical fashion.

\emph{\bf{Construction of $B_{j,p}^{0,1,\pm}$}.} Differentiating the $\widetilde{R}_{j}^{0}$ in \eqref{equ3} w.r.t. $t$, we obtain
\begin{equation}
\begin{split}\label{equ5}
\partial_{t}\widetilde{R}_{j}^{0}=\partial_{t}R_{j}^{0}
+\epsilon\sum_{p\in\{\pm1\}}B_{j,p}^{0,1,+}(\partial_{t}\psi_{c},R_{p}^{1})
+\epsilon\sum_{p\in\{\pm1\}}B_{j,p}^{0,1,+}(\psi_{c},\partial_{t}R_{p}^{1})\\
+\epsilon\sum_{p\in\{\pm1\}}B_{j,p}^{0,1,-}(\partial_{t}\phi_{c},R_{p}^{1})
+\epsilon\sum_{p\in\{\pm1\}}B_{j,p}^{0,1,-}(\phi_{c},\partial_{t}R_{p}^{1}).
\end{split}
\end{equation}
Recall that $\|\partial_{t}\widehat{\psi}_{c}-i\omega\widehat{\psi}_{c}\|_{L^{1}(s)}\leq C_{A}\epsilon^{2}$ and $\|\partial_{t}\widehat{\phi}_{c}+i\omega\widehat{\phi}_{c}\|_{L^{1}(s)}\leq C_{B}\epsilon^{2}$ in Lemma \ref{LO}. Then provided the transformation $B_{j,p}^{0,1,\pm}$ is well-defined and bounded we have
\begin{equation}
\begin{split}\label{equ6}
\partial_{t}\widetilde{R}_{j}^{0}=&j\Omega\widetilde{R}_{j}^{0}
+\epsilon \sum_{p\in\{\pm1\}}\Big[-j\Omega B_{j,p}^{0,1,+}(\psi_{c},R_{p}^{1})
+P^{0}\displaystyle\frac{1}{2\vartheta}\partial_{x}
N^{+}_{j,p}(\psi_{c},\vartheta R_{p}^{1})\\
&\ \ \ \ \ \ \ \ \ \ \ \ \ \ \ \ \ \ \ \ \ \ \ \ +B_{j,p}^{0,1,+}(\Omega\psi_{c},R_{p}^{1})
+B_{j,p}^{0,1,+}(\psi_{c},p\Omega R_{p}^{1})\Big]\\
&+\epsilon \sum_{p\in\{\pm1\}}\Big[-j\Omega B_{j,p}^{0,1,-}(\phi_{c},R_{p}^{1})
+P^{0}\displaystyle\frac{1}{2\vartheta}\partial_{x}
N^{-}_{j,p}(\phi_{c},\vartheta R_{p}^{1})\\
&\ \ \ \ \ \ \ \ \ \ \ \ \ \ \ \ -B_{j,p}^{0,1,-}(\Omega\phi_{c},R_{p}^{1})
+B_{j,p}^{0,1,-}(\phi_{c},p\Omega R_{p}^{1})\Big]\\
&+\epsilon^{2}\sum_{p,q\in\{\pm1\}}B_{j,p}^{0,1,+}
\Big(\psi_{c},P^{1}\displaystyle\frac{1}{2\vartheta}\partial_{x}
\big(N^{+}_{p,q}(\psi_{c},\vartheta R_{q}^{1})
+N^{-}_{p,q}(\phi_{c},\vartheta R_{q}^{1})\big)\Big)\\
&+\epsilon^{2}\sum_{p,q\in\{\pm1\}}B_{j,p}^{0,1,-}
\Big(\phi_{c},P^{1}\displaystyle\frac{1}{2\vartheta}\partial_{x}
\big(N^{+}_{p,q}(\psi_{c},\vartheta R_{q}^{1})
+N^{-}_{p,q}(\phi_{c},\vartheta R_{q}^{1})\big)\Big)
+\epsilon^{2}\mathcal{F}_{j}^{2},
\end{split}
\end{equation}
where $\epsilon^{2}\mathcal{F}_{j}^{2}$ satisfies
\begin{equation}
\begin{split}\label{j2}
\|\epsilon^{2}\mathcal{F}_{j}^{2}\|_{H^{s'}}\leq C \epsilon^{2}(1+\|R\|_{H^{s}}+\epsilon^{3/2}\|R\|^{2}_{H^{s}}
+\epsilon^{3}\|R\|^{3}_{H^{s}}),
\end{split}
\end{equation}
for any $s'$ and $s\geq6$. Note that $\mathcal{F}_{j}^{2}$ does not lose any derivatives due to the factor of $P^{0}$.
To eliminate all terms of order $\mathcal{O}(\epsilon)$, we choose $B_{j,p}^{0,1,\pm}$ to satisfy the following equalities,
\begin{equation*}
\begin{split}
-j\Omega B_{j,p}^{0,1,+}(\psi_{c},R_{p}^{1})
+P^{0}\displaystyle\frac{1}{2\vartheta}\partial_{x}
N^{+}_{j,p}(\psi_{c}\vartheta R_{p}^{1})+B_{j,p}^{0,1,+}(\Omega\psi_{c},R_{p}^{1})
+B_{j,p}^{0,1,+}(\psi_{c},p\Omega R_{p}^{1})=0,\\
-j\Omega B_{j,p}^{0,1,-}(\phi_{c},R_{p}^{1})
+P^{0}\displaystyle\frac{1}{2\vartheta}\partial_{x}
N^{-}_{j,p}(\phi_{c}\vartheta R_{p}^{1})-B_{j,p}^{0,1,-}(\Omega\phi_{c},R_{p}^{1})
+B_{j,p}^{0,1,-}(\phi_{c},p\Omega R_{p}^{1})=0,
\end{split}
\end{equation*}
which means that the kernel of $B^{0,1,\pm}$ should be of the form
\begin{equation}
\begin{split}\label{equ7}
b_{j,p}^{0,1,\pm}(k,k-m,m)=\frac{ \widehat{P}^{0}(k)k\zeta_{j,p}^{\pm}(k,k-m,m)}
{j\omega(k)\mp\omega(k-m)-p\omega(m)}\frac{\widehat{\vartheta}(m)}{2\widehat{\vartheta}(k)}.
\end{split}
\end{equation}
Because of the fact that the $\widehat{P}^{0}$ and $\widehat{\psi}_{c}, \widehat{\phi}_{c}$ have supports located near $k=0$ and $k-m=\pm k_{0}$, the kernel function in \eqref{equ7} must be considered in the region $|k|<\delta$ and $|k-m\pm k_{0}|<\delta$ for a sufficiently small $\delta>0$ but independent of $0<\epsilon\ll1$. As a result, the wave number $m$ needs to be restricted to bound away from $0$, i.e. $m\approx\pm k_{0}$. Therefore, only the trivial resonance at $k=0$ plays a role for $B_{j,p}^{0,1,\pm}$. The kernel function $b_{j,p}^{0,1,\pm}$ can be estimated as follows. First considering the denominator of the kernel function in \eqref{equ7} near $k=0$, we have
\begin{equation*}
\begin{split}
j\omega(k)\mp\omega(k-m)-p\omega(m)
=j\omega'(0)k\mp(\omega(-m)+\omega'(-m)k)-p\omega(m)+\mathcal{O}(k^{2}).
\end{split}
\end{equation*}
When $p=\mp1$, we have $\mp\omega(-m)-p\omega(m)=\pm 2\omega(m)\approx\pm2\omega(k_{0})\neq0$. This means that this quantity is bounded from below by some $\mathcal{O}(1)$ term. When $p=\pm1$, we have $\mp\omega(-m)-p\omega(m)=0$ and $\omega'(\pm k_{0})\neq \pm \omega'(0)$, then there exists a positive constant $C$ such that
\begin{equation}
\begin{split}\label{equ8}
|j\omega(k)\mp\omega(k-m)-p\omega(m)|\geq C|k|.
\end{split}
\end{equation}
On the other hand, the numerator of the kernel function in \eqref{equ7} satisfies $|\widehat{P}_{0}(k)k\zeta_{j,p}^{\pm}(k,k-m,m)|\leq C|k|$, and thus there exist $C\geq0$ such that
\begin{equation}
\begin{split}\label{equ9}
|\widehat{\vartheta}(k)b_{j,p}^{0,1,\pm}(k,k-m,m)|\leq C,
\end{split}
\end{equation}
for all $|k|<\delta$.

Since $\widehat{P}^{0}(k)$ has supports located near $k=0$, we have $b_{j,p}^{0,1,\pm}=0$ for $|k|\geq\delta$. That means $B^{0,1,\pm}$ does not lose any derivatives. Thus given any $s'$, there exists $C_{s'}$ independent of $\epsilon$ such that
\begin{equation}
\begin{split}\label{b01}
\|\epsilon B_{j,p}^{0,1,+}(\psi_{c},R_{p}^{1})\|_{H^{s'}}\leq C_{s'}\|R^{1}\|_{H^{s}}
\end{split}
\end{equation}
for $R^{1}\in H^{s}$ for some $s>1$. Similar estimate holds for $\epsilon B_{j,p}^{0,1,-}$. In particular, this estimate holds when $s'=s$. Generally, we cannot have $C_{s'}\sim\mathcal{O}(\epsilon)$ because of $\widehat{\vartheta}^{-1}(k)\sim\epsilon^{-1}$ for $|k|<\delta$ in \eqref{equ7}. That is why we leave the following two terms which seems to be $\mathcal{O}(\epsilon^{2})$ in the equation \eqref{equ6}:
\begin{equation}
\begin{split}\label{leave}
&\epsilon^{2}\sum_{p,q\in\{\pm1\}}B_{j,p}^{0,1,+}
\Big(\psi_{c},P^{1}\displaystyle\frac{1}{2\vartheta}\partial_{x}
\big(N^{+}_{p,q}(\psi_{c},\vartheta R_{q}^{1})
+N^{-}_{p,q}(\phi_{c},\vartheta R_{q}^{1})\big)\Big)\\
&+\epsilon^{2}\sum_{p,q\in\{\pm1\}}B_{j,p}^{0,1,-}
\Big(\phi_{c},P^{1}\displaystyle\frac{1}{2\vartheta}\partial_{x}
\big(N^{+}_{p,q}(\psi_{c},\vartheta R_{q}^{1})
+N^{-}_{p,q}(\phi_{c},\vartheta R_{q}^{1})\big)\Big).
\end{split}
\end{equation}
In fact, since the kernel function of the transformation $B^{0,1}$ is $\mathcal{O}(\epsilon^{-1})$ for certain wave numbers, these two terms are only $\mathcal{O}(\epsilon)$ and need to be eliminated further by a second normal-form transformations.

\emph{\bf{Construction of $B_{j,p}^{1,0,\pm}$.}} Differentiating the $\widetilde{R}_{j}^{1}$ in \eqref{equ3} w.r.t. $t$, we obtain
\begin{equation}
\begin{split}\label{equ10}
\partial_{t}\widetilde{R}_{j}^{1}=&\partial_{t}R_{j}^{1} +\epsilon\sum_{p\in\{\pm1\}}B_{j,p}^{1,0,+} (\partial_{t}\psi_{c},R_{p}^{0}) +\epsilon\sum_{p\in\{\pm1\}}B_{j,p}^{1,0,+} (\psi_{c},\partial_{t}R_{p}^{0})\\
&+\epsilon\sum_{p\in\{\pm1\}}B_{j,p}^{1,0,-}(\partial_{t}\phi_{c},R_{p}^{0})
+\epsilon\sum_{p\in\{\pm1\}}B_{j,p}^{1,0,-}(\phi_{c},\partial_{t}R_{p}^{0}).
\end{split}
\end{equation}
Applying Lemma \ref{LO} and the evolutionary equation \eqref{Rj0}-\eqref{Rj1} for $R$ as above, we have
\begin{equation}
\begin{split}\label{equ11}
\partial_{t}\widetilde{R}_{j}^{1}=&j\Omega\widetilde{R}_{j}^{1}
+\epsilon \sum_{p\in\{\pm1\}}\Big[-j\Omega B_{j,p}^{1,0,+}(\psi_{c},R_{p}^{0})
+P^{1}\displaystyle\frac{1}{2\vartheta}\partial_{x}
N^{+}_{j,p}(\psi_{c},\vartheta_{0} R_{p}^{0})\\
&\ \ \ \ \ \ \ \ \ \ \ \ \ \ \ \ \ \ \ \ \ \ \ \ +B_{j,p}^{1,0,+}(\Omega\psi_{c},R_{p}^{0})
+B_{j,p}^{1,0,+}(\psi_{c},p\Omega R_{p}^{0})\Big]\\
&+\epsilon \sum_{p\in\{\pm1\}}\Big[-j\Omega B_{j,p}^{1,0,-}(\phi_{c},R_{p}^{0})
+P^{1}\displaystyle\frac{1}{2\vartheta}\partial_{x}
N^{-}_{j,p}(\phi_{c},\vartheta_{0} R_{p}^{0})\\
&\ \ \ \ \ \ \ \ \ \ \ \ \ \ \ \ -B_{j,p}^{1,0,-}(\Omega\phi_{c},R_{p}^{0})
+B_{j,p}^{1,0,-}(\phi_{c},p\Omega R_{p}^{0})\Big]\\
&+\epsilon\sum_{p\in\{\pm1\}}P^{1}\displaystyle\frac{1}{2\vartheta}\partial_{x}
(N^{+}_{j,p}(\psi_{c},\vartheta R_{p}^{1})+
N^{-}_{j,p}(\phi_{c},\vartheta R_{p}^{1}))
+\epsilon^{2}\mathcal{F}_{j}^{3},
\end{split}
\end{equation}
where $\epsilon^{2}\mathcal{F}_{j}^{3}$ satisfies that
\begin{equation}
\begin{split}\label{j3}
\|\epsilon^{2}\mathcal{F}_{j}^{3}\|_{H^{s-1}}\leq C \epsilon^{2}(1+\|R\|_{H^{s}}+\epsilon^{3/2}\|R\|^{2}_{H^{s}}
+\epsilon^{3}\|R\|^{3}_{H^{s}}),
\end{split}
\end{equation}
for any $s\geq6$. Note that $\mathcal{F}_{j}^{3}$ loses one derivative due to the quasilinear term in the ion Euler-Poisson system \eqref{EP}.
Here we have used $\widehat{\vartheta}_{0}(0)=0$ with $\widehat{\vartheta}_{0}=\widehat{\vartheta}-\epsilon$ to solve a resonance problem at $k=\pm k_{0}$ in \eqref{equ11}. In fact, we have used the equality:
\begin{equation*}
\begin{split}
\epsilon P^{1}\displaystyle\frac{1}{2\vartheta}\partial_{x}
(N^{+}_{j,p}(\psi_{c},\vartheta R_{p}^{0})+
N^{-}_{j,p}(\phi_{c},\vartheta R_{p}^{0}))
= &
\epsilon P^{1}\displaystyle\frac{1}{2\vartheta}\partial_{x}
\big(N^{+}_{j,p}(\psi_{c},\vartheta_{0} R_{p}^{0})+
N^{-}_{j,p}(\phi_{c},\vartheta_{0} R_{p}^{0})\big)\\
&+\epsilon^{2}P^{1}\displaystyle\frac{1}{2\vartheta}\partial_{x}
\big(N^{+}_{j,p}(\psi_{c},R_{p}^{0})+
N^{-}_{j,p}(\phi_{c},R_{p}^{0})\big),
\end{split}
\end{equation*}
where the additional term $\epsilon^{2}P^{1}\displaystyle\frac{1}{2\vartheta}\partial_{x}
\big(N^{+}_{j,p}(\psi_{c},R_{p}^{0})+
N^{-}_{j,p}(\phi_{c},R_{p}^{0})\big)$ is included in last term of the form $\mathcal{O}(\epsilon^{2})$ in \eqref{equ11} since $\widehat{\vartheta}^{-1}(k)$ is $\mathcal{O}(1)$ on the support of $\widehat{P}^{1}$.

To eliminate all terms of $\mathcal{O}(\epsilon)$, we choose $B^{1,0,\pm}$ in \eqref{equ11} such that
\begin{equation}
\begin{split}\label{equ12}
-j\Omega B_{j,p}^{1,0,+}(\psi_{c},R_{p}^{0})
+P^{1}\displaystyle\frac{1}{2\vartheta}\partial_{x}
N^{+}_{j,p}(\psi_{c},\vartheta_{0} R_{p}^{0})
+B_{j,p}^{1,0,+}(\Omega\psi_{c},R_{p}^{0})
+B_{j,p}^{1,0,+}(\psi_{c},p\Omega R_{p}^{0})=0,\\
-j\Omega B_{j,p}^{1,0,-}(\phi_{c},R_{p}^{0})
+P^{1}\displaystyle\frac{1}{2\vartheta}\partial_{x}
N^{-}_{j,p}(\phi_{c},\vartheta_{0} R_{p}^{0})-B_{j,p}^{1,0,-}(\Omega\phi_{c},R_{p}^{0})
+B_{j,p}^{1,0,-}(\phi_{c},p\Omega R_{p}^{0})=0,
\end{split}
\end{equation}
for $j,p\in\{\pm1\}$. That means the kernel of $B^{1,0,\pm}$ should be of the form
\begin{equation}
\begin{split}\label{equa7}
b_{j,p}^{1,0,\pm}(k,k-m,m)=\frac{ \widehat{P}^{1}(k)k\zeta_{j,p}^{\pm}(k,k-m,m)}
{j\omega(k)\mp\omega(k-m)-p\omega(m)}\frac{\widehat{\vartheta}_{0}(m)}{2\widehat{\vartheta}(k)}.
\end{split}
\end{equation}
However, in order to explain the nontrivial resonance in the wave numbers $k=\pm k_{0}$ clearly, we will not use the kernel function \eqref{equa7} by computing directly from \eqref{equ12} but rather use Lemma \ref{L6} to obtain a modified equation for $B_{j,p}^{1,0,\pm}$. This will make the normal-form transformation easier to be bounded, only at the expense of adding some additional terms of $\mathcal{O}(\epsilon^{2})$ in \eqref{equ11}. Specifically, we apply Lemma \ref{L6} to make the following changes.
\begin{description}
  \item[(a)]  Replace $P^{1}\frac{1}{2\vartheta}\partial_{x}
N^{+}_{j,p}(\psi_{c},\vartheta_{0} R_{p}^{0})$ by $P^{1}\frac{1}{2\vartheta}\partial_{x}
(N^{+,1}_{j,p}(\psi_{1},\vartheta_{0} R_{p}^{0})+N^{+,-1}_{j,p}(\psi_{-1},\vartheta_{0} R_{p}^{0}))$, where
\begin{equation*}
\begin{split}
\mathcal{F}(P^{1}&\frac{1}{2\vartheta}\partial_{x}
N^{+,1}_{j,p}(\psi_{1},\vartheta_{0} R_{p}^{0}))(k)
+\mathcal{F}(P^{1}\frac{1}{2\vartheta}\partial_{x}
N^{+,-1}_{j,p}(\psi_{-1},\vartheta_{0} R_{p}^{0}))(k)\\
=&\widehat{P}^{1}(k)\frac{ik}{2\widehat{\vartheta}(k)}\int\zeta^{+}_{j,p}(k)
\widehat{\vartheta}_{0}(k-k_{0})\widehat{\psi}_{1}(k-m)\widehat{R}_{p}^{0}(m)dm\\
&+\widehat{P}^{1}(k)\frac{ik}{2\widehat{\vartheta}(k)}\int\zeta^{+}_{j,p}(k)
\widehat{\vartheta}_{0}(k+k_{0})\widehat{\psi}_{-1}(k-m) \widehat{R}_{p}^{0}(m)dm.
\end{split}
\end{equation*}
  \item[(b)] Replace $B_{j,p}^{1,0,+}(\Omega\psi_{c},R_{p}^{0})$ by $B_{j,p}^{1,0,+}(\Omega_{c}\psi_{c},R_{p}^{0})$, where
\begin{equation*}
\begin{split}
\mathcal{F}((B_{j,p}^{1,0,+}(\Omega_{c}\psi_{c},R_{p}^{0}))(k)
=&\int b_{j,p}^{1,0,+,1}(k)i\omega(k_{0})\widehat{\psi}_{1}(k-m)\widehat{R}_{p}^{0}(m)dm\\
&+\int b_{j,p}^{1,0,+,-1}(k)i\omega(-k_{0})\widehat{\psi}_{-1}(k-m) \widehat{R}_{p}^{0}(m)dm.
\end{split}
\end{equation*}
  \item[(c)] Replace $B_{j,p}^{1,0,+}(\psi_{c},p\Omega R_{p}^{0})$ by $B_{j,p}^{1,0,+}(\psi_{c},p\Omega_{c} R_{p}^{0})$, where
\begin{equation*}
\begin{split}
\mathcal{F}(B_{j,p}^{1,0,+}(\psi_{c},p\Omega_{c} R_{p}^{0}))(k)
=&\int b_{j,p}^{1,0,+,1}(k)ip\omega(k-k_{0})\widehat{\psi}_{1}(k-m)\widehat{R}_{p}^{0}(m)dm\\
&+\int b_{j,p}^{1,0,+,-1}(k)ip\omega(k+k_{0})\widehat{\psi}_{-1}(k-m) \widehat{R}_{p}^{0}(m)dm,
\end{split}
\end{equation*}
with a similar form for $B^{1,0,-}$.
\end{description}
Inserting these changes into \eqref{equ12} we obtain that the kernel function of $B^{1,0,+}$ satisfies
\begin{equation}
\begin{split}\label{b10}
b_{j,p}^{1,0,\pm,1}=\frac{\widehat{P}^{1}(k)k\zeta_{j,p}^{\pm}(k)}
{j\omega(k)\mp\omega(k_{0})-p\omega(k-k_{0})}
\frac{\widehat{\vartheta}_{0}(k-k_{0})}{2\widehat{\vartheta}(k)},
\end{split}
\end{equation}
with a similar form for $b_{j,p}^{1,0,\pm,-1}$.

Due to the supports of $\widehat{P}^{1}$, $\widehat{\psi}_{c}$ and $\widehat{\phi}_{c}$, the kernel function in \eqref{b10} must be considered in the region $|k|>\delta$ and $|k-m-k_{0}|<\delta$.
Therefore, only the nontrivial resonance at $k=k_{0}$ whenever $j=\pm1$ will play a role for $B_{j,p}^{1,0,\pm,1}$.
However, since we have a bound on the denominator
\begin{equation*}
\begin{split}
|j\omega(k)\mp\omega(k_{0})-p\omega(k-k_{0})|\geq C|k-k_{0}|,
\end{split}
\end{equation*}
this singularity is offset since $|\widehat{\vartheta}_{0}(k-k_{0})|\leq C|k-k_{0}|$. Thus, the normal-form transformation $B^{1,0,\pm}$ is well-defined and $\mathcal{O}(1)$ bounded. Thanks to the compact support of $\widehat{R}^{0}$, we have
\begin{equation}
\begin{split}\label{B10}
\|\epsilon B_{j,p}^{1,0,+}(\psi_{c},R_{p}^{0})\|_{H^{s'}}\leq C\epsilon\|R^{0}\|_{H^{s}},
\end{split}
\end{equation}
for any $s'>0$ and $R^{0}\in H^{s}$. Similar estimate for $\epsilon B_{j,p}^{1,0,-}$ holds.

We sum up the results for the first normal-form transformation above in the following Proposition.
\begin{proposition}\label{P1}
Define $(\widetilde{R}_{j}^{0},\widetilde{R}_{j}^{1})$ with $j=\pm1$ via the transformation \eqref{equ3} with the kernel functions \eqref{equ7} and \eqref{equa7}. This transformation maps $(R_{j}^{0},R_{j}^{1})\in H^{s}\times H^{s}$ into $(\widetilde{R}_{j}^{0},\widetilde{R}_{j}^{1})\in H^{s}\times H^{s}$ for all $s>0$ and is invertible on its range. Furthermore, if we write the inverse transformations as
\begin{equation*}
\begin{split}
R_{j}^{0}=\widetilde{R}_{j}^{0}+\epsilon\mathcal{B}^{-1}_{0,1}
(\widetilde{R}_{j}^{0},\widetilde{R}_{j}^{1}), \ \ \ \
R_{j}^{1}=\widetilde{R}_{j}^{1}+\epsilon\mathcal{B}^{-1}_{1,0}
(\widetilde{R}_{j}^{0},\widetilde{R}_{j}^{1}),
\end{split}
\end{equation*}
then there exist constants $C_{0},C_{1}$ such that the inverse transformations satisfy the estimates
\begin{equation*}
\begin{split}
\left\|\epsilon\mathcal{B}^{-1}_{0,1}(\widetilde{R}_{j}^{0},\widetilde{R}_{j}^{1})\right\|_{H^{s}}\leq C_{0}\left(\|\widetilde{R}_{j}^{0}\|_{H^{s}}+\|\widetilde{R}_{j}^{1}\|_{H^{s}}\right),\\
\left\|\epsilon\mathcal{B}^{-1}_{1,0}(\widetilde{R}_{j}^{0},\widetilde{R}_{j}^{1})\right\|_{H^{s}}\leq C_{1}\epsilon\left(\|\widetilde{R}_{j}^{0}\|_{H^{s}}+\|\widetilde{R}_{j}^{1}\|_{H^{s}}\right).
\end{split}
\end{equation*}
\end{proposition}
\begin{proof}
According to \eqref{b01} and \eqref{B10}, there is no loss of regularities for the transformation, thus this transformation is invertible by a simple application of the Neumann series.
\end{proof}
Finally, we can rewrite the system \eqref{Rj0}-\eqref{Rj1} based on the above transformations as
\begin{equation}
\begin{split}\label{Rj00}
\partial_{t}\widetilde{R}_{j}^{0}=&j\Omega\widetilde{R}_{j}^{0}+\epsilon^{2}\sum_{p,q\in\{\pm1\}}
B_{j,p}^{0,1,+}
\Big(\psi_{c},P_{1}\displaystyle\frac{1}{2\vartheta}\partial_{x}
\big(N^{+}_{p,q}(\psi_{c},\vartheta \widetilde{R}_{q}^{1})
+N^{-}_{p,q}(\phi_{c},\vartheta \widetilde{R}_{q}^{1})\big)\Big)\\
&+\epsilon^{2}\sum_{p,q\in\{\pm1\}}B_{j,p}^{0,1,-}
\Big(\phi_{c},P_{1}\displaystyle\frac{1}{2\vartheta}\partial_{x}
\big(N^{+}_{p,q}(\psi_{c},\vartheta \widetilde{R}_{q}^{1})
+N^{-}_{p,q}(\phi_{c},\vartheta \widetilde{R}_{q}^{1})\big)\Big)+\epsilon^{2}\mathcal{F}_{j}^{4},
\end{split}
\end{equation}
\begin{equation}
\begin{split}\label{Rj01}
\partial_{t}\widetilde{R}_{j}^{1}=j\Omega\widetilde{R}_{j}^{1}
+\epsilon\sum_{p\in\{\pm1\}}P^{1}\displaystyle\frac{1}{2\vartheta}\partial_{x}
(N^{+}_{j,p}(\psi_{c},\vartheta \widetilde{R}_{p}^{1})+
N^{-}_{j,p}(\phi_{c},\vartheta \widetilde{R}_{p}^{1}))+\epsilon^{2}\mathcal{F}_{j}^{5},\ \ \ \ \ \ \ \ \ \ \ \
\end{split}
\end{equation}
where $\epsilon^{2}\mathcal{F}_{j}^{4}$ and $\epsilon^{2}\mathcal{F}_{j}^{5}$ have similar estimates \eqref{j2} and \eqref{j3} like the terms of $\epsilon^{2}\mathcal{F}_{j}^{2}$ and $\epsilon^{2}\mathcal{F}_{j}^{3}$ respectively.

\subsection{\textbf{The second normal-form transformation}}
Now we construct a second normal-form transformation to eliminate the terms of $\mathcal{O}(\epsilon)$ from the equation of \eqref{Rj00}. Before doing this we first apply Lemma \ref{L6} to extract the real dangerous $\mathcal{O}(\epsilon)$ terms.

Recall that $\psi_{c}=\psi_{1}+\psi_{-1}$ and $\phi_{c}=\phi_{1}+\phi_{-1}$ with $\psi_{\pm1}$ and $\phi_{\pm 1}$ supported in a neighborhood of size $\delta$ of $\pm k_{0}$ in Fourier space. Then we can write
\begin{equation}
\begin{split}\label{ee}
&\epsilon^{2}B_{j,p}^{0,1,+}
\left(\psi_{c},P_{1}\displaystyle\frac{1}{2\vartheta}\partial_{x}
\left(N^{+}_{p,q}(\psi_{c},\vartheta \widetilde{R}_{q}^{1})
+N^{-}_{p,q}(\phi_{c},\vartheta \widetilde{R}_{q}^{1})\right)\right)\\
=&\epsilon^{2}B_{j,p}^{0,1,+}
\left(\psi_{1},P_{1}\displaystyle\frac{1}{2\vartheta}\partial_{x}
\left(N^{+}_{p,q}(\psi_{-1},\vartheta \widetilde{R}_{q}^{1})
+N^{-}_{p,q}(\phi_{-1},\vartheta \widetilde{R}_{q}^{1})\right)\right)\\
&+\epsilon^{2}B_{j,p}^{0,1,+}
\left(\psi_{-1},P_{1}\displaystyle\frac{1}{2\vartheta}\partial_{x}
\left(N^{+}_{p,q}(\psi_{1},\vartheta \widetilde{R}_{q}^{1})
+N^{-}_{p,q}(\phi_{1},\vartheta \widetilde{R}_{q}^{1})\right)\right)\\
&+\epsilon^{2}B_{j,p}^{0,1,+}
\left(\psi_{1},P_{1}\displaystyle\frac{1}{2\vartheta}\partial_{x}
\left(N^{+}_{p,q}(\psi_{1},\vartheta \widetilde{R}_{q}^{1})
+N^{-}_{p,q}(\phi_{1},\vartheta \widetilde{R}_{q}^{1})\right)\right)\\
&+\epsilon^{2}B_{j,p}^{0,1,+}
\left(\psi_{-1},P_{1}\displaystyle\frac{1}{2\vartheta}\partial_{x}
\left(N^{+}_{p,q}(\psi_{-1},\vartheta \widetilde{R}_{q}^{1})
+N^{-}_{p,q}(\phi_{-1},\vartheta \widetilde{R}_{q}^{1})\right)\right),
\end{split}
\end{equation}
and
\begin{equation}
\begin{split}\label{eee}
&\epsilon^{2}B_{j,p}^{0,1,-}
\left(\phi_{c},P_{1}\displaystyle\frac{1}{2\vartheta}\partial_{x}
\left(N^{+}_{p,q}(\psi_{c},\vartheta \widetilde{R}_{q}^{1})
+N^{-}_{p,q}(\phi_{c},\vartheta \widetilde{R}_{q}^{1})\right)\right)\\
=&\epsilon^{2}B_{j,p}^{0,1,-}
\left(\phi_{1},P_{1}\displaystyle\frac{1}{2\vartheta}\partial_{x}
\left(N^{+}_{p,q}(\psi_{-1},\vartheta \widetilde{R}_{q}^{1})
+N^{-}_{p,q}(\phi_{-1},\vartheta \widetilde{R}_{q}^{1})\right)\right)\\
&+\epsilon^{2}B_{j,p}^{0,1,-}
\left(\phi_{-1},P_{1}\displaystyle\frac{1}{2\vartheta}\partial_{x}
\left(N^{+}_{p,q}(\psi_{1},\vartheta \widetilde{R}_{q}^{1})
+N^{-}_{p,q}(\phi_{1},\vartheta \widetilde{R}_{q}^{1})\right)\right)\\
&+\epsilon^{2}B_{j,p}^{0,1,-}
\left(\phi_{1},P_{1}\displaystyle\frac{1}{2\vartheta}\partial_{x}
\left(N^{+}_{p,q}(\psi_{1},\vartheta \widetilde{R}_{q}^{1})
+N^{-}_{p,q}(\phi_{1},\vartheta \widetilde{R}_{q}^{1})\right)\right)\\
&+\epsilon^{2}B_{j,p}^{0,1,-}
\left(\phi_{-1},P_{1}\displaystyle\frac{1}{2\vartheta}\partial_{x}
\left(N^{+}_{p,q}(\psi_{-1},\vartheta \widetilde{R}_{q}^{1})
+N^{-}_{p,q}(\phi_{-1},\vartheta \widetilde{R}_{q}^{1})\right)\right).
\end{split}
\end{equation}
Each term on the RHS of \eqref{ee} can be rewritten as
\begin{equation}
\begin{split}\label{nv}
&\epsilon^{2}\mathcal{F}\left(B_{j,p}^{0,1,+}
\left(\psi_{\mu},P_{1}\displaystyle\frac{1}{2\vartheta}\partial_{x}
\left(N^{+}_{p,q}(\psi_{\nu},\vartheta \widetilde{R}_{q}^{1})
+N^{-}_{p,q}(\phi_{\nu},\vartheta \widetilde{R}_{q}^{1})\right)\right)\right)(k)\\
&=\frac{\epsilon^{2}}{2}\int b_{j,p}^{0,1,+}(k,k-\ell,\ell)
\widehat{\psi}_{\mu}(k-\ell)\widehat{P}_{1}(\ell)\widehat{\vartheta}^{-1}(\ell)i\ell\\
&\times\left(
\int(\zeta_{p,q}^{+}(\ell,\ell-m,m)\widehat{\psi}_{\nu}(\ell-m)
+\zeta_{p,q}^{-}(\ell,\ell-m,m)\widehat{\phi}_{\nu}(\ell-m))
\widehat{\vartheta}(m)\widehat{\widetilde{R_{q}^{1}}}(m)dm\right)d\ell,
\end{split}
\end{equation}
where $\mu,\nu\in\{1,-1\}$. We can further apply Lemma \ref{L6} to rewrite \eqref{nv} as
\begin{equation*}
\begin{split}
\frac{\epsilon^{2}}{2}&\mathcal{F}\left(B_{j,p}^{0,1,+}\left(\psi_{\mu}, P_{1}(\cdot-\mu k_{0})i(\cdot-\mu k_{0})\vartheta^{-1}(\cdot-\mu k_{0})N^{+}_{p,q}(
\psi_{\nu},\vartheta(\cdot-\nu k_{0})\widetilde{R}_{q}^{1}\right)\right)(k)\\
+&\frac{\epsilon^{2}}{2}\mathcal{F}\left(B_{j,p}^{0,1,+}\left(\psi_{\mu}, P_{1}(\cdot-\mu k_{0})i(\cdot-\mu k_{0})\vartheta^{-1}(\cdot-\mu k_{0})N^{-}_{p,q}(
\phi_{\nu},\vartheta(\cdot-\nu k_{0})\widetilde{R}_{q}^{1}\right)\right)(k)
+\mathcal{O}(\epsilon^{2})\\
=&\frac{\epsilon^{2}}{2}\int b_{j,p}^{0,1,+}(k,\mu k_{0},k-\mu k_{0})\widehat{\psi}_{\mu}(k-\ell)
\widehat{P}_{1}(k-\mu k_{0})i(k-\mu k_{0})\widehat{\vartheta}^{-1}(k-\mu k_{0})\\
&\times\bigg(
\int\Big(\zeta_{p,q}^{+}(k-\mu k_{0},\nu k_{0},k-(\mu+\nu)k_{0})\widehat{\psi}_{\nu}(\ell-m)\\
& \ \ \ \ \ \ \ \ \ +\zeta_{p,q}^{-}(k-\mu k_{0},\nu k_{0},k-(\mu+\nu)k_{0})\widehat{\phi}_{\nu}(\ell-m)\Big)
\widehat{\vartheta}(k-(\mu+\nu)k_{0})\widehat{\widetilde{R_{q}^{1}}}(m)dm\bigg)d\ell\\
&+\mathcal{O}(\epsilon^{2}).
\end{split}
\end{equation*}
We have the abbreviated kernel function for each term on the RHS of \eqref{ee} and \eqref{eee}
\begin{equation}
\begin{split}\label{kernel}
\epsilon^{2}\tilde{b}_{j,p}^{0,1,\pm,\mu,\nu}(k)
&=\epsilon^{2}b_{j,p}^{0,1,\pm}(k)\widehat{P}_{1}(k-\mu k_{0})i(k-\mu k_{0})\widehat{\vartheta}^{-1}(k-\mu k_{0})\zeta_{p,q}^{\mp}(k-\mu k_{0})
\widehat{\vartheta}(k-(\mu+\nu)k_{0})\\
&=\epsilon^{2}\frac{\widehat{P}_{0}(k)k\widehat{P}_{1}(k-\mu k_{0})i(k-\mu k_{0})\zeta_{j,p}^{\pm}(k)\zeta_{p,q}^{\mp}(k-\mu k_{0})}{j\omega(k)
\mp\omega(\mu k_{0})-p\omega(k-\mu k_{0})}\frac{\widehat{\vartheta}(k-(\mu+\nu) k_{0})}{2\widehat{\vartheta}(k)}.
\end{split}
\end{equation}
Note that $\tilde{b}_{j,p}^{0,1,\pm,\mu,\nu}(k)$ is $\mathcal{O}(1)$ bounded when $\mu+\nu=0$ since the factor $\widehat{\vartheta}(k)^{-1}$ cancels with the factor $\widehat{\vartheta}(k)$ in the kernel \eqref{kernel}. Then the first two of the terms in \eqref{ee} and \eqref{eee} are $\mathcal{O}(\epsilon^{2})$. We have the following lemma.
\begin{lemma}\label{L7}
There exists $C>0$ such that
\begin{align*}
&\left\|\epsilon^{2}B_{j,p}^{0,1,+}
\left(\psi_{1},P_{1}\displaystyle\frac{1}{2\vartheta}\partial_{x}
\left(N^{+}_{p,q}(\psi_{-1},\vartheta \widetilde{R}_{q}^{1})
+N^{-}_{p,q}(\phi_{-1},\vartheta \widetilde{R}_{q}^{1})\right)\right)
\right\|_{H^{s}}
\leq C\epsilon^{2}\left\|\widetilde{R}_{q}^{1}\right\|_{H^{s}},\\
&\left\|\epsilon^{2}B_{j,p}^{0,1,+}
\left(\psi_{-1},P_{1}\displaystyle\frac{1}{2\vartheta}\partial_{x}
\left(N^{+}_{p,q}(\psi_{1},\vartheta \widetilde{R}_{q}^{1})
+N^{-}_{p,q}(\phi_{1},\vartheta \widetilde{R}_{q}^{1})\right)\right)\right\|_{H^{s}}
\leq C\epsilon^{2}\left\|\widetilde{R}_{q}^{1}\right\|_{H^{s}},\\
&\left\|\epsilon^{2}B_{j,p}^{0,1,-}
\left(\phi_{1},P_{1}\displaystyle\frac{1}{2\vartheta}\partial_{x}
\left(N^{+}_{p,q}(\psi_{-1},\vartheta \widetilde{R}_{q}^{1})
+N^{-}_{p,q}(\phi_{-1},\vartheta \widetilde{R}_{q}^{1})\right)\right)
\right\|_{H^{s}}
\leq C\epsilon^{2}\left\|\widetilde{R}_{q}^{1}\right\|_{H^{s}},\\
&\left\|\epsilon^{2}B_{j,p}^{0,1,-}
\left(\phi_{-1},P_{1}\displaystyle\frac{1}{2\vartheta}\partial_{x}
\left(N^{+}_{p,q}(\psi_{1},\vartheta \widetilde{R}_{q}^{1})
+N^{-}_{p,q}(\phi_{1},\vartheta \widetilde{R}_{q}^{1})\right)\right)\right\|_{H^{s}}
\leq C\epsilon^{2}\left\|\widetilde{R}_{q}^{1}\right\|_{H^{s}}.
\end{align*}
\end{lemma}

Lemma \ref{L7} does not imply that the last two terms on the RHS of \eqref{ee} and \eqref{eee} are of order $\mathcal{O}(\epsilon^2)$. In fact, we can simplify the kernel of these four terms with the help of Lemma \ref{L6}. For example, the kernel of the third term from \eqref{ee} has the form
 \begin{equation}
\begin{split}
\epsilon^{2}b_{j,p}^{0,1,+}(k)\widehat{P}_{1}(k-k_{0})i(k- k_{0})\vartheta^{-1}(k-k_{0})\left(\zeta_{p,q}^{+}(k-k_{0})
+\zeta_{p,q}^{-}(k-k_{0})\right)
\widehat{\vartheta}(k-2k_{0})
\end{split}
\end{equation}
plus errors that are of size $\mathcal{O}(\epsilon^{2})$. Different from the terms considered in Lemma \ref{L7}, the factor of $\widehat{\vartheta}(k)$ is not contained in the numerator of this expression. Thus the $\widehat{\vartheta}(k)$ in the denominator of $b_{j,p}^{0,1,\pm}(k)$ cannot be offset. Therefore, we need to use a second normal-form transformations to eliminate the last two terms on the RHS of \eqref{ee} and \eqref{eee}.

Let
\begin{equation}
\begin{split}\label{equation88}
\mathcal{R}_{j}^{0}=&\widetilde{R}_{j}^{0}+\epsilon\sum_{p,q\in\{\pm1\}} \Big(D_{j,p,q}^{1,+}(\psi_{1},\psi_{1},\widetilde{R}_{q}^{1})
+T_{j,p,q}^{1,+}(\psi_{1},\phi_{1},\widetilde{R}_{q}^{1}) +D_{j,p,q}^{1,-}(\psi_{-1},\psi_{-1},\widetilde{R}_{q}^{1})\\
& \ \ \ \ \ \ \ \ \ \ \ \ \ \ \ \ +T_{j,p,q}^{1,-}(\psi_{-1},\phi_{-1},\widetilde{R}_{q}^{1})
+D_{j,p,q}^{2,+}(\phi_{1},\phi_{1},\widetilde{R}_{q}^{1})
+T_{j,p,q}^{2,+}(\phi_{1},\psi_{1},\widetilde{R}_{q}^{1})\\
& \ \ \ \ \ \ \ \ \ \ \ \ \ \ \ \ +D_{j,p,q}^{2,-}(\phi_{-1},\phi_{-1},\widetilde{R}_{q}^{1})
+T_{j,p,q}^{2,-}(\phi_{-1},\psi_{-1},\widetilde{R}_{q}^{1})
\Big),\\
\mathcal{R}_{j}^{1}=&\widetilde{R}_{j}^{1}.
\end{split}
\end{equation}
Differentiating the expression for $\mathcal{R}_{j}^{0}$, we find that the terms of $\mathcal{\mathcal{O}(\epsilon)}$ in \eqref{ee} will be eliminated if $D_{j,p,q}^{1,+}(\psi_{1},\psi_{1},\widetilde{R}_{q}^{1})$ satisfies
\begin{equation}
\begin{split}\label{equation89}
j\Omega D_{j,p,q}^{1,+}&(\psi_{1},\psi_{1},\widetilde{R}_{q}^{1})
-D_{j,p,q}^{1,+}(\Omega\psi_{1},\psi_{1},\widetilde{R}_{q}^{1})
-D_{j,p,q}^{1,+}(\psi_{1},\Omega\psi_{1},\widetilde{R}_{q}^{1})\\
&-D_{j,p,q}^{1,+}(\psi_{1},\psi_{1},q\Omega \widetilde{R}_{q}^{1})
=\frac{\epsilon}{2}B_{j,p}^{0,1,+}\left(\psi_{1},P_{1}\displaystyle\frac{1}{2\vartheta}\partial_{x}
N^{+}_{p,q}(\psi_{1},\vartheta \widetilde{R}_{q}^{1})\right),
\end{split}
\end{equation}
then we need to choose
\begin{equation*}
\begin{split}
\epsilon \widehat{D}_{j,p,q}^{1,+}(\psi_{1},\psi_{1},\widetilde{R}_{q}^{1})(k)
=\frac{\epsilon^{2}}{4}\int b_{j,p}^{0,1,+}(k)\widehat{\psi}_{1}(k-\ell)\widehat{P}^{1}(k-k_{0})
\widehat{\vartheta}^{-1}(k-k_{0})(k-k_{0})\\
\times\int\frac{\zeta_{p,q}^{+}(k-k_{0})\widehat{\vartheta}(k-2k_{0})}
{j\omega(k)-\omega(k_{0})-\omega(k_{0})-q\omega(k-2k_{0})}
\widehat{\psi}_{1}(\ell-m) \widehat{\widetilde{R}}_{q}^{1}(m)dmd\ell,
\end{split}
\end{equation*}
with similar expressions for $D_{j,p,q}^{n,\pm}$ and $T_{j,p,q}^{n,\pm}$ with $n=1, \ 2$. Thanks to the compact support of $\widehat{P}^{0}$ in $b_{j,p}^{0,1,+}(k)$ and $\widehat{\psi}_{1}$, we have used $k\approx0, k-\ell\approx k_{0}$ and $\ell-m\approx k_{0}$, so we have $m\approx 2k_{0}$. Note that the numerator is of $\mathcal{O}(\epsilon)$ in the above expression. The denominator satisfies
 \begin{equation*}
\begin{split}
j\omega(k)-\omega(k_{0})-\omega(k_{0})-q\omega(k-2k_{0})\approx -2\omega(k_{0})-q\omega(2k_{0})\neq0 \ \ \text{for} q=\pm1.
\end{split}
\end{equation*}
Therefore, $\epsilon D_{j,p,q}^{1,+}$ is $\mathcal{O}(\epsilon)$-bounded and certainly this normal-form transformation is well-defined and invertible. For any $s>0$, there exists $C$ such that
\begin{equation}
\begin{split}\label{DT}
\left\|\epsilon D_{j,p,q}^{1,+}(\psi_{1},\psi_{1},\widetilde{R}_{q}^{1})\|_{H^{s}}
\leq C\epsilon\|\widetilde{R}^{1}\right\|_{H^{s}}.
\end{split}
\end{equation}
We also use similar procedures to construct and estimate the expressions for $D_{j,p,q}^{n,\pm}$ and $T_{j,p,q}^{n,\pm}$ with $n=1, 2$. Summarizing, we have
\begin{proposition}\label{P2}
Assume $(\widetilde{R}_{j}^{0}, \widetilde{R}_{j}^{1})$ with $j=\pm1$ satisfies \eqref{Rj00}-\eqref{Rj01} and define $(\mathcal{R}_{j}^{0}, \mathcal{R}_{j}^{1})$ via the transformation \eqref{equation88}. Then this transformation maps $(\widetilde{R}_{j}^{0},\widetilde{R}_{j}^{1})\in H^{s}\times H^{s}$ into $(\mathcal{R}_{j}^{0},\mathcal{R}_{j}^{1})\in H^{s}\times H^{s}$ for all $s>0$ and is invertible on its range. Furthermore, if we write the inverse transformations as
\begin{equation*}
\begin{split}
\widetilde{R}_{j}^{0}=\mathcal{R}_{j}^{0}+\epsilon\mathcal{D}^{-1}_{0,1}
(\mathcal{R}_{j}^{0},\mathcal{R}_{j}^{1}), \
\widetilde{R}_{j}^{1}=\mathcal{R}_{j}^{1},
\end{split}
\end{equation*}
then there exist constant $C_{2}$ such that the inverse transformations satisfy the estimates
\begin{equation*}
\begin{split}
\left\|\epsilon\mathcal{D}^{-1}_{0,1}
(\mathcal{R}_{j}^{0},\mathcal{R}_{j}^{1})\right\|_{H^{s}}\leq C_{2}\left(\|\mathcal{R}_{j}^{0}\|_{H^{s}}+\|\mathcal{R}_{j}^{1}\|_{H^{s}}\right).
\end{split}
\end{equation*}
\end{proposition}
\begin{proof}
The transformation \eqref{equation88} is invertible since there are also no loss of derivatives in the transformation \eqref{equation88} thanks to the factor of $P^{0}$.
\end{proof}

Combining these two normal-form transformations \eqref{equ3} with \eqref{equation88} together, we can obtain
\begin{equation}
\begin{split}\label{equation92}
\mathcal{R}_{j}^{0}=&R_{j}^{0}+\epsilon Z^{0}_{j}(R^{0},R^{1}),\\
\mathcal{R}_{j}^{1}=&R_{j}^{1}+\epsilon Z^{1}_{j}(R^{0}),
\end{split}
\end{equation}
where
\begin{align*}
Z^{0}_{j}(R^{1})&=\sum_{p\in\{\pm1\}}\left[B_{j,p}^{0,1,+}(\psi_{c},R_{p}^{1})
+B_{j,p}^{0,1,-}(\phi_{c},R_{p}^{1})\right]\\
&+\sum_{p,q\in\{\pm1\}}\bigg(D_{j,p,q}^{1,+}\Big(\psi_{1},\psi_{1},
R_{q}^{1}+\epsilon\sum_{l\in\{\pm1\}}(B_{q,l}^{1,0,+}(\psi_{c},R_{l}^{0})
+B_{q,l}^{1,0,-}(\phi_{c},R_{l}^{0}))\Big)\\
& \ \ \ \ \ \ \ \ \ \ \ \ \ \ \ \ \ \ +T_{j,p,q}^{1,+}(\psi_{1},\phi_{1},R_{q}^{1}+\epsilon\sum_{l\in\{\pm1\}}(B_{q,l}^{1,0,+}(\psi_{c},R_{l}^{0})
+B_{q,l}^{1,0,-}(\phi_{c},R_{l}^{0}))\Big)\\ & \ \ \ \ \ \ \ \ \ \ \ \ \ \ \ \ \ \ +D_{j,p,q}^{1,-}(\psi_{-1},\psi_{-1},R_{q}^{1}+\epsilon\sum_{l\in\{\pm1\}}(B_{q,l}^{1,0,+}(\psi_{c},R_{l}^{0})
+B_{q,l}^{1,0,-}(\phi_{c},R_{l}^{0}))\Big)\\
& \ \ \ \ \ \ \ \ \ \ \ \ \ \ \ \ \ \ +T_{j,p,q}^{1,-}(\psi_{-1},\phi_{-1},R_{q}^{1}+\epsilon\sum_{l\in\{\pm1\}}(B_{q,l}^{1,0,+}(\psi_{c},R_{l}^{0})
+B_{q,l}^{1,0,-}(\phi_{c},R_{l}^{0}))\Big)\\
& \ \ \ \ \ \ \ \ \ \ \ \ \ \ \ \ \ \ +D_{j,p,q}^{2,+}(\phi_{1},\phi_{1},R_{q}^{1}+\epsilon\sum_{l\in\{\pm1\}}(B_{q,l}^{1,0,+}(\psi_{c},R_{l}^{0})
+B_{q,l}^{1,0,-}(\phi_{c},R_{l}^{0}))\Big)\\
& \ \ \ \ \ \ \ \ \ \ \ \ \ \ \ \ \ \ +T_{j,p,q}^{2,+}(\phi_{1},\psi_{1},R_{q}^{1}+\epsilon\sum_{l\in\{\pm1\}}(B_{q,l}^{1,0,+}(\psi_{c},R_{l}^{0})
+B_{q,l}^{1,0,-}(\phi_{c},R_{l}^{0}))\Big)\\ & \ \ \ \ \ \ \ \ \ \ \ \ \ \ \ \ \ \ +D_{j,p,q}^{2,-}(\phi_{-1},\phi_{-1},R_{q}^{1}+\epsilon\sum_{l\in\{\pm1\}}(B_{q,l}^{1,0,+}(\psi_{c},R_{l}^{0})
+B_{q,l}^{1,0,-}(\phi_{c},R_{l}^{0}))\Big)\\
& \ \ \ \ \ \ \ \ \ \ \ \ \ \ \ \ \ \ +T_{j,p,q}^{2,-}(\phi_{-1},\psi_{-1},R_{q}^{1}+\epsilon\sum_{l\in\{\pm1\}}(B_{q,l}^{1,0,+}(\psi_{c},R_{l}^{0})
+B_{q,l}^{1,0,-}(\phi_{c},R_{l}^{0}))\Big)
\bigg),\\
Z^{1}_{j}(R^{0})=&\sum_{p\in\{\pm1\}}\Big[B_{j,p}^{1,0,+}(\psi_{c},R_{p}^{0})
+B_{j,p}^{1,0,-}(\phi_{c},R_{p}^{0})\Big].
\end{align*}
Using Lemma \ref{L7}, \eqref{equ9},  \eqref{DT} and \eqref{B10}, we have
\begin{equation}
\begin{split}\label{RR}
\left\|\epsilon Z^{0}(R^{0},R^{1})\right\|_{H^{s'}}\lesssim\left\|R\right\|_{H^{s}}, \ \ \
\left\|\epsilon Z^{1}(R^{0})\right\|_{H^{s'}}\lesssim\left\|\epsilon R^{0}\right\|_{H^{s}},\ \ \ \ \ \forall s',s\geq6.
\end{split}
\end{equation}

From Proposition \ref{P1} and \ref{P2}, we have the following proposition
\begin{proposition}\label{P3}
Assume $(R_{j}^{0}, R_{j}^{1})$ with $j=\pm1$ satisfies \eqref{Rj0}-\eqref{Rj1} and define $(\mathcal{R}_{j}^{0}, \mathcal{R}_{j}^{1})$ via the transformation \eqref{equation92}. Then this transformation maps $(R_{j}^{0},R_{j}^{1})\in H^{s}\times H^{s}$ into $(\mathcal{R}_{j}^{0},\mathcal{R}_{j}^{1})\in H^{s}\times H^{s}$ for all $s>6$ and is invertible on its range. Furthermore, if we write the inverse transformations as
\begin{equation*}
\begin{split}
R_{j}^{0}=\mathcal{R}_{j}^{0}+\epsilon\mathcal{Z}^{-1}_{0}
(\mathcal{R}_{j}^{0},\mathcal{R}_{j}^{1}), \
R_{j}^{1}=\mathcal{R}_{j}^{1}+\epsilon\mathcal{Z}^{-1}_{1}
(\mathcal{R}_{j}^{0},\mathcal{R}_{j}^{1}),
\end{split}
\end{equation*}
then there exist constants $C_{3}, \ C_{4}$ such that the inverse transformations satisfy the estimates
\begin{equation*}
\begin{split}
\left\|\epsilon\mathcal{Z}^{-1}_{0}
(\mathcal{R}_{j}^{0},\mathcal{R}_{j}^{1})\right\|_{H^{s}}\leq C_{3}\left(\|\mathcal{R}_{j}^{0}\|_{H^{s}}+\|\mathcal{R}_{j}^{1}\|_{H^{s}}\right),\\
\left\|\epsilon\mathcal{Z}^{-1}_{1}
(\mathcal{R}_{j}^{0},\mathcal{R}_{j}^{1})\right\|_{H^{s}}\leq C_{4}\epsilon\left(\|\mathcal{R}_{j}^{0}\|_{H^{s}}+\|\mathcal{R}_{j}^{1}\|_{H^{s}}\right).
\end{split}
\end{equation*}
\end{proposition}

Using the original system \eqref{Rj0}-\eqref{Rj1} for $(R^{0},R^{1})$ and the transformation \eqref{equation92}, we obtain the transformed system for $(\mathcal{R}^{0},\mathcal{R}^{1})$ as follows
\begin{equation}
\begin{split}\label{Rj11}
\partial_{t}\mathcal{R}_{j}^{0}=&j\Omega\mathcal{R}_{j}^{0}+\epsilon^{2}\mathcal{F}_{j}^{6},\\
\partial_{t}\mathcal{R}_{j}^{1}=&j\Omega\mathcal{R}_{j}^{1}
+\epsilon\sum_{p\in\{\pm1\}}P^{1}\displaystyle\frac{1}{2\vartheta}\partial_{x}
\left(N^{+}_{j,p}(\psi_{c},\vartheta \mathcal{R}_{p}^{1})+
N^{-}_{j,p}(\phi_{c},\vartheta \mathcal{R}_{p}^{1})\right)+\epsilon^{2}\mathcal{F}_{j}^{7},
\end{split}
\end{equation}
where $j\in\{\pm1\}$. The terms $\mathcal{F}_{j}^{6}$ and $\mathcal{F}_{j}^{7}$ have similar estimates \eqref{j2} and \eqref{j3} for $\mathcal{F}_{j}^{2}$ and $\mathcal{F}_{j}^{3}$ respectively, i.e. there is no loss of smooth for $\mathcal{F}_{j}^{6}$, while $\mathcal{F}_{j}^{7}$ loses one derivative. To estimate the error $(\mathcal{R}^{0},\mathcal{R}^{1})$ of the system \eqref{Rj11} in the next section, we need to extract all terms which lose derivatives from $\mathcal{F}_{j}^{7}$. For that purpose, we combine the error equations \eqref{Rj0}-\eqref{Rj1} with the expression of $\mathcal{F}^{1}$ in \eqref{F1} and the normal-form transformation \eqref{equation92} and then rewrite the transformed system \eqref{Rj11} as
\begin{equation}
\begin{split}\label{final}
\partial_{t}\mathcal{R}_{j}^{0}=&j\Omega\mathcal{R}_{j}^{0}+\epsilon^{2}\mathcal{F}_{j}^{6},\\
\partial_{t}\mathcal{R}_{j}^{1}=&j\Omega\mathcal{R}_{j}^{1}
+\epsilon\sum_{p\in\{\pm1\}}\frac{\partial_{x}}{2}
\left(N^{+}_{j,p}(\psi_{c},\mathcal{R}_{p}^{1})+
N^{-}_{j,p}(\phi_{c},\mathcal{R}_{p}^{1})\right)\\
&+\epsilon^{2}\partial_{x}\sum_{p\in\{\pm1\}} \mathcal{A}_{j,p}\mathcal{R}_{p}^{1} +\epsilon^{2}\mathcal{F}_{j}^{8}.
\end{split}
\end{equation}
Here we have ignored $\vartheta$ in front of $\mathcal{R}_{p}^{1}$ since $\mathcal{R}_{p}^{1}=P^{1}\mathcal{R}_{p}$ and $\hat{\vartheta}(m)=1$ on the support of $\widehat{P}^{1}$ in the Fourier space, and operators $\mathcal{A}_{j,p}$ are given by
\begin{equation}
\begin{split}\label{ABn}
\mathcal{A}_{j,p}=&\frac{1}{2}\left(p\varphi_{s_{3}}q
+q\varphi_{s_{4}}
+\frac{jp}{q}(q\varphi_{s_{4}})q
-\frac{j}{q}\varphi_{s_{3}}\right),
\end{split}
\end{equation}
with
\begin{equation}
\begin{split}\label{s34}
&\varphi_{s_{3}}:=\varphi_{s_{1}}
+\frac{\epsilon^{\frac{1}{2}}}{2}(\mathcal{R}_{1}^{1}+\mathcal{R}_{-1}^{1}), \
\varphi_{s_{4}}:=\varphi_{s_{2}}
+\frac{\epsilon^{\frac{1}{2}}}{2}(\mathcal{R}_{1}^{1}-\mathcal{R}_{-1}^{1}),
\end{split}
\end{equation}
where $\psi_{c}, \ \phi_{c}, \ \varphi_{s_{1}}, \ \varphi_{s_{2}}$ are defined in the equation \eqref{pc}. The residual terms $\epsilon^{-\frac{5}{2}}P^{0}Res_{j}(\epsilon\Theta)$ and $\epsilon^{-\frac{5}{2}}P^{1}Res_{j}(\epsilon\Theta)$ are included in $\epsilon^{2}\mathcal{F}^{6}$ and $\epsilon^{2}\mathcal{F}^{8}$, respectively, thanks to \eqref{Aesti-1} in Lemma \ref{L2} and we have
\begin{equation}
\begin{split}\label{J68}
\epsilon^{2}\|(\mathcal{F}_{j}^{6},\mathcal{F}^{8}_{j})\|_{H^{s}}\leq C \epsilon^{2}(1+\|\mathcal{R}\|_{H^{s}}+\epsilon^{3/2}\|\mathcal{R}\|^{2}_{H^{s}}
+\epsilon^{3}\|\mathcal{R}\|^{3}_{H^{s}}), \ \ \ \forall s\geq6.
\end{split}
\end{equation}
That means $\mathcal{F}_{j}^{6}$ and $\mathcal{F}^{8}_{j}$ do not lose derivatives.

Finally, for the sake of clearer presentation in the evolutionary equation \eqref{oel} for $\partial_{t}E_{\ell}$ in Section 8, we extract all terms losing one derivative from $\epsilon\sum_{p\in\{\pm1\}}\frac{\partial_{x}}{2}
\left(N^{+}_{j,p}(\psi_{c},\mathcal{R}_{p}^{1})+
N^{-}_{j,p}(\phi_{c},\mathcal{R}_{p}^{1})\right)$ and combine them with $\epsilon^{2}\partial_{x}\sum_{p\in\{\pm1\}}\mathcal{A}_{j,p}\mathcal{R}_{p}^{1}$ in equation \eqref{final} into the following term $\epsilon\partial_{x}
\sum_{p\in\{\pm1\}}\mathcal{D}_{j,p}\mathcal{R}_{p}^{1}$. Then we can further rewrite the transformed error equations \eqref{final} as
\begin{equation}
\begin{split}\label{finally}
\partial_{t}\mathcal{R}_{j}^{0}=&j\Omega\mathcal{R}_{j}^{0}+\epsilon^{2}\mathcal{F}_{j}^{6},\\
\partial_{t}\mathcal{R}_{j}^{1}=&j\Omega\mathcal{R}_{j}^{1}
+\epsilon\partial_{x}
\sum_{p\in\{\pm1\}}\mathcal{D}_{j,p}\mathcal{R}_{p}^{1}+\epsilon\mathcal{F}_{j}^{9},
\end{split}
\end{equation}
where $\mathcal{D}_{j,p}$ satisfies
\begin{equation}
\begin{split}\label{s56}
\mathcal{D}_{j,p}=\frac{1}{2}
\left(p\varphi_{s_{5}}q
+q\varphi_{s_{6}}
+\frac{jp}{q}(q\varphi_{s_{6}})q
-\frac{j}{q}\varphi_{s_{5}}\right),
\end{split}
\end{equation}
with
\begin{align*}
\varphi_{s_{5}}:=\psi_{c}+\phi_{c}+\epsilon\varphi_{s_{3}}, \ \
\varphi_{s_{6}}:=\psi_{c}-\phi_{c}+\epsilon\varphi_{s_{4}},
\end{align*}
in light of the expression in the equations \eqref{alpha} and \eqref{ABn}. Similar to \eqref{J68}, $\mathcal{F}^{9}_{j}$ has the following estimate
\begin{equation}
\begin{split}\label{J69}
\epsilon\|\mathcal{F}^{9}_{j}\|_{H^{s}}\leq C \epsilon(1+\|R\|_{H^{s}}+\epsilon^{3/2}\|R\|^{2}_{H^{s}}
+\epsilon^{3}\|R\|^{3}_{H^{s}}), \ \ \ \forall s\geq6.
\end{split}
\end{equation}

\section{\textbf{Normal-form transformation for the high frequency terms}}

From the above procedure, we have obtained the transformed system \eqref{final}. To obtain the uniform energy estimate for the error $(\mathcal{R}^{0},\mathcal{R}^{1})$ in the time of order $\mathcal{O}(\epsilon^{-2})$, the high frequency terms of $\mathcal{O}(\epsilon)$ still have to be eliminated by the corresponding normal-form transformation $B^{1,1}$ below. However, because of the regularity problems caused by quasilinear quadratic terms, we cannot apply $B^{1,1}$ directly to eliminate the high frequency terms of $\mathcal{O}(\epsilon)$ from \eqref{final} as done in the above section. In the next section, we will use the transformation $B^{1,1}$ to define a modified energy function $\mathcal{E}_{s}$ and show that $\mathcal{E}_{s}$ is equivalent to $\|\mathcal{R}\|^{2}_{H^{s}}$ by analyzing the properties of the normal-form transformation $B^{1,1}$ in this section.

Similar to the form of the normal-form transformation $B^{0,1}$ and $B^{1,0}$, we define
\begin{equation}
\begin{split}\label{B11}
\widehat{B}_{j,p}^{1,1,+}(\psi_{c},\mathcal{R}_{p}^{1})(k)
=\int b_{j,p}^{1,1,+}(k,k-m,m)\widehat{\psi}_{c}(k-m)\widehat{\mathcal{R}}_{p}^{1}(m)dm,\\
\widehat{B}_{j,p}^{1,1,-}(\phi_{c},\mathcal{R}_{p}^{1})(k)
=\int b_{j,p}^{1,1,-}(k,k-m,m)\widehat{\phi}_{c}(k-m)\widehat{\mathcal{R}}_{p}^{1}(m)dm,
\end{split}
\end{equation}
with
\begin{equation*}
\begin{split}
b_{j,p}^{1,1,\pm}=&\frac{1}{2}\frac{k\zeta_{j,p}^{\pm}(k,k-m,m)}
{j\omega(k)\mp\omega(k-m)-p\omega(m)}.
\end{split}
\end{equation*}
Here we have removed the factor $\vartheta$ since the support of $\widehat{P}^{1}$ restricts us to consider in the region $|k|>\delta$, $|m|>\delta$ and $|k-m\pm k_{0}|<\delta$ and hence $\widehat{\vartheta}(k)=\widehat{\vartheta}(m)=1$. Note that $B_{j,p}^{1,1,\pm}$ loses one derivative due to the growth of $k\zeta_{j,p}^{\pm}\sim k$ as $|k|\rightarrow\infty$, i.e. we can only bound $B_{j,p}^{1,1,\pm}$ in $H^{s-1}$ rather than $H^{s}$ if $R_{p}^{1}\in H^{s}$. In what follows we analyze carefully the construction of the normal-form transformation $B_{j,p}^{1,1,\pm}$.
\begin{lemma}\label{L8}
The operators $B_{j,p}^{1,1,\pm}$ have the following properties.
\begin{description}
  \item[(a)] Fix $h\in L^{2}(\mathbb{R},\mathbb{R})$, then the mapping $f\rightarrow B_{j,j}^{1,1,\pm}(h,f)$ defines a continuous linear map from $H^{1}(\mathbb{R},\mathbb{R})$ into $L^{2}(\mathbb{R},\mathbb{R})$ and $f\rightarrow B_{j,-j}^{1,1,\pm}(h,f)$ defines a continuous linear map from $L^{2}(\mathbb{R},\mathbb{R})$ into $L^{2}(\mathbb{R},\mathbb{R})$. In particular, for all $f\in H^{1}(\mathbb{R},\mathbb{R})$ we have
\begin{equation}
\begin{split}\label{15}
&B_{j,j}^{1,1,\pm}(h,f)=\partial_{x}(G^{\pm}_{j,j}h \ f)+Q^{\pm}_{j,j}(h,f),\\
&B_{j,-j}^{1,1,\pm}(h,f)=G^{\pm}_{j,-j}h \ f+Q^{\pm}_{j,-j}(h,f),
\end{split}
\end{equation}
where
\begin{equation}
\begin{split}\label{GQ}
&\widehat{G^{\pm}_{j,j}h}(k)=\pm\widehat{q}(k_{0})\frac{\chi(k)}
{i(jk\mp\omega(k))}\widehat{h}(k),\\
&\widehat{G^{\pm}_{j,-j}h}(k)=-\chi(k)\widehat{h}(k),\\
&\|Q^{\pm}_{j,\pm j}(h,f)\|_{H^{1}}\leq C\|(h,f)\|_{L^{2}}.
\end{split}
\end{equation}
  \item[(b)] For all $f\in H^{1}(\mathbb{R},\mathbb{R})$ we have
\begin{equation}
\begin{split}\label{16}
-j\Omega B_{j,p}^{1,1,+}(\psi_{c},f)+B_{j,p}^{1,1,+}(\Omega\psi_{c},f)
+B_{j,p}^{1,1,+}(\psi_{c},p\Omega f)
=-\frac{P^{1}\partial_{x}}{2\vartheta}N_{j,p}^{+}(\psi_{c},f),\\
-j\Omega B_{j,p}^{1,1,-}(\phi_{c},f)-B_{j,p}^{1,1,-}(\Omega\phi_{c},f)
+B_{j,p}^{1,1,-}(\phi_{c},p\Omega f)
=-\frac{P^{1}\partial_{x}}{2\vartheta}N_{j,p}^{-}(\phi_{c},f).
\end{split}
\end{equation}
  \item[(c)] For all $f,g,h\in H^{1}(\mathbb{R},\mathbb{R})$ we have
\begin{equation}
\begin{split}\label{17}
\int fB_{j,j}^{1,1,\pm}(h,g)dx=-\int B_{j,j}^{1,1,\pm}(h,f)gdx
+\int S^{\pm}_{j,j}(\partial_{x}h,f)gdx,
\end{split}
\end{equation}
where
\begin{equation*}
\begin{split}
&\widehat{S}_{j,j}^{\pm}(\partial_{x}h,f)g(k)=\int s_{j,j}^{\pm}(k,k-m,m)
\widehat{\partial_{x}h}(k-m)\widehat{f}(m)dm,\\
&s_{j,j}^{\pm}(k,k-m,m)
=\frac{k\zeta_{j,j}(k,k-m,m)-m\zeta_{j,j}(m,k-m,k))
}{2i(k-m)(j\omega(k)\mp\omega(k-m)-j\omega(m))}.
\end{split}
\end{equation*}
We have
\begin{equation}
\begin{split}\label{18}
S_{j,j}^{\pm}(\partial_{x}h,f)=G^{\pm}_{j,j}\partial_{x}h f+\widetilde{Q}^{\pm}_{j,j}(\partial_{x}h,f),
\end{split}
\end{equation}
with
\begin{equation}
\begin{split}\label{18=2}
\|\widetilde{Q}^{\pm}_{j,j}(\partial_{x}h,f)\|_{H^{s}}\leq C\|(h,f)\|_{L^{2}}.
\end{split}
\end{equation}
\end{description}
\end{lemma}

\begin{proof}
Because of the support properties of $\widehat{\psi}_{c}, \widehat{\phi}_{c}$ and $\widehat{P}^{1}$, only the kernel function $b^{1,1,\pm}_{j,p}$ of the normal-form transformation $B^{1,1,\pm}_{j,p}$ has to be analysed. For $|k-m\pm k_{0}|<\delta$, $|k|>\delta$ and $|m|>\delta$, we have
\begin{equation*}
\begin{split}
|j\omega(k)\mp\omega(k-m)-p\omega(m)|\geq C,
\end{split}
\end{equation*}
for some constant $C>0$, implying that $|b_{j,p}^{1,1,\pm}(k,k-m,m)|<\infty$.
Next, we analyze the asymptotic behavior of the $b_{j,p}^{1,1,n}(k,k-m,m)$ for $|k|\rightarrow\infty$. According to the form of the dispersion relation $\omega(k)$ in \eqref{equation3} and the expression for $\zeta^{\pm}_{j,p}(k,k-m,m)$ in \eqref{alpha}, we have
\begin{equation}
\begin{split}\label{omega}
\omega(k)=&kq(k)=k+\mathcal{O}(|k|^{-1}), \\
\omega'(k)=&1+\mathcal{O}(|k|^{-2}),\\
\zeta^{\pm}_{j,j}(k,k-m,m)=&j\hat{q}(m)\pm\hat{q}(k-m) \pm\frac{1}{\hat{q}(k)}\hat{q}(k-m)\hat{q}(m)
-\frac{j}{\hat{q}(k)}-\frac{j}{\hat{q}(k)}\frac{1}{\langle k\rangle^{2}}
\frac{1}{\langle k-m\rangle^{2}}\frac{1}{\langle m\rangle^{2}}\\
=&\pm2\widehat{q}(k_{0})+\mathcal{O}(|k|^{-2}),\\
\zeta^{\pm}_{j,-j}(k,k-m,m)=&-j\hat{q}(m)\pm\hat{q}(k-m) \mp\frac{1}{\hat{q}(k)}\hat{q}(k-m)\hat{q}(m) -\frac{j}{\hat{q}(k)}-\frac{j}{\hat{q}(k)}\frac{1}{\langle k\rangle^{2}}
\frac{1}{\langle k-m\rangle^{2}}\frac{1}{\langle m\rangle^{2}}\\
=&-2j+\mathcal{O}(|k|^{-2}),
\end{split}
\end{equation}
for $|k|\rightarrow\infty$.
By the mean value theorem we get
\begin{equation*}
\begin{split}
b_{j,j}^{1,1,\pm}(k,k-m,m)=&\frac{1}{2}\frac{k\zeta_{j,j}^{\pm}(k)\chi(k-m)}
{j(\omega(k)-\omega(m))\mp\omega(k-m)}\\
=&\frac{1}{2}\frac{k\zeta_{j,j}^{\pm}(k)\chi(k-m)}{j(k-m)\omega'(k-\theta(k-m))\mp\omega(k-m)},
\end{split}
\end{equation*}
for some $\theta\in[0,1]$. Using again the fact that supp$\chi$ is compact, we conclude with the help of the expressions \eqref{omega} that
\begin{align*}
b_{j,j}^{1,1,\pm}(k,k-m,m)
=&\frac{1}{2}\frac{(\pm2\widehat{q}(k_{0})k+\mathcal{O}(|k|^{-1}))\chi(k-m)}{j(k-m)(1+\mathcal{O}(|k|^{-2}))\mp\omega(k-m)}\\
=&\frac{\pm\widehat{q}(k_{0})k\chi(k-m)}{j(k-m)\mp\omega(k-m)}+\mathcal{O}(|k|^{-1}),  \ \ \text{for} \  |k|\rightarrow\infty.
\end{align*}
Exploiting once more the compactness of supp$\chi$ as well as \eqref{omega} yields
\begin{equation*}
\begin{split}
b_{j,-j}^{1,1,\pm}(k,k-m,m)=&\frac{1}{2}\frac{k\zeta_{j,-j}^{\pm}(k)\chi(k-m)}
{(j(\omega(k)+\omega(m))\mp\omega(k-m))}\\
=&\frac{1}{2}\frac{k(-2j+\mathcal{O}(|k|^{-2}))\chi(k-m)}{j\omega(k)(1+\mathcal{O}(|k|^{-1}))}\\
=&-\chi(k-m)+\mathcal{O}(|k|^{-2})
\end{split}
\end{equation*}
for $|k|\rightarrow\infty$.
These asymptotic expansions of the $b_{j_{1}j_{2}}^{1,1,n}(k,k-m,m)$ imply \eqref{15}. Finally, since
\begin{equation*}
\begin{split}
b_{j,p}^{1,1,\pm}(-k,-(k-m),-m)=b_{j,p}^{1,1,\pm}(k,k-m,m)\in\mathbb{R}
\end{split}
\end{equation*}
and $\psi_{c}$ and $\phi_{c}$ are real-valued, we obtain the validity of all assertions of (a).

(b) is a direct consequence of the construction of the operators $B_{j,p}^{1,1,\pm}$.

In order to prove (c), we compute for all $f, g, h\in H^{1}(\mathbb{R},\mathbb{R})$ that
\begin{align*}
(f,&B_{j,j}^{1,1,\pm}(h,g))\\
=&\int\overline{\widehat{f}(k)}\widehat{B}_{j,j}^{1,1,\pm}(h,g)(k)dk\\
=&\frac{1}{2}\int\int\overline{\widehat{f}(k)}\frac{k\zeta^{\pm}_{j,j}(k,k-m,m)}
{j\omega(k)\mp\omega(k-m)-j\omega(m)}\widehat{h}(k-m)\widehat{g}(m)dmdk\\
=&\frac{1}{2}\int\int\overline{\widehat{g}(-m)}\frac{k\zeta^{\pm}_{j,j}(k,k-m,m)}
{j\omega(k)\mp\omega(k-m)-j\omega(m)}\widehat{h}(k-m)\widehat{f}(-k)dkdm\\
=&\frac{1}{2}\int\int\overline{\widehat{g}(k)}
\frac{-m\zeta^{\pm}_{j,j}(-m,k-m,-k)}
{j\omega(k)\mp\omega(k-m)-j\omega(m)}\widehat{h}(k-m)\widehat{f}(m)dmdk\\
=&\frac{1}{2}\int\int\overline{\widehat{g}(k)}\frac{-k\zeta^{\pm}_{j,j}(k,k-m,m)}
{j\omega(k)\mp\omega(k-m)-j\omega(m)}\widehat{h}(k-m)\widehat{f}(m)dmdk\\
&+\frac{1}{2}\int\int\overline{\widehat{g}(k)}
\frac{k\zeta^{\pm}_{j,j}(k,k-m,m)-m\zeta^{\pm}_{j,j}(m,k-m,k)}
{j\omega(k)\mp\omega(k-m)-j\omega(m)}\widehat{h}(k-m)\widehat{f}(m)dmdk\\
=&-\int_{\mathbb{R}}\overline{\widehat{g}(k)}\widehat{B}_{j,j}^{1,1,\pm}(h,f)(k)dk
+\int_{\mathbb{R}}\overline{\widehat{g}(k)}\widehat{S}_{j,j}^{\pm}(\partial_{x}h,f)(k)dk\\
=&-(g,B_{j,j}^{1,1,\pm}(h,f))+(g,S_{j,j}^{\pm}(\partial_{x}h,f)).
\end{align*}
It then follows \eqref{17}.
\end{proof}
We also need the following identities to control the time evolution of $\mathcal{E}_{s}$ defined in the next section.
\begin{lemma}\label{L9}
Let $j\in\{\pm1\}$, $a_{j}\in H^{2}(\mathbb{R},\mathbb{R})$ and $f_{j}\in H^{1}(\mathbb{R},\mathbb{R})$. Then we have
\begin{equation}
\begin{split}\label{part1}
&\int_{\mathbb{R}}a_{j}f_{j}\partial_{x}f_{j}dx=-\frac{1}{2}\int_{\mathbb{R}}\partial_{x}a_{j}f_{j}^{2}dx,
\end{split}
\end{equation}
\begin{align}\label{part2}
\sum _{j\in\{\pm1\}}\int_{\mathbb{R}}a_{j}f_{j}\partial_{x}f_{-j}dx =&\frac{1}{2}\int_{\mathbb{R}}(a_{-1}-a_{1})(f_{1}+f_{-1})\partial_{x}(f_{1}-f_{-1})dx\\
&+\mathcal{O}((\|a_{1}\|_{H^{2}(\mathbb{R},\mathbb{R})}+\|a_{-1}\|_{H^{2}(\mathbb{R},\mathbb{R})})
(\|f_{1}\|_{L^{2}(\mathbb{R},\mathbb{R})} +\|f_{-1}\|_{L^{2}(\mathbb{R},\mathbb{R})}),\nonumber
\end{align}
\begin{equation}
\begin{split}\label{part5}
\widehat{q}(k)-1=&\frac{1}{\sqrt{1+k^{2}}(\sqrt{1+k^{2}}+\sqrt{2+k^{2}})}=\mathcal{O}(k^{-2}),\\
\frac{1}{\widehat{q}(k)}-1=&\frac{1}{\widehat{q}(k)}(1-\widehat{q}(k))=\mathcal{O}(k^{-2}),
\end{split}
\end{equation}
and
\begin{equation}
\begin{split}\label{part4}
(\widehat{G}^{+}_{1,1}+\widehat{G}^{+}_{-1,-1})(k)=(\widehat{G}^{-}_{1,1}+\widehat{G}^{-}_{-1,-1})(k)
=2\widehat{q}(k_{0})(-\frac{1}{ik}+ik)\widehat{q}(k),
\end{split}
\end{equation}
\begin{equation}
\begin{split}\label{part6}
\widehat{G}_{1,-1}^{\pm}(k)=\widehat{G}^{\pm}_{-1,1}(k)=-1, \ G_{1,-1}^{\pm}-G^{\pm}_{-1,1}=0.
\end{split}
\end{equation}
\end{lemma}
\begin{proof}
The proof of \eqref{part1}-\eqref{part5} refers to Lemma 5.3 in \cite{LP19}. The last two \eqref{part4}-\eqref{part6} can be obtained in light of the form of $G_{j,\pm j}$ in \eqref{GQ}.
\end{proof}

\section{\textbf{The error estimates}}
In this section we define a modified energy $\mathcal{E}_{s}$ by applying the normal-form transformation $B^{1,1}$ defined in \eqref{B11} to control the error $(\mathcal{R}^{0},\mathcal{R}^{1})$ for the time of order $\mathcal{O}(\epsilon^{-2})$.

Hence, Theorem \ref{Thm2} is proved if we can show that $\mathcal{E}_{s}$ can be bounded, and then Theorem \ref{Thm1} is also proved by using the relationship \eqref{rho} between the Euler-Poisson system \eqref{EP} and the diagonalized system \eqref{equation7}.

Specifically, we define
\begin{equation}
\begin{split}\label{es}
\mathcal{E}_{s}=\sum_{\ell=0}^{s}E_{\ell},
\end{split}
\end{equation}
with
\begin{equation}
\begin{split}\label{equ14}
E_{\ell}=&\sum_{j\in\{\pm1\}}\Bigg\{\frac{1}{2}\Big(\int
(\partial_{x}^{\ell}\mathcal{R}_{j}^{0})^{2}dx+
\int(\partial_{x}^{\ell}\mathcal{R}_{j}^{1})^{2}dx\Big)\\
&+\epsilon\sum_{p\in\{\pm1\}}\bigg(
\int\partial_{x}^{\ell}\mathcal{R}_{j}^{1}
\cdot\partial_{x}^{\ell}\Big(B_{j,p}^{1,1,+}(\psi_{c},\mathcal{R}_{p}^{1})
+B_{j,p}^{1,1,-}(\phi_{c},\mathcal{R}_{p}^{1})dx\Big)\bigg)\Bigg\},
\end{split}
\end{equation}
where $B^{1,1,\pm}$ is defined in \eqref{B11}. 

We will show that $\mathcal{E}_{s}$ is equivalent to $\|R\|^{2}_{H^{s}}$.
\begin{proposition}\label{C1}
$\sqrt{\mathcal{E}_{s}}$ is equivalent to $\|R\|_{H^{s}}$ for sufficiently small $\epsilon>0$.
\end{proposition}
\begin{proof}
For $s=0$, we obtain by using \eqref{15} of (a) in Lemma \ref{L8} and integration by parts
\begin{align*}
\mathcal{E}_{0}=E_{0}=&\sum_{j\in\{\pm1\}}\Bigg\{\frac{1}{2}\Big(\int
(\mathcal{R}_{j}^{0})^{2}dx+
\int(\mathcal{R}_{j}^{1})^{2}dx\Big)\\
&+\epsilon\sum_{p\in\{\pm1\}}\bigg(
\int\mathcal{R}_{j}^{1}
\cdot\Big(B_{j,p}^{1,1,+}(\psi_{c},\mathcal{R}_{p}^{1})
+B_{j,p}^{1,1,-}(\phi_{c},\mathcal{R}_{p}^{1})dx\Big)\bigg)\Bigg\}\\
=&\sum_{j\in\{\pm1\}}\Bigg\{\frac{1}{2}\Big(\int
(\mathcal{R}_{j}^{0})^{2}dx+
\int(\mathcal{R}_{j}^{1})^{2}dx\Big)
+\frac{\epsilon}{2}\int(\mathcal{R}_{j}^{1})^{2}
\cdot\partial_{x}(G_{j,j}^{+}\psi_{c}+G_{j,j}^{-}\phi_{c})dx\\
& \ \ \ \ \ \ \ \ \ \ \ +\epsilon\int\mathcal{R}_{j}^{1}
\cdot (Q_{j,j}^{+}(\psi_{c},\mathcal{R}_{j}^{1})+Q_{j,j}^{-}(\phi_{c},\mathcal{R}_{j}^{1}))dx\\
& \ \ \ \ \ \ \ \ \ \ \ +\epsilon\int\mathcal{R}_{j}^{1}
\cdot(G_{j,-j}^{+}\psi_{c}+G_{j,-j}^{-}\phi_{c})\cdot\mathcal{R}_{-j}^{1}dx\\
& \ \ \ \ \ \ \ \ \ \ \ +\epsilon\int\mathcal{R}_{j}^{1}
\cdot(Q_{j,-j}^{+}(\psi_{c},\mathcal{R}_{-j}^{1})+Q_{j,-j}^{-}(\phi_{c},\mathcal{R}_{-j}^{1}))dx\Bigg\}\\
\leq&\sum_{j\in\{\pm1\}}\Bigg\{\frac{1}{2}\left\|\mathcal{R}_{j}^{0}\right\|^{2}_{L^{2}}
+\frac{1}{2}\left\|\mathcal{R}_{j}^{1}\right\|^{2}_{L^{2}}
+\frac{\epsilon}{2}\left\|\mathcal{R}_{j}^{1}\right\|^{2}_{L^{2}}
\left\|\partial_{x}(G_{j,j}^{+}\psi_{c}+G_{j,j}^{-}\phi_{c})\right\|_{L^{\infty}}\\
& \ \ \ \ \ \ \ \ \ \ \ +\epsilon\left\|\mathcal{R}_{j}^{1}\right\|_{L^{2}}
\left\|Q_{j,j}^{+}(\psi_{c},\mathcal{R}_{j}^{1})
+Q_{j,j}^{-}(\phi_{c},\mathcal{R}_{j}^{1})\right\|_{L^{2}}\\
& \ \ \ \ \ \ \ \ \ \ \ +\epsilon\left\|\mathcal{R}_{j}^{1}\right\|_{L^{2}}\left\|\mathcal{R}_{-j}^{1}\right\|_{L^{2}}
\left\|G_{j,-j}^{+}\psi_{c}+G_{j,-j}^{-}\phi_{c}\right\|_{L^{\infty}}\\
& \ \ \ \ \ \ \ \ \ \ \ +\epsilon\left\|\mathcal{R}_{j}^{1}\right\|_{L^{2}}
\left\|Q_{j,-j}^{+}(\psi_{c},\mathcal{R}_{-j}^{1})
+Q_{j,-j}^{-}(\phi_{c},\mathcal{R}_{-j}^{1})\right\|_{L^{2}}\Bigg\}\\
\leq&\sum_{j\in\{\pm1\}}\Bigg\{\frac{1}{2}\left\|\mathcal{R}_{j}^{0}\right\|^{2}_{L^{2}}
+\frac{1}{2}\left\|\mathcal{R}_{j}^{1}\right\|^{2}_{L^{2}}
+\frac{\epsilon}{2}\left\|\mathcal{R}_{j}^{1}\right\|^{2}_{L^{2}}
\left\|G_{j,j}^{+}\psi_{c}+G_{j,j}^{-}\phi_{c}\right\|_{H^{2}}\\
& \ \ \ \ \ \ \ \ \ \ \ +\epsilon\left\|\mathcal{R}_{j}^{1}\right\|^{2}_{L^{2}}
+\epsilon\left\|Q_{j,j}^{+}(\psi_{c},\mathcal{R}_{j}^{1})
+Q_{j,j}^{-}(\phi_{c},\mathcal{R}_{j}^{1})\right\|^{2}_{L^{2}}\\
& \ \ \ \ \ \ \ \ \ \ \ +\epsilon\left\|\mathcal{R}_{-j}^{1}\right\|^{2}_{L^{2}}
\left\|G_{j,-j}^{+}\psi_{c}+G_{j,-j}^{-}\phi_{c}\right\|^{2}_{H^{1}}\\
& \ \ \ \ \ \ \ \ \ \ \ +\epsilon\left\|Q_{j,-j}^{+}(\psi_{c},\mathcal{R}_{-j}^{1})
+Q_{j,-j}^{-}(\phi_{c},\mathcal{R}_{-j}^{1})\right\|^{2}_{L^{2}}\Bigg\}\\
\leq&\frac{1}{2}\left\|\mathcal{R}^{0}\right\|^{2}_{L^{2}}+\frac{1}{2}\left\|\mathcal{R}^{1}\right\|^{2}_{L^{2}}
+\epsilon C\left\|\mathcal{R}^{1}\right\|^{2}_{L^{2}},
\end{align*}
where we have used the Cauchy-Schwarz inequality, Sobolev embedding inequality $H^{1}\hookrightarrow L^{\infty}$, \eqref{GQ}, \eqref{Aesti-3} and Remark 4.3. Thus, we obtain $\sqrt{\mathcal{E}_{0}}$ is equivalent to $\|\mathcal{R}\|_{L^{2}}$ for sufficiently small $\epsilon>0$. The equivalence of $\sqrt{\mathcal{E}_{0}}$ and  $\|R\|_{L^{2}}$ can be obtained by applying the relation \eqref{equation92} between $R$ and $\mathcal{R}$, \eqref{RR} and Proposition \ref{P3}. The other side inequality is obvious.

For $s\geq1$, using \eqref{15} of (a) in Lemma \ref{L8}, integration by parts again and commutator notation \eqref{w}, we have
\begin{align*}
\mathcal{E}_{s}=\sum_{\ell=0}^{s}E_{\ell}
=\sum_{\ell=0}^{s}&\sum_{j\in\{\pm1\}}\Bigg\{\frac{1}{2}\Big(\int
(\partial_{x}^{\ell}\mathcal{R}_{j}^{0})^{2}dx+
\int(\partial_{x}^{\ell}\mathcal{R}_{j}^{1})^{2}dx\Big)\\
&+\epsilon\sum_{p\in\{\pm1\}}\bigg(
\int\partial_{x}^{\ell}\mathcal{R}_{j}^{1}
\cdot\partial_{x}^{\ell}\Big(B_{j,p}^{1,1,+}(\psi_{c},\mathcal{R}_{p}^{1})
+B_{j,p}^{1,1,-}(\phi_{c},\mathcal{R}_{p}^{1})\Big)dx\bigg)\Bigg\}\\
=\sum_{\ell=0}^{s}\sum_{j\in\{\pm1\}}\Bigg\{&\frac{1}{2}\Big(\int
(\partial_{x}^{\ell}\mathcal{R}_{j}^{0})^{2}dx+
\int(\partial_{x}^{\ell}\mathcal{R}_{j}^{1})^{2}dx\Big)\\
&-\frac{\epsilon}{2}
\int\left(\partial_{x}^{\ell}\mathcal{R}_{j}^{1}\right)^{2}
\cdot\partial_{x}\left(G_{j,j}^{+}\psi_{c}
+G_{j,j}^{-}\phi_{c}\right)dx\\
&+\epsilon
\int\partial_{x}^{\ell}\mathcal{R}_{j}^{1}
\cdot\left[\partial_{x}^{\ell+1},G_{j,j}^{+}\psi_{c}
+G_{j,j}^{-}\phi_{c}\right]\mathcal{R}_{j}^{1}dx\\
&+\epsilon
\int\partial_{x}^{\ell}\mathcal{R}_{j}^{1}
\cdot\partial_{x}^{\ell}\Big(Q_{j,j}^{+}(\psi_{c},\mathcal{R}_{j}^{1})
+Q_{j,j}^{-}(\phi_{c},\mathcal{R}_{j}^{1})\Big)dx\\
&+\epsilon
\int\partial_{x}^{\ell}\mathcal{R}_{j}^{1}
\cdot\partial_{x}^{\ell}\mathcal{R}_{-j}^{1}\cdot\left(G_{j,-j}^{+}\psi_{c}
+G_{j,-j}^{-}\phi_{c}\right)dx\\
&+\epsilon
\int\partial_{x}^{\ell}\mathcal{R}_{j}^{1}
\cdot\left[\partial_{x}^{\ell},G_{j,-j}^{+}\psi_{c}
+G_{j,-j}^{-}\phi_{c}\right]\mathcal{R}_{-j}^{1}dx\\
&+\epsilon\int\partial_{x}^{\ell}\mathcal{R}_{j}^{1}
\cdot\partial_{x}^{\ell}\Big(Q_{j,-j}^{+}(\psi_{c},\mathcal{R}_{-j}^{1})
+Q_{j,-j}^{-}(\phi_{c},\mathcal{R}_{-j}^{1})\Big)dx\Bigg\}.
\end{align*}
Using the Cauchy-Schwarz inequality, Sobolev embedding inequality $H^{1}\hookrightarrow L^{\infty}$ again and commutator estimate \eqref{w11}, we have
\begin{align*}
\mathcal{E}_{s}
\leq\sum_{j\in\{\pm1\}}\Bigg\{&
\|\mathcal{R}_{j}^{0}\|^{2}_{H^{s}}
+\|\mathcal{R}_{j}^{1}\|^{2}_{H^{s}}
+\epsilon\|\mathcal{R}_{j}^{1}\|^{2}_{H^{s}}
\|G_{j,j}^{+}\psi_{c}
+G_{j,j}^{-}\phi_{c}\|_{H^{2}}\\
&+\epsilon\|\mathcal{R}_{j}^{1}\|^{2}_{H^{s}}
+\|\partial_{x}(G_{j,j}^{+}\psi_{c}
+G_{j,j}^{-}\phi_{c})\|^{2}_{H^{s}}
\|\mathcal{R}_{j}^{1}\|^{2}_{H^{1}}\\
&+\|Q_{j,j}^{+}(\psi_{c},\mathcal{R}_{j}^{1})
+Q_{j,j}^{-}(\phi_{c},\mathcal{R}_{j}^{1})\|^{2}_{H^{s}}\\
&+\|\mathcal{R}_{-j}^{1}\|^{2}_{H^{s}}
\|G_{j,-j}^{+}\psi_{c}
+G_{j,-j}^{-}\phi_{c}\|^{2}_{H^{1}}\\
&+\|\partial_{x}(G_{j,-j}^{+}\psi_{c}
+G_{j,-j}^{-}\phi_{c})\|_{H^{1}}
\|\mathcal{R}_{-j}^{1}\|^{2}_{H^{s-1}}\\
&+\|G_{j,-j}^{+}\psi_{c}
+G_{j,-j}^{-}\phi_{c}\|^{2}_{H^{s}}
\|\mathcal{R}_{-j}^{1}\|^{2}_{H^{1}}\\
&+\|Q_{j,-j}^{+}(\psi_{c},\mathcal{R}_{-j}^{1})
+Q_{j,-j}^{-}(\phi_{c},\mathcal{R}_{-j}^{1})\|^{2}_{H^{s}}\Bigg\}\\
\leq \|\mathcal{R}^{0}\|^{2}_{H^{s}}
&+\|\mathcal{R}^{1}\|^{2}_{H^{s}}
+\epsilon C\|\mathcal{R}^{1}\|^{2}_{H^{s}},
\end{align*}
where we have used \eqref{GQ}, \eqref{Aesti-3} and Remark 4.3 again. The other side inequality is also obvious. Thus, we obtain the equivalence of $\sqrt{\mathcal{E}_{s}}$ and $\|\mathcal{R}\|_{H^{s}}$ for sufficiently small $\epsilon>0$, and then the equivalence of  $\sqrt{\mathcal{E}_{s}}$ and $\|R\|_{H^{s}}$ by applying the relation \eqref{equation92} between $R$ and $\mathcal{R}$, \eqref{RR} and Proposition \ref{P3}.
\end{proof}

Now, we are prepared to analyze $\partial_{t}E_{\ell}$. We compute
\begin{equation*}
\begin{split}
\partial_{t}E_{\ell}=&\sum_{j\in\{\pm1\}}\Bigg\{\int\partial_{x}^{\ell}\mathcal{R}_{j}^{0}\cdot
    \partial_{t}\partial_{x}^{\ell}\mathcal{R}_{j}^{0}dx
    +\int\partial_{x}^{\ell}\mathcal{R}_{j}^{1}\cdot
    \partial_{t}\partial_{x}^{\ell}\mathcal{R}_{j}^{1}dx\\
    & \ \ \ \ \ \ \ \ \ \ +\epsilon\sum_{p\in\{\pm1\}}\Big(\int\partial_{t}\partial_{x}^{\ell}\mathcal{R}_{j}^{1}\cdot
    \partial_{x}^{\ell}(B_{j,p}^{1,1,+}(\psi_{c},\mathcal{R}_{p}^{1})
    +B_{j,p}^{1,1,-}(\phi_{c},\mathcal{R}_{p}^{1}))dx\\
    & \ \ \ \ \ \ \ \ \ \ \ \ \ \ \  \ \ \ \  \ \ \ \  \ \ \ \ +\int\partial_{x}^{\ell}\mathcal{R}_{j}^{1}\cdot
    \partial_{x}^{\ell}(B_{j,p}^{1,1,+}(\partial_{t}\psi_{c},\mathcal{R}_{p}^{1})
    +B_{j,p}^{1,1,-}(\partial_{t}\phi_{c},\mathcal{R}_{p}^{0}))dx\\
    & \ \ \ \ \ \ \ \ \ \ \ \ \ \ \  \ \ \ \  \ \ \ \  \ \ \ \ +\int\partial_{x}^{\ell}\mathcal{R}_{j}^{1}\cdot
    \partial_{x}^{\ell}(B_{j,p}^{1,1,+}(\psi_{c},\partial_{t}\mathcal{R}_{p}^{1})
    +B_{j,p}^{1,1,-}(\phi_{c},\partial_{t}\mathcal{R}_{p}^{1}))dx\Big)
\Bigg\}.
\end{split}
\end{equation*}
According to the transformed equations \eqref{final} and \eqref{finally}, we have
\begin{align}\label{oel}
\partial_{t}&E_{\ell}=\sum_{j\in\{\pm1\}}\bigg\{j\int\partial_{x}^{\ell}\mathcal{R}_{j}^{0}\cdot
\Omega\partial_{x}^{\ell}\mathcal{R}_{j}^{0}dx
+\epsilon^{2}\int\partial_{x}^{\ell}\mathcal{R}_{j}^{0}\cdot\partial_{x}^{\ell}\mathcal{F}_{j}^{6}dx
+j\int\partial_{x}^{\ell}\mathcal{R}_{j}^{1}\cdot
\Omega\partial_{x}^{\ell}\mathcal{R}_{j}^{1}dx      \\
& \ \ \ \ \ \ \ \ \ \ \ \ \ \ \ \ \ \ +\epsilon^{2}\int\partial_{x}^{\ell}\mathcal{R}_{j}^{1}\cdot\partial_{x}^{\ell}\mathcal{F}^{8}_{j}dx
+\epsilon^{2}\sum_{n\in\{\pm1\}}\int\partial_{x}^{\ell}
\mathcal{R}_{j}^{1}\cdot\partial_{x}^{\ell+1}(\mathcal{A}_{j,n}\mathcal{R}_{n}^{1})dx\bigg\}\nonumber\\
&+\epsilon\sum_{j,p\in\{\pm1\}}\bigg\{\frac{1}{2}\int\partial_{x}^{\ell}\mathcal{R}_{j}^{1}\cdot
\partial_{x}^{\ell+1}N^{+}_{j,p}(\psi_{c},\mathcal{R}_{p}^{1})dx
+\int j\Omega\partial_{x}^{\ell}\mathcal{R}_{j}^{1}\cdot
\partial_{x}^{\ell}B_{j,p}^{1,1,+}(\psi_{c},\mathcal{R}_{p}^{1})dx\nonumber\\
& \ \ \ \ \ \ \ \ \ \ \ \ \ \ \ \ \ +\int\partial_{x}^{\ell}\mathcal{R}_{j}^{1}\cdot
\partial_{x}^{\ell}B_{j,p}^{1,1,+}(\Omega\psi_{c},\mathcal{R}_{p}^{1})dx
+\int\partial_{x}^{\ell}\mathcal{R}_{j}^{1}\cdot
\partial_{x}^{\ell}B_{j,p}^{1,1,+}(\psi_{c},p\Omega\mathcal{R}_{p}^{1})dx\bigg\}\nonumber\\
&+\epsilon\sum_{j,p\in\{\pm1\}}\bigg\{\frac{1}{2}\int\partial_{x}^{\ell}\mathcal{R}_{j}^{1}\cdot
\partial_{x}^{\ell+1}N^{-}_{j,p}(\phi_{c},\mathcal{R}_{p}^{1})dx
+\int j\Omega\partial_{x}^{\ell}\mathcal{R}_{j}^{1}\cdot
\partial_{x}^{\ell}B_{j,p}^{1,1,-}(\phi_{c},\mathcal{R}_{p}^{1})dx\nonumber\\
& \ \ \ \ \ \ \ \ \ \ \ \ \ \ \ \ \ -\int\partial_{x}^{\ell}\mathcal{R}_{j}^{1}\cdot
\partial_{x}^{\ell}B_{j,p}^{1,1,-}(\Omega\phi_{c},\mathcal{R}_{p}^{1})dx
+\int\partial_{x}^{\ell}\mathcal{R}_{j}^{1}\cdot
\partial_{x}^{\ell}B_{j,p}^{1,1,-}(\phi_{c},p\Omega\mathcal{R}_{p}^{1})dx\bigg\}\nonumber\\
&+\epsilon\sum_{j,p\in\{\pm1\}}\bigg\{\int\partial_{x}^{\ell}\mathcal{R}_{j}^{1}\cdot
\partial_{x}^{\ell}\big(B_{j,p}^{1,1,+}(\partial_{t}\psi_{c}-\Omega\psi_{c},\mathcal{R}_{p}^{1})
+B_{j,p}^{1,1,-}(\partial_{t}\phi_{c}+\Omega\phi_{c},\mathcal{R}_{p}^{1})\big)dx\bigg\}\nonumber\\
&+\epsilon^{2}\sum_{j,p\in\{\pm1\}}
\bigg\{\int\partial_{x}^{\ell}\mathcal{F}^{9}_{j}\cdot
\partial_{x}^{\ell}(B_{j,p}^{1,1,+}(\psi_{c},\mathcal{R}_{p}^{1})
    +B_{j,p}^{1,1,-}(\phi_{c},\mathcal{R}_{p}^{1}))dx\nonumber\\
& \ \ \ \ \ \ \ \ \ \ \ \ \ \ \ \ \ \ \ +\int\partial_{x}^{\ell}\mathcal{R}_{j}^{1}\cdot
\partial_{x}^{\ell}(B_{j,p}^{1,1,+}(\psi_{c},\mathcal{F}^{9}_{p})
+B_{j,p}^{1,1,-}(\phi_{c},\mathcal{F}^{9}_{p}))dx\bigg\}\nonumber\\
&+\epsilon^{2}\sum_{j,p,n\in\{\pm1\}}
\bigg\{\int\partial_{x}^{\ell+1}(\mathcal{D}_{j,n}\mathcal{R}_{n}^{1})\cdot
\partial_{x}^{\ell}(B_{j,p}^{1,1,+}(\psi_{c},\mathcal{R}_{p}^{1})
+B_{j,p}^{1,1,-}(\phi_{c},\mathcal{R}_{p}^{1}))dx\nonumber\\
& \ \ \ \ \ \ \ \ \ \ \ \  \ \ \ \ \ \ \ +\int\partial_{x}^{\ell}\mathcal{R}_{j}^{1}\cdot
\partial_{x}^{\ell}(B_{j,p}^{1,1,+}(\psi_{c},\partial_{x}(\mathcal{D}_{p,n}\mathcal{R}_{n}^{1}))
    +B_{j,p}^{1,1,-}(\phi_{c},\partial_{x}(\mathcal{D}_{p,n}\mathcal{R}_{n}^{1})))dx\bigg\}\nonumber.
\end{align}
For the first brace of \eqref{oel}, thanks to the skew symmetry of $\Omega$, we have
\begin{align*}
\int\partial_{x}^{\ell}\mathcal{R}_{j}^{0}\cdot
\Omega\partial_{x}^{\ell}\mathcal{R}_{j}^{0}dx=0, \
\int\partial_{x}^{\ell}\mathcal{R}_{j}^{1}\cdot
\Omega\partial_{x}^{\ell}\mathcal{R}_{j}^{1}dx=0.
\end{align*}
By using \eqref{J68}, we have
\begin{align*}
\epsilon^{2}\int\partial_{x}^{\ell}\mathcal{R}_{j}^{0}\cdot\partial_{x}^{\ell}\mathcal{F}_{j}^{6}dx
&\leq C\epsilon^{2}(1+\|\mathcal{R}\|^{2}_{H^{s}}+\epsilon^{3/2}\|\mathcal{R}\|^{3}_{H^{s}}
+\epsilon^{3}\|\mathcal{R}\|^{4}_{H^{s}})\\
&=\epsilon^{2}\mathcal{O}(1+\mathcal{E}_{s}+\epsilon^{3/2}\mathcal{E}_{s}^{3/2}+\epsilon^{3}\mathcal{E}_{s}^{2}).
\end{align*}
Similarly, the term $\epsilon^{2}\int\partial_{x}^{\ell}\mathcal{R}_{j}^{1}\cdot\partial_{x}^{\ell}\mathcal{F}^{8}_{j}dx$ can be controlled by $C\epsilon^{2}(1+\mathcal{E}_{s}+\epsilon^{3/2}\mathcal{E}_{s}^{3/2}+\epsilon^{3}\mathcal{E}_{s}^{2})$.


We see that the second and third braces of \eqref{oel} vanish by using the equation \eqref{16} in Lemma \ref{L8} for the normal-form transformation $B_{j,p}^{1,1,\pm}$.

For the fourth brace of \eqref{oel}, using \eqref{15} in Lemma \ref{L8}, commutator notation \eqref{w} in Lemma \ref{L10} and integration by parts, we have
\begin{align*}
\epsilon&\sum_{j,p\in\{\pm1\}}\bigg\{\int\partial_{x}^{\ell}\mathcal{R}_{j}^{1}\cdot
\partial_{x}^{\ell}\big(B_{j,p}^{1,1,+}(\partial_{t}\psi_{c}-\Omega\psi_{c},\mathcal{R}_{p}^{1})
+B_{j,p}^{1,1,-}(\partial_{t}\phi_{c}+\Omega\phi_{c},\mathcal{R}_{p}^{1})\big)dx\bigg\}\\
=&\epsilon\sum_{j\in\{\pm1\}}\bigg\{\int\partial_{x}^{\ell}\mathcal{R}_{j}^{1}\cdot
\partial_{x}^{\ell}\big(B_{j,j}^{1,1,+}(\partial_{t}\psi_{c}-\Omega\psi_{c},\mathcal{R}_{j}^{1})
+B_{j,j}^{1,1,-}(\partial_{t}\phi_{c}+\Omega\phi_{c},\mathcal{R}_{j}^{1})\big)dx\\
& \ \ \ \ \ \ \ \ \ \ \ \ +\int\partial_{x}^{\ell}\mathcal{R}_{j}^{1}\cdot
\partial_{x}^{\ell}\big(B_{j,-j}^{1,1,+}(\partial_{t}\psi_{c}-\Omega\psi_{c},\mathcal{R}_{-j}^{1})
+B_{j,-j}^{1,1,-}(\partial_{t}\phi_{c}+\Omega\phi_{c},\mathcal{R}_{-j}^{1})\big)dx\bigg\}\\
=&\epsilon\sum_{j\in\{\pm1\}}\bigg\{\int\partial_{x}^{\ell}\mathcal{R}_{j}^{1}\cdot
\partial_{x}^{\ell+1}\big(G_{j,j}^{+}(\partial_{t}\psi_{c}-\Omega\psi_{c})\mathcal{R}_{j}^{1}
+G_{j,j}^{-}(\partial_{t}\phi_{c}+\Omega\phi_{c})\mathcal{R}_{j}^{1}\big)dx\\
& \ \ \ \ \ \ \ \ \ \ \ \ +\int\partial_{x}^{\ell}\mathcal{R}_{j}^{1}\cdot
\partial_{x}^{\ell}\big(Q_{j,j}^{+}(\partial_{t}\psi_{c}-\Omega\psi_{c},\mathcal{R}_{j}^{1})
+Q_{j,j}^{-}(\partial_{t}\phi_{c}+\Omega\phi_{c},\mathcal{R}_{j}^{1})\big)dx\\
& \ \ \ \ \ \ \ \ \ \ \ \ +\int\partial_{x}^{\ell}\mathcal{R}_{j}^{1}\cdot
\partial_{x}^{\ell}\big(G_{j,-j}^{+}(\partial_{t}\psi_{c}-\Omega\psi_{c})\mathcal{R}_{-j}^{1}
+G_{j,-j}^{-}(\partial_{t}\phi_{c}+\Omega\phi_{c})\mathcal{R}_{-j}^{1}\big)dx\\
& \ \ \ \ \ \ \ \ \ \ \ \ +\int\partial_{x}^{\ell}\mathcal{R}_{j}^{1}\cdot
\partial_{x}^{\ell}\big(Q_{j,-j}^{+}(\partial_{t}\psi_{c}-\Omega\psi_{c},\mathcal{R}_{-j}^{1})
+Q_{j,-j}^{-}(\partial_{t}\phi_{c}+\Omega\phi_{c},\mathcal{R}_{-j}^{1})\big)dx\bigg\}\\
=&\epsilon\sum_{j\in\{\pm1\}}\bigg\{-\frac{1}{2}\int\partial_{x}^{\ell}\mathcal{R}_{j}^{1}\cdot
\partial_{x}\big(G_{j,j}^{+}(\partial_{t}\psi_{c}-\Omega\psi_{c})+G_{j,j}^{-}(\partial_{t}\phi_{c}+\Omega\phi_{c})\big)
\cdot\partial_{x}^{\ell}\mathcal{R}_{j}^{1}dx\\
& \ \ \ \ \ \ \ \ \ \ \ \ +\int\partial_{x}^{\ell}\mathcal{R}_{j}^{1}\cdot
\big[\partial_{x}^{\ell+1},G_{j,j}^{+}(\partial_{t}\psi_{c}-\Omega\psi_{c})
+G_{j,j}^{-}(\partial_{t}\phi_{c}+\Omega\phi_{c})\big]
\mathcal{R}_{j}^{1}dx\\
& \ \ \ \ \ \ \ \ \ \ \ \ +\int\partial_{x}^{\ell}\mathcal{R}_{j}^{1}\cdot
\partial_{x}^{\ell}\big(Q_{j,j}^{+}(\partial_{t}\psi_{c}-\Omega\psi_{c},\mathcal{R}_{j}^{1})
+Q_{j,j}^{-}(\partial_{t}\phi_{c}+\Omega\phi_{c},\mathcal{R}_{j}^{1})\big)dx\\
& \ \ \ \ \ \ \ \ \ \ \ \ +\int\partial_{x}^{\ell}\mathcal{R}_{j}^{1}\cdot
\big(G_{j,-j}^{+}(\partial_{t}\psi_{c}-\Omega\psi_{c})
+G_{j,-j}^{-}(\partial_{t}\phi_{c}+\Omega\phi_{c})\big)\cdot\partial_{x}^{\ell}\mathcal{R}_{-j}^{1}dx\\
& \ \ \ \ \ \ \ \ \ \ \ \ +\int\partial_{x}^{\ell}\mathcal{R}_{j}^{1}\cdot
\left[\partial_{x}^{\ell},G_{j,-j}^{+}(\partial_{t}\psi_{c}-\Omega\psi_{c})
+G_{j,-j}^{-}(\partial_{t}\phi_{c}+\Omega\phi_{c})\right]\mathcal{R}_{-j}^{1}dx\\
& \ \ \ \ \ \ \ \ \ \ \ \ +\int\partial_{x}^{\ell}\mathcal{R}_{j}^{1}\cdot
\partial_{x}^{\ell}\big(Q_{j,-j}^{+}(\partial_{t}\psi_{c}-\Omega\psi_{c},\mathcal{R}_{-j}^{1})
+Q_{j,-j}^{-}(\partial_{t}\phi_{c}+\Omega\phi_{c},\mathcal{R}_{-j}^{1})\big)dx\bigg\}\\
&\leq C \epsilon\sum_{j\in\{\pm1\}}\bigg\{\|\partial_{x}^{\ell}\mathcal{R}_{j}^{1}\|_{L^{2}}
\|\partial_{x}\big(G_{j,j}^{+}(\partial_{t}\psi_{c}-\Omega\psi_{c})
+G_{j,j}^{-}(\partial_{t}\phi_{c}+\Omega\phi_{c})\big)\partial_{x}^{\ell}\mathcal{R}_{j}^{1}\|_{L^{2}}\\
& \ \ \ \ \ \ \ \ \ \ \ \ \ \ \ \ \  +\|\partial_{x}^{\ell}\mathcal{R}_{j}^{1}\|_{L^{2}}
\|\partial_{x}^{\ell}\big(Q_{j,j}^{+}(\partial_{t}\psi_{c}-\Omega\psi_{c},\mathcal{R}_{j}^{1})
+Q_{j,j}^{-}(\partial_{t}\phi_{c}+\Omega\phi_{c},\mathcal{R}_{j}^{1})\big)\|_{L^{2}}\\
& \ \ \ \ \ \ \ \ \ \ \ \ \ \ \ \ \   +\|\partial_{x}^{\ell}\mathcal{R}_{j}^{1}\|_{L^{2}}
\|\big(G_{j,-j}^{+}(\partial_{t}\psi_{c}-\Omega\psi_{c})
+G_{j,-j}^{-}(\partial_{t}\phi_{c}+\Omega\phi_{c})\big)
\partial_{x}^{\ell}\mathcal{R}_{-j}^{1}\|_{L^{2}}\\
& \ \ \ \ \ \ \ \ \ \ \ \ \ \ \ \ \  +\|\partial_{x}^{\ell}\mathcal{R}_{j}^{1}\|_{L^{2}}
\|\partial_{x}\big(G_{j,-j}^{+}(\partial_{t}\psi_{c}-\Omega\psi_{c})
+G_{j,-j}^{-}(\partial_{t}\phi_{c}+\Omega\phi_{c})\big)
\partial_{x}^{\ell-1}\mathcal{R}_{-j}^{1}\|_{L^{2}}\\
& \ \ \ \ \ \ \ \ \ \ \ \ \ \ \ \ \  +\|\partial_{x}^{\ell}\mathcal{R}_{j}^{1}\|_{L^{2}}
\|\partial_{x}^{\ell}\big(Q_{j,-j}^{+}(\partial_{t}\psi_{c}-\Omega\psi_{c},\mathcal{R}_{-j}^{1})
+Q_{j,-j}^{-}(\partial_{t}\phi_{c}+\Omega\phi_{c},\mathcal{R}_{-j}^{1})\big)\|_{L^{2}}
\bigg\}\\
&\leq C \epsilon\sum_{j\in\{\pm1\}}\bigg\{\|\partial_{x}^{\ell}\mathcal{R}_{j}^{1}\|_{L^{2}}
\|\partial_{x}\big(G_{j,j}^{+}(\partial_{t}\psi_{c}-\Omega\psi_{c})
+G_{j,j}^{-}(\partial_{t}\phi_{c}+\Omega\phi_{c})\big)\|_{L^{\infty}}
\|\partial_{x}^{\ell}\mathcal{R}_{j}^{1}\|_{L^{2}}\\
& \ \ \ \ \ \ \ \ \ \ \ \ \ \ \ \ \  +\|\partial_{x}^{\ell}\mathcal{R}_{j}^{1}\|_{L^{2}}
\|\big(Q_{j,j}^{+}(\partial_{t}\psi_{c}-\Omega\psi_{c},\mathcal{R}_{j}^{1})
+Q_{j,j}^{-}(\partial_{t}\phi_{c}+\Omega\phi_{c},\mathcal{R}_{j}^{1})\big)\|_{H^{s}}\\
& \ \ \ \ \ \ \ \ \ \ \ \ \ \ \ \ \  +\|\partial_{x}^{\ell}\mathcal{R}_{j}^{1}\|_{L^{2}}
\|\big(G_{j,-j}^{+}(\partial_{t}\psi_{c}-\Omega\psi_{c})
+G_{j,-j}^{-}(\partial_{t}\phi_{c}+\Omega\phi_{c})\big)\|_{L^{\infty}}
\|\partial_{x}^{\ell}\mathcal{R}_{-j}^{1}\|_{L^{2}}\\
& \ \ \ \ \ \ \ \ \ \ \ \ \ \ \ \ \  +\|\partial_{x}^{\ell}\mathcal{R}_{j}^{1}\|_{L^{2}}
\|\partial_{x}\big(G_{j,-j}^{+}(\partial_{t}\psi_{c}-\Omega\psi_{c})
+G_{j,-j}^{-}(\partial_{t}\phi_{c}+\Omega\phi_{c})\big)\|_{L^{\infty}}
\|\partial_{x}^{\ell-1}\mathcal{R}_{-j}^{1}\|_{L^{2}}\\
& \ \ \ \ \ \ \ \ \ \ \ \ \ \ \ \ \   +\|\partial_{x}^{\ell}\mathcal{R}_{j}^{1}\|_{L^{2}}
\|\big(Q_{j,-j}^{+}(\partial_{t}\psi_{c}-\Omega\psi_{c},\mathcal{R}_{-j}^{1})
+Q_{j,-j}^{-}(\partial_{t}\phi_{c}+\Omega\phi_{c},\mathcal{R}_{-j}^{1})\big)\|_{H^{s}}
\bigg\}\\
&\leq C \epsilon\|\mathcal{R}^{1}\|^{2}_{H^{s}}
\|(\partial_{t}\widehat{\psi}_{c}-i\omega\widehat{\psi}_{c})
+(\partial_{t}\widehat{\phi}_{c}+i\omega\widehat{\phi}_{c})\|_{L^{1}(s)}\\
&\leq C\epsilon^{3}\|\mathcal{R}^{1}\|^{2}_{H^{s}},
\end{align*}
where we have used H\"older inequality, Sobolev embedding $H^{1}\hookrightarrow L^{\infty}$, Remark \ref{R5}, commutator estimate \eqref{w11} in Lemma \ref{L10} and \eqref{l1s} in Lemma \ref{LO}.

For the fifth brace of \eqref{oel}, we have
\begin{align*}
\epsilon^{2}&\sum_{j,p\in\{\pm1\}}
\bigg\{\int\partial_{x}^{\ell}\mathcal{F}^{9}_{j}\cdot
\partial_{x}^{\ell}(B_{j,p}^{1,1,+}(\psi_{c},\mathcal{R}_{p}^{1})
    +B_{j,p}^{1,1,-}(\phi_{c},\mathcal{R}_{p}^{1}))dx\\
& \ \ \ \ \ \ \ \ \ \ \ +\int\partial_{x}^{\ell}\mathcal{R}_{j}^{1}\cdot
\partial_{x}^{\ell}(B_{j,p}^{1,1,+}(\psi_{c},\mathcal{F}^{9}_{p})
+B_{j,p}^{1,1,-}(\phi_{c},\mathcal{F}^{9}_{p}))dx\bigg\}\\
=\epsilon^{2}&\sum_{j\in\{\pm1\}}
\bigg\{\int\partial_{x}^{\ell}\mathcal{F}^{9}_{j}\cdot
\partial_{x}^{\ell}(B_{j,j}^{1,1,+}(\psi_{c},\mathcal{R}_{j}^{1})
    +B_{j,j}^{1,1,-}(\phi_{c},\mathcal{R}_{j}^{1}))dx\\
& \ \ \ \ \ \ \ \ \ \ \ \ +\int\partial_{x}^{\ell}\mathcal{R}_{j}^{1}\cdot
\partial_{x}^{\ell}(B_{j,j}^{1,1,+}(\psi_{c},\mathcal{F}^{9}_{j})
+B_{j,j}^{1,1,-}(\phi_{c},\mathcal{F}^{9}_{j}))dx\bigg\}\\
+&\epsilon^{2}\sum_{j\in\{\pm1\}}
\bigg\{\int\partial_{x}^{\ell}\mathcal{F}^{9}_{j}\cdot
\partial_{x}^{\ell}(B_{j,-j}^{1,1,+}(\psi_{c},\mathcal{R}_{-j}^{1})
    +B_{j,-j}^{1,1,-}(\phi_{c},\mathcal{R}_{-j}^{1}))dx\\
& \ \ \ \ \ \ \ \ \ \ \ \ \ \ +\int\partial_{x}^{\ell}\mathcal{R}_{j}^{1}\cdot
\partial_{x}^{\ell}(B_{j,-j}^{1,1,+}(\psi_{c},\mathcal{F}^{9}_{-j})
+B_{j,-j}^{1,1,-}(\phi_{c},\mathcal{F}^{9}_{-j}))dx\bigg\}\\
=&\epsilon^{2}\sum_{j\in\{\pm1\}}
\bigg\{\int\partial_{x}^{\ell}\mathcal{F}^{9}_{j}\cdot
\partial_{x}^{\ell}(S_{j,j}^{+}(\partial_{x}\psi_{c},\mathcal{R}_{j}^{1})
    +S_{j,j}^{-}(\partial_{x}\phi_{c},\mathcal{R}_{j}^{1}))dx\bigg\}\\
+&\epsilon^{2}\sum_{j\in\{\pm1\}}
\bigg\{\int\partial_{x}^{\ell}\mathcal{F}^{9}_{j}\cdot
\partial_{x}^{\ell}((G_{j,-j}^{+}\psi_{c}
    +G_{j,-j}^{-}\phi_{c})\mathcal{R}_{-j}^{1}))dx\\
& \ \ \ \ \ \ \ \ \ \ \ \ \ \ +\int\partial_{x}^{\ell}\mathcal{R}_{j}^{1}\cdot
\partial_{x}^{\ell}((G_{j,-j}^{,+}\psi_{c}
+G_{j,-j}^{-}\phi_{c})\mathcal{F}^{9}_{-j}))dx\bigg\}\\
+&\epsilon^{2}\sum_{j\in\{\pm1\}}
\bigg\{\int\partial_{x}^{\ell}\mathcal{F}^{9}_{j}\cdot
\partial_{x}^{\ell}(Q_{j,-j}^{+}(\psi_{c},\mathcal{R}_{-j}^{1})
    +Q_{j,-j}^{-}(\phi_{c},\mathcal{R}_{-j}^{1}))dx\\
& \ \ \ \ \ \ \ \ \ \ \ \ \ \ +\int\partial_{x}^{\ell}\mathcal{R}_{j}^{1}\cdot
\partial_{x}^{\ell}(Q_{j,-j}^{+}(\psi_{c},\mathcal{F}^{9}_{-j})
+Q_{j,-j}^{,-}(\phi_{c},\mathcal{F}^{9}_{-j}))dx\bigg\}.\\
\end{align*}
Therefore, we obtain
\begin{align}\label{EE}
&\partial_{t}E_{\ell}=\epsilon^{2}\sum_{j,n\in\{\pm1\}}\int\partial_{x}^{\ell}
\mathcal{R}_{j}^{1}\cdot\partial_{x}^{\ell+1}(\mathcal{A}_{j,n}\mathcal{R}_{n}^{1})dx\nonumber\\
&+\epsilon^{2}\sum_{j,n\in\{\pm1\}}
\Big[\int\partial_{x}^{\ell+1}(\mathcal{D}_{j,n}\mathcal{R}_{n}^{1})\cdot
\partial_{x}^{\ell}(B_{j,j}^{1,1,+}(\psi_{c},\mathcal{R}_{j}^{1})
+B_{j,j}^{1,1,-}(\phi_{c},\mathcal{R}_{j}^{1}))dx\nonumber\\
& \ \ \ \ \ \ \ \ \  \ \ \ \ \ \ \ \ \
+\int\partial_{x}^{\ell}\mathcal{R}_{j}^{1}\cdot
\partial_{x}^{\ell}(B_{j,j}^{1,1,+}(\psi_{c},\partial_{x}(\mathcal{D}_{j,n}\mathcal{R}_{n}^{1}))
 +B_{j,j}^{1,1,-}(\phi_{c},\partial_{x}(\mathcal{D}_{j,n}\mathcal{R}_{n}^{1})))dx\Big]\nonumber\\
&+\epsilon^{2}\sum_{j,n\in\{\pm1\}}
\Big[\int\partial_{x}^{\ell+1}(\mathcal{D}_{j,n}\mathcal{R}_{n}^{1})\cdot
\partial_{x}^{\ell}(B_{j,-j}^{1,1,+}(\psi_{c},\mathcal{R}_{-j}^{1})
+B_{j,-j}^{1,1,-}(\phi_{c},\mathcal{R}_{-j}^{1}))dx\nonumber\\
& \ \ \ \ \ \ \ \ \  \ \ \ \ \ \ \ \ \ +\int\partial_{x}^{\ell}\mathcal{R}_{j}^{1}\cdot
\partial_{x}^{\ell}(B_{j,-j}^{1,1,+}(\psi_{c},\partial_{x}(\mathcal{D}_{-j,n}\mathcal{R}_{n}^{1}))
+B_{j,-j}^{1,1,-}(\phi_{c},\partial_{x}(\mathcal{D}_{-j,n}\mathcal{R}_{n}^{1})))dx\Big]\nonumber\\
&+\epsilon^{2}\mathcal{O}(1+\mathcal{E}_{s}+\epsilon^{3/2}\mathcal{E}_{s}^{3/2}+\epsilon^{3}\mathcal{E}_{s}^{2})\nonumber\\
&=:\mathcal{P}
+\mathcal{Q}+\mathcal{K}+\epsilon^{2}\mathcal{O}(1+\mathcal{E}_{s}+\epsilon^{3/2}\mathcal{E}_{s}^{3/2}
+\epsilon^{3}\mathcal{E}_{s}^{2}).
\end{align}

\emph{Analysis of $\mathcal{Q}$.} We apply Leibniz's rule and integration by parts to extract all terms with more than $\ell$ spatial derivatives falling on $\mathcal{R}_{\pm1}^{1}$, to obtain
\begin{align*}
\mathcal{Q}=&\epsilon^{2}\sum_{j,n\in\{\pm1\}}
\bigg[\int\partial_{x}^{\ell+1}(\mathcal{D}_{j,n}\mathcal{R}_{n}^{1})\cdot
\partial_{x}^{\ell}\big(B_{j,j}^{1,1,+}(\psi_{c},\mathcal{R}_{j}^{1})
+B_{j,j}^{1,1,-}(\phi_{c},\mathcal{R}_{j}^{1})\big)dx\\
& \ \ \ \ \ \ \ \ \ \ \ \ \ \ +\int\partial_{x}^{\ell}\mathcal{R}_{j}^{1}\cdot
\partial_{x}^{\ell}\big(B_{j,j}^{1,1,+}(\psi_{c},
\partial_{x}(\mathcal{D}_{j,n}\mathcal{R}_{n}^{1}))
 +B_{j,j}^{1,1,-}(\phi_{c},\partial_{x}(\mathcal{D}_{j,n}\mathcal{R}_{n}^{1}))\big)dx\bigg]\\
=&\epsilon^{2}\sum_{j,n\in\{\pm1\}}
\bigg[\int\partial_{x}^{\ell+1}(\mathcal{D}_{j,n}\mathcal{R}_{n}^{1})\cdot
\big(B_{j,j}^{1,1,+}(\psi_{c},\partial_{x}^{\ell}\mathcal{R}_{j}^{1})
+B_{j,j}^{1,1,-}(\phi_{c},\partial_{x}^{\ell}\mathcal{R}_{j}^{1})\big)dx\\
& \ \ \ \ \ \ \ \ \ \ \ \ \ \  +\ell\int\partial_{x}^{\ell+1}(\mathcal{D}_{j,n}\mathcal{R}_{n}^{1})\cdot
\big(B_{j,j}^{1,1,+}(\partial_{x}\psi_{c},\partial_{x}^{\ell-1}\mathcal{R}_{j}^{1})
+B_{j,j}^{1,1,-}(\partial_{x}\phi_{c},\partial_{x}^{\ell-1}\mathcal{R}_{j}^{1})\big)dx\\
& \ \ \ \ \ \ \ \ \ \ \ \ \ \ +\int\partial_{x}^{\ell}\mathcal{R}_{j}^{1}\cdot
\big(B_{j,j}^{1,1,+}(\psi_{c},
\partial_{x}^{\ell+1}(\mathcal{D}_{j,n}\mathcal{R}_{n}^{1}))
+B_{j,j}^{1,1,-}(\phi_{c},\partial_{x}^{\ell+1}(\mathcal{D}_{j,n}\mathcal{R}_{n}^{1}))\big)dx\\
& \ \ \ \ \ \ \ \ \ \ \ \ \ \ +\ell\int\partial_{x}^{\ell}\mathcal{R}_{j}^{1}\cdot
\big(B_{j,j}^{1,1,+}(\partial_{x}\psi_{c},
\partial_{x}^{\ell}(\mathcal{D}_{j,n}\mathcal{R}_{n}^{1}))
+B_{j,j}^{1,1,-}(\partial_{x}\phi_{c},\partial_{x}^{\ell}(\mathcal{D}_{j,n}\mathcal{R}_{n}^{1}))
\big)dx\bigg]\\
&+\epsilon^{2}\mathcal{O}(1+\mathcal{E}_{s}+\epsilon^{3/2}\mathcal{E}_{s}^{3/2}
+\epsilon^{3}\mathcal{E}_{s}^{2})\\
=&\epsilon^{2}\sum_{j,n\in\{\pm1\}}
\bigg[\int\partial_{x}^{\ell+1}(\mathcal{D}_{j,n}\mathcal{R}_{n}^{1})\cdot
\big(B_{j,j}^{1,1,+}(\psi_{c},\partial_{x}^{\ell}\mathcal{R}_{j}^{1})
+B_{j,j}^{1,1,-}(\phi_{c},\partial_{x}^{\ell}\mathcal{R}_{j}^{1})\big)dx\\
& \ \ \ \ \ \ \ \ \ \ \ \ \ \ +\int\partial_{x}^{\ell}\mathcal{R}_{j}^{1}\cdot
\big(B_{j,j}^{1,1,+}(\psi_{c},
\partial_{x}^{\ell+1}(\mathcal{D}_{j,n}\mathcal{R}_{n}^{1}))
+B_{j,j}^{1,1,-}(\phi_{c},\partial_{x}^{\ell+1}(\mathcal{D}_{j,n}\mathcal{R}_{n}^{1})\big)dx\\
& \ \ \ \ \ \ \ \ \ \ \ \ \ \ -\ell\int\partial_{x}^{\ell}(\mathcal{D}_{j,n}\mathcal{R}_{n}^{1})\cdot
\big(B_{j,j}^{1,1,+}(\partial_{x}\psi_{c},\partial_{x}^{\ell}\mathcal{R}_{j}^{1})
+B_{j,j}^{1,1,-}(\partial_{x}\phi_{c},\partial_{x}^{\ell-1}\mathcal{R}_{j}^{1})\big)dx\\
& \ \ \ \ \ \ \ \ \ \ \ \ \ \ +\ell\int\partial_{x}^{\ell}\mathcal{R}_{j}^{1}\cdot
\big(B_{j,j}^{1,1,+}(\partial_{x}\psi_{c},
\partial_{x}^{\ell}(\mathcal{D}_{j,n}\mathcal{R}_{n}^{1}))
+B_{j,j}^{1,1,-}(\partial_{x}\phi_{c},\partial_{x}^{\ell}(\mathcal{D}_{j,n}\mathcal{R}_{n}^{1}))\big)dx\bigg]\\
&+\epsilon^{2}\mathcal{O}(1+\mathcal{E}_{s}+\epsilon^{3/2}\mathcal{E}_{s}^{3/2}
+\epsilon^{3}\mathcal{E}_{s}^{2}).
\end{align*}
According to \eqref{17} and the asymptotic expansion \eqref{15} and \eqref{18}, we have
\begin{align*}
\mathcal{Q}=&\epsilon^{2}\sum_{j,n\in\{\pm1\}}
\bigg[\int\partial_{x}^{\ell}\mathcal{R}_{j}^{1}\cdot
\big(S_{j,j}^{+}(\partial_{x}\psi_{c},\partial_{x}^{\ell+1}(\mathcal{D}_{j,n}\mathcal{R}_{n}^{1}))
+S_{j,j}^{-}(\partial_{x}\phi_{c},\partial_{x}^{\ell+1}(\mathcal{D}_{j,n}\mathcal{R}_{n}^{1}))\big)dx\\
& \ \ \ \ \ \ \ \ \ \ \ \ \ +2\ell\int\partial_{x}^{\ell}\mathcal{R}_{j}^{1}\cdot
\big(B_{j,j}^{1,1,+}(\partial_{x}\psi_{c},
\partial_{x}^{\ell}(\mathcal{D}_{j,n}\mathcal{R}_{n}^{1}))
+B_{j,j}^{1,1,-}(\partial_{x}\phi_{c},\partial_{x}^{\ell}(\mathcal{D}_{j,n}\mathcal{R}_{n}^{1}))\big)dx\bigg]\\
&+\epsilon^{2}\mathcal{O}(1+\mathcal{E}_{s}+\epsilon^{3/2}\mathcal{E}_{s}^{3/2}
+\epsilon^{3}\mathcal{E}_{s}^{2})\\
=&\epsilon^{2}\sum_{j,n\in\{\pm1\}}
\bigg[\int\partial_{x}^{\ell}\mathcal{R}_{j}^{1}\cdot
\big(G_{j,j}^{+}\partial_{x}\psi_{c}+G_{j,j}^{-}\partial_{x}\phi_{c}\big)\cdot \partial_{x}^{\ell+1}(\mathcal{D}_{j,n}\mathcal{R}_{n}^{1})dx\\
& \ \ \ \ \ \ \ \ \ \ \ \ \ +2\ell\int\partial_{x}^{\ell}\mathcal{R}_{j}^{1}\cdot
\partial_{x}\big((G_{j,j}^{+}\partial_{x}\psi_{c}+G_{j,j}^{-}\partial_{x}\phi_{c})  \partial_{x}^{\ell}(\mathcal{D}_{j,n}\mathcal{R}_{n}^{1})\big)dx\bigg]\\
&+\epsilon^{2}\mathcal{O}(1+\mathcal{E}_{s}+\epsilon^{3/2}\mathcal{E}_{s}^{3/2}
+\epsilon^{3}\mathcal{E}_{s}^{2})\\
=&(2\ell+1)\epsilon^{2}\sum_{j,n\in\{\pm1\}} \int\partial_{x}^{\ell}\mathcal{R}_{j}^{1}\cdot
\big(G_{j,j}^{+}\partial_{x}\psi_{c}+G_{j,j}^{-}\partial_{x}\phi_{c}\big)\cdot  \partial_{x}^{\ell+1}(\mathcal{D}_{j,n}\mathcal{R}_{n}^{1})dx\\
&+\epsilon^{2}\mathcal{O}(1+\mathcal{E}_{s}+\epsilon^{3/2}\mathcal{E}_{s}^{3/2}
+\epsilon^{3}\mathcal{E}_{s}^{2})\\
=&(2\ell+1)\epsilon^{2}\sum_{j\in\{\pm1\}}
\int\partial_{x}^{\ell}\mathcal{R}_{j}^{1}\cdot
(G_{j,j}^{+}\partial_{x}\psi_{c}+G_{j,j}^{-}\partial_{x}\phi_{c})\cdot \partial_{x}^{\ell+1}\big(q\varphi_{s_{6}}\mathcal{R}_{j}^{1}
-\frac{j}{q}\varphi_{s_{5}}\mathcal{R}_{-j}^{1})\big)dx\\
&+\epsilon^{2}\mathcal{O}(1+\mathcal{E}_{s}+\epsilon^{3/2}\mathcal{E}_{s}^{3/2}
+\epsilon^{3}\mathcal{E}_{s}^{2})\\
=&-j(2\ell+1)\epsilon^{2}\sum_{j\in\{\pm1\}}
\int(G_{j,j}^{+}\partial_{x}\psi_{c}+G_{j,j}^{-}\partial_{x}\phi_{c})\frac{1}{q}\varphi_{s_{5}}\cdot
\partial_{x}^{\ell}\mathcal{R}_{j}^{1}\cdot\partial_{x}^{\ell+1}\mathcal{R}_{-j}^{1}dx\\
&+\epsilon^{2}\mathcal{O}(1+\mathcal{E}_{s}+\epsilon^{3/2}\mathcal{E}_{s}^{3/2}
+\epsilon^{3}\mathcal{E}_{s}^{2})\\
=&-(\ell+\frac{1}{2})\epsilon^{2}
\int(G_{1,1}^{+}+G_{-1,-1}^{+})(\partial_{x}\psi_{c}+
\partial_{x}\phi_{c})\frac{1}{q}\varphi_{s_{5}}\cdot
\partial_{x}^{\ell}(\mathcal{R}_{1}^{1}+\mathcal{R}_{-1}^{1})\cdot
\partial_{x}^{\ell+1}(\mathcal{R}_{1}^{1}-\mathcal{R}_{-1}^{1})dx\\
&+\epsilon^{2}\mathcal{O}(1+\mathcal{E}_{s}+\epsilon^{3/2}\mathcal{E}_{s}^{3/2}
+\epsilon^{3}\mathcal{E}_{s}^{2})\\
=&-(2\ell+1)\epsilon^{2}\widehat{q}(k_{0})
\int(-1+\partial_{x}^{2})q(\psi_{c}+\phi_{c})\frac{1}{q}\varphi_{s_{5}}\cdot
\partial_{x}^{\ell}(\mathcal{R}_{1}^{1}+\mathcal{R}_{-1}^{1})\cdot
\partial_{x}^{\ell+1}(\mathcal{R}_{1}^{1}-\mathcal{R}_{-1}^{1})dx\\
&+\epsilon^{2}\mathcal{O}(1+\mathcal{E}_{s}+\epsilon^{3/2}\mathcal{E}_{s}^{3/2}
+\epsilon^{3}\mathcal{E}_{s}^{2}),
\end{align*}
where we have used \eqref{finally} with \eqref{ABn}, \eqref{part2}-\eqref{part4} and integration by parts.

On the other hand, by the equation \eqref{finally}-\eqref{ABn}, we have
\begin{equation}
\begin{split}\label{equ}
\partial_{t}(R_{1}^{1}+R_{-1}^{1})=&\partial_{x}q(\mathcal{R}_{1}^{1}-\mathcal{R}_{-1}^{1})+ \epsilon\partial_{x}(\varphi_{s_{5}}q(\mathcal{R}_{1}^{1}-\mathcal{R}_{-1}^{1})
+q\varphi_{s_{6}}(\mathcal{R}_{1}^{1}+\mathcal{R}_{-1}^{1}))
+\epsilon(\mathcal{F}_{1}^{9}+\mathcal{F}_{-1}^{9}).
\end{split}
\end{equation}
Taking $\partial_{x}^{\ell}$ on the equation \eqref{equ} and using commutator notation \eqref{w}, we have
\begin{align}\label{equu}
\partial_{x}^{\ell+1}&(\mathcal{R}_{1}^{1}-\mathcal{R}_{-1}^{1})
=\frac{1}{1+\epsilon\varphi_{s_{5}}}\Big[
\partial_{t}\partial_{x}^{\ell}(\mathcal{R}_{1}^{1}+\mathcal{R}_{-1}^{1})-\epsilon q \varphi_{s_{6}}\partial_{x}^{\ell+1}(\mathcal{R}_{1}^{1}+\mathcal{R}_{-1}^{1})\nonumber\\
&-\epsilon[\partial_{x}^{\ell+1},\varphi_{s_{5}}]q(\mathcal{R}_{1}^{1}-\mathcal{R}_{-1}^{1})
-\epsilon[\partial_{x}^{\ell+1},q\varphi_{s_{6}}](\mathcal{R}_{1}^{1}+\mathcal{R}_{-1}^{1})
-\epsilon\partial_{x}^{\ell}(\mathcal{F}_{1}^{9}+\mathcal{F}_{-1}^{9})\Big]\\
&+(1-q)\partial_{x}^{\ell+1}(\mathcal{R}_{1}^{1}-\mathcal{R}_{-1}^{1})\nonumber.
\end{align}
Then by \eqref{equu}, \eqref{part5}, integration by parts and commutator estimates \eqref{w11}, we have
\begin{align}\label{QQ}
\mathcal{Q}=&-(2\ell+1)\epsilon^{2}\widehat{q}(k_{0})
\int\frac{(-1+\partial_{x}^{2})q(\psi_{c}+\phi_{c})\frac{1}{q}\varphi_{s_{5}}}{1+\epsilon\varphi_{s_{5}}}\cdot
\partial_{x}^{\ell}(\mathcal{R}_{1}^{1}+\mathcal{R}_{-1}^{1})\cdot
\partial_{t}\partial_{x}^{\ell}(\mathcal{R}_{1}^{1}+\mathcal{R}_{-1}^{1})dx\nonumber\\
&+(2\ell+1)\epsilon^{3}\widehat{q}(k_{0})
\int\frac{(-1+\partial_{x}^{2})q(\psi_{c}+\phi_{c})\frac{1}{q}\varphi_{s_{5}}q\varphi_{s_{6}}}
{1+\epsilon\varphi_{s_{5}}}\cdot
\partial_{x}^{\ell}(\mathcal{R}_{1}^{1}+\mathcal{R}_{-1}^{1})\cdot
\partial_{x}^{\ell+1}(\mathcal{R}_{1}^{1}+\mathcal{R}_{-1}^{1})dx\nonumber\\
&+\epsilon^{2}(1+\mathcal{E}_{s}+\epsilon^{3/2}\mathcal{E}_{s}^{3/2}
+\epsilon^{3}\mathcal{E}_{s}^{2})\nonumber\\
=&-(\ell+\frac{1}{2})\epsilon^{2}\widehat{q}(k_{0})\frac{d}{dt}
\int\frac{(-1+\partial_{x}^{2})q(\psi_{c}+\phi_{c})\frac{1}{q}\varphi_{s_{5}}}{1+\epsilon\varphi_{s_{5}}}\cdot
(\partial_{x}^{\ell}(\mathcal{R}_{1}^{1}+\mathcal{R}_{-1}^{1}))^{2}dx\\
&+\epsilon^{2}\mathcal{O}(1+\mathcal{E}_{s}+\epsilon^{3/2}\mathcal{E}_{s}^{3/2}
+\epsilon^{3}\mathcal{E}_{s}^{2})\nonumber.
\end{align}

\emph{Analysis of $\mathcal{K}$.} According to \eqref{15}, \eqref{part2}, \eqref{part5}, \eqref{part6}, integration by parts and Lemma \ref{L9}, we have
\begin{align}\label{KK}
\mathcal{K}=&\epsilon^{2}\sum_{j,n\in\{\pm1\}}
\Big[\int\partial_{x}^{\ell+1}(\mathcal{D}_{j,n}\mathcal{R}_{n}^{1})\cdot
(B_{j,-j}^{1,1,+}(\psi_{c},\partial_{x}^{\ell}\mathcal{R}_{-j}^{1})
+B_{j,-j}^{1,1,-}(\phi_{c},\partial_{x}^{\ell}\mathcal{R}_{-j}^{1}))dx\nonumber\\
& \ \ \ \ \ \ \ \ \  \ \ \ \ \ \ +\int\partial_{x}^{\ell}\mathcal{R}_{j}^{1}\cdot
(B_{j,-j}^{1,1,+}(\psi_{c},\partial_{x}^{\ell+1}(\mathcal{D}_{-j,n}\mathcal{R}_{n}^{1}))
+B_{j,-j}^{1,1,-}(\phi_{c},\partial_{x}^{\ell+1}(\mathcal{D}_{-j,n}\mathcal{R}_{n}^{1})))dx\Big]\nonumber\\
&+\epsilon^{2}\mathcal{O}(\mathcal{E}_{s}+\epsilon^{3/2}\mathcal{E}_{s}^{3/2})\nonumber\\
=&\epsilon^{2}\sum_{j\in\{\pm1\}}\Big[\int\partial_{x}^{\ell+1}(\mathcal{D}_{j,j}\mathcal{R}_{j}^{1})\cdot
(B_{j,-j}^{1,1,+}(\psi_{c},\partial_{x}^{\ell}\mathcal{R}_{-j}^{1})
+B_{j,-j}^{1,1,-}(\phi_{c},\partial_{x}^{\ell}\mathcal{R}_{-j}^{1}))dx\nonumber\\
& \ \ \ \ \ \ \ \ \  \ \ \ \  +\int\partial_{x}^{\ell}\mathcal{R}_{j}^{1}\cdot
(B_{j,-j}^{1,1,+}(\psi_{c},\partial_{x}^{\ell+1}(\mathcal{D}_{-j,-j}\mathcal{R}_{-j}^{1}))
+B_{j,-j}^{1,1,-}(\phi_{c},\partial_{x}^{\ell+1}(\mathcal{D}_{-j,-j}\mathcal{R}_{-j}^{1})))dx\Big]\nonumber\\
&+\epsilon^{2}\mathcal{O}(\mathcal{E}_{s}+\epsilon^{3/2}\mathcal{E}_{s}^{3/2})\nonumber\\
=&2\epsilon^{2}\sum_{j\in\{\pm1\}} \int q\varphi_{s_{6}}(
G_{j,-j}^{+}\psi_{c}+G_{j,-j}^{-}\phi_{c})\cdot \partial_{x}^{\ell}\mathcal{R}_{j}^{1}\cdot\partial_{x}^{\ell+1}\mathcal{R}_{-j}^{1}dx
+\epsilon^{2}\mathcal{O}(\mathcal{E}_{s}+\epsilon^{3/2}\mathcal{E}_{s}^{3/2})\nonumber\\
=&-2\epsilon^{2} \int q\varphi_{s_{6}}(
(G_{1,-1}^{+}-G_{-1,1}^{+})\psi_{c}+(G_{1,-1}^{-}-G_{-1,1}^{-})\phi_{c})\cdot \partial_{x}^{\ell+1}(\mathcal{R}_{1}^{1}+\mathcal{R}_{-1}^{1})
\cdot\partial_{x}^{\ell}(\mathcal{R}_{1}^{1}-\mathcal{R}_{-1}^{1})dx\nonumber\\
&+\epsilon^{2}\mathcal{O}(\mathcal{E}_{s}+\epsilon^{3/2}\mathcal{E}_{s}^{3/2})\nonumber\\
=&\epsilon^{2}\mathcal{O}(\mathcal{E}_{s}+\epsilon^{3/2}\mathcal{E}_{s}^{3/2}).
\end{align}

\emph{Analysis of $\mathcal{P}$.} According to \eqref{ABn}, \eqref{part2}-\eqref{part5} and integration by parts, we have
\begin{align*}
\mathcal{P}=&\epsilon^{2}\sum_{j,n\in\{\pm1\}}\int\partial_{x}^{\ell}
\mathcal{R}_{j}^{1}\cdot\partial_{x}^{\ell+1}(\mathcal{A}_{j,n}\mathcal{R}_{n}^{1})dx\\
=&\epsilon^{2}\sum_{j\in\{\pm1\}}\int\partial_{x}^{\ell}
\mathcal{R}_{j}^{1}\cdot\partial_{x}^{\ell+1}(\mathcal{A}_{j,-j}\mathcal{R}_{-j}^{1})dx
+\epsilon^{2}\mathcal{O}(\mathcal{E}_{s}+\epsilon^{3/2}\mathcal{E}_{s}^{3/2})\\
=&j\epsilon^{2}\sum_{j\in\{\pm1\}}\int\varphi_{s_{3}}\cdot\partial_{x}^{\ell}
\mathcal{R}_{j}^{1}\cdot\partial_{x}^{\ell+1}\mathcal{R}_{-j}^{1}dx
+\epsilon^{2}\mathcal{O}(\mathcal{E}_{s}+\epsilon^{3/2}\mathcal{E}_{s}^{3/2})\\
=&-\frac{\epsilon^{2}}{2}\int\varphi_{s_{3}}\cdot\partial_{x}^{\ell}
(\mathcal{R}_{1}^{1}+\mathcal{R}_{-1}^{1})\cdot
\partial_{x}^{\ell+1}(\mathcal{R}_{1}^{1}-\mathcal{R}_{-1}^{1})dx
+\epsilon^{2}\mathcal{O}(\mathcal{E}_{s}+\epsilon^{3/2}\mathcal{E}_{s}^{3/2}).
\end{align*}
Then using \eqref{equu}, \eqref{part5}, integration by parts and commutator estimates \eqref{w11} once more, we have
\begin{equation}
\begin{split}\label{PP}
\mathcal{P}=&-\frac{\epsilon^{2}}{2}\int\varphi_{s_{3}}\cdot\partial_{x}^{\ell}
(\mathcal{R}_{1}^{1}+\mathcal{R}_{-1}^{1})\cdot
\partial_{x}^{\ell+1}(\mathcal{R}_{1}^{1}-\mathcal{R}_{-1}^{1})dx
+\epsilon^{2}\mathcal{O}(\mathcal{E}_{s}+\epsilon^{3/2}\mathcal{E}_{s}^{3/2})\\
=&-\frac{\epsilon^{2}}{2}\int\frac{\varphi_{s_{3}}}{1+\epsilon\varphi_{s_{5}}}
\partial_{x}^{\ell}(\mathcal{R}_{1}^{1}+\mathcal{R}_{-1}^{1})
\partial_{t}\partial_{x}^{\ell}
(\mathcal{R}_{1}^{1}+\mathcal{R}_{-1}^{1})dx\\
&-\frac{\epsilon^{3}}{2}\int\frac{\varphi_{s_{3}}q\varphi_{s_{6}}}{1+\epsilon\varphi_{s_{5}}}
\partial_{x}^{\ell}(\mathcal{R}_{1}^{1}+\mathcal{R}_{-1}^{1})
\partial_{x}^{\ell+1}
(\mathcal{R}_{1}^{1}+\mathcal{R}_{-1}^{1})dx
+\epsilon^{2}\mathcal{O}(\mathcal{E}_{s}+\epsilon^{3/2}\mathcal{E}_{s}^{3/2})\\
=&-\epsilon^{2}\frac{d}{dt}\int\frac{\varphi_{s_{3}}}{1+\epsilon\varphi_{s_{5}}}
(\partial_{x}^{\ell}(\mathcal{R}_{1}^{1}+\mathcal{R}_{-1}^{1}))^{2}dx
+\epsilon^{2}\mathcal{O}(\mathcal{E}_{s}+\epsilon^{3/2}\mathcal{E}_{s}^{3/2}).
\end{split}
\end{equation}

Combing \eqref{EE} and \eqref{QQ}-\eqref{PP}, we further define a modified energy
\begin{align*}
\widetilde{\mathcal{E}_{s}}=\mathcal{E}_{s}+\epsilon^{2}\sum_{\ell=1}^{s}h_{\ell},
\end{align*}
with
\begin{equation*}
\begin{split}
h_{\ell}=
\int\frac{((\ell+\frac{1}{2})\widehat{q}(k_{0})(-1+\partial_{x}^{2})q(\psi_{c}+\phi_{c})
\frac{1}{q}\varphi_{s_{5}}+\varphi_{s_{3}}}{1+\epsilon\varphi_{s_{5}}}\cdot
(\partial_{x}^{\ell}(\mathcal{R}_{1}^{1}+\mathcal{R}_{-1}^{1}))^{2}dx.
\end{split}
\end{equation*}
According to the form of $\varphi_{s_{3}}$ and $\varphi_{s_{5}}$ in \eqref{s34} and \eqref{s56}, we see that $\sum_{\ell=1}^{s}h_{\ell}=\mathcal{O}(\|\mathcal{R}\|^{2}_{H^{s}})$ as long as $\|\mathcal{R}\|^{2}_{H^{s}}=\mathcal{O}(1)$, and $\epsilon^{2}\partial_{t}\sum_{\ell=1}^{s}h_{\ell}$ eliminates all the integral terms on the right-hand side of the evolution equation of $\mathcal{E}_{s}$ with a factor $\partial_{x}^{s+1}\mathcal{R}$. Consequently, we obtain
\begin{equation*}
\begin{split}
\partial_{t}\widetilde{\mathcal{E}_{s}}\lesssim\epsilon^{2}(\widetilde{\mathcal{E}_{s}}
+\epsilon^{1/2}\widetilde{\mathcal{E}_{s}}^{3/2}+\epsilon\widetilde{\mathcal{E}_{s}}^{2}+1),
\end{split}
\end{equation*}
as long as $\|\mathcal{R}\|^{2}_{H^{s}}=\mathcal{O}(1)$.
So Gronwall's inequality yields the $\mathcal{O}(1)$ boundedness of $\widetilde{\mathcal{E}_{s}}$ and hence the $\mathcal{O}(1)$ boundedness of $R$ for all $t\in[0,T_{0}/\epsilon^{2}]$, thanks to the equivalence of $\|R\|_{H^{s}}$ and $\sqrt{\mathcal{E}_{s}}$ in Proposition \ref{C1}. Then we finish the proof of Theorem \ref{Thm2} and hence of Theorem \ref{Thm1}.

\bigskip

Conflict of Interest: The authors declare that they have no conflict of interest.

\end{document}